\documentclass[12pt]{amsart}
\usepackage{enumitem,amssymb,aliascnt,geometry,stmaryrd,mathtools,microtype,xltabular}
\usepackage{tikz-cd}
\usepackage[pdfusetitle]{hyperref}

\hypersetup{ pdfkeywords={Algebraic cobordism, group actions on varieties, fixed-point theorems},
hidelinks}

\usepackage{etoolbox}
\makeatletter
\pretocmd{\section}{\addtocontents{toc}{\protect\addvspace{6\p@}\bfseries}}{}{}
\pretocmd{\subsection}{\addtocontents{toc}{\protect\normalfont}}{}{}
\makeatother

\DeclareMathOperator{\Spec}{Spec}
\DeclareMathOperator{\im}{im}
\DeclareMathOperator{\Tor}{Tor}

\DeclareMathOperator{\coker}{coker}
\DeclareMathOperator{\CH}{CH}
\DeclareMathOperator{\Th}{Th}
\DeclareMathOperator{\MGL}{MGL}
\DeclareMathOperator{\Laz}{\mathbb{L}}
\DeclareMathOperator{\rank}{rank}
\DeclareMathOperator{\Hb}{HB}
\DeclareMathOperator{\Hz}{H\mathbb{Z}}

\DeclareMathOperator{\fgl}{[+]}
\DeclareMathOperator{\length}{\ell}

\newcommand{\eff}{\mathrm{eff}}
\newcommand{\character}[1]{\mathcal{L}_{#1}}
\newcommand{\T}{Q}
\newcommand{\Ipn}[1]{I_p({#1})}
\newcommand{\JK}{J}
\newcommand{\JKe}{J_{\eff}}
\newcommand{\Speck}{k}
\newcommand{\Oc}{\mathcal{O}}
\newcommand{\Pc}{\mathcal{P}}
\newcommand{\pp}{\mathfrak{p}}

\newcommand{\Mcc}{\mathcal{M}}
\newcommand{\Mcd}[1]{\Mcc^{\le {#1}}}
\newcommand{\bb}{\mathbf{b}}
\newcommand{\ag}{\mathbf{a}}

\newcommand{\Fc}{\mathcal{F}}
\newcommand{\Sy}[2]{\mathcal{S}_{{#1}}({#2})}
\newcommand{\lc}{\llbracket}
\newcommand{\rc}{\rrbracket}

\newcommand{\Zz}{\mathbb{Z}}
\newcommand{\Pp}{\mathbb{P}}
\newcommand{\La}{L}
\newcommand{\Nn}{\mathbb{N}}
\newcommand{\Qq}{\mathbb{Q}}
\newcommand{\Gm}{\mathbb{G}_m}
\newcommand{\Fp}{\mathbb{F}_p}
\newcommand{\Tan}{T}

\newcommand{\fglr}[1]{[{#1}]}
\newcommand{\su}[2]{#1_{(#2)}}
\newcommand{\wgt}[2]{{#1}(#2)}
\newcommand{\Apar}[1]{\mathcal{A}_{({#1})}}

\newtheorem{thm}{Theorem}

\newtheorem{cor}[thm]{Corollary}

\renewcommand{\thethm}{\arabic{thm}}

\swapnumbers

\newtheorem{theorem}{Theorem}[section]
\newaliascnt{proposition}{theorem}
\newtheorem{proposition}[proposition]{Proposition}
\newaliascnt{lemma}{theorem}
\newtheorem{lemma}[lemma]{Lemma}
\newaliascnt{corollary}{theorem}
\newtheorem{corollary}[corollary]{Corollary}

\theoremstyle{definition}
\newaliascnt{remark}{theorem}
\newtheorem{remark}[theorem]{Remark}
\newaliascnt{example}{theorem}
\newtheorem{example}[example]{Example}
\newaliascnt{definition}{theorem}
\newtheorem{definition}[definition]{Definition}

\newtheoremstyle{par}
{}
{}
{}
{}
{}
{.}
{ }
{}%
\theoremstyle{par}
\newtheorem{para}[theorem]{}

\numberwithin{equation}{theorem}

\begin{document}
\begin{abstract}
	We determine all restrictions on the dimension of the fixed locus of a diagonalizable group acting on a smooth projective variety that arise from the Chern numbers of the ambient variety.
	We reduce the problem to finding lower bounds for actions of \(p\)-groups, which we achieve by analyzing the equivariant cobordism ring with the help of the concentration theorem.
	To do so, we construct enough explicit examples of actions that realize the expected lower bound.
	We then prove that this family is maximal in the equivariant cobordism ring, in an appropriate sense.
\end{abstract}

\author{Olivier Haution}
\title{Dimension of fixed loci of diagonalizable groups via algebraic cobordism}
\email{olivier.haution at unimib.it}
\address{Dipartimento di Matematica e Applicazioni, Università degli Studi di Milano-Bicocca, via Roberto Cozzi 55, 20125 Milano, Italy}


\date{\today}

\maketitle

\setcounter{tocdepth}{2}
\tableofcontents

\numberwithin{theorem}{subsection}
\numberwithin{lemma}{subsection}
\numberwithin{proposition}{subsection}
\numberwithin{corollary}{subsection}
\numberwithin{example}{subsection}
\numberwithin{definition}{subsection}
\numberwithin{remark}{subsection}

\section{Introduction}
\renewcommand{\theequation}{\thethm.{\alph{equation}}}

\subsection{Existence of fixed points}
\label{sect:fixed_points}
When does a group action on an algebraic variety admit fixed points?
This fundamental question is typically answered by means of \emph{fixed-point theorems}, which provide sufficient conditions for the existence of fixed points, based on nonequivariant information on the ambient variety.

In this paper, the information considered will be the collection of the Chern numbers of the variety, which are certain numerical invariants indexed by the partitions of its dimension.
These are constructed using intersection theory, and roughly speaking quantify the geometric complexity of the variety.
Let us fix a base field, assumed to be algebraically closed of characteristic zero for this introduction.
Identifying smooth projective varieties having the same Chern numbers yields a ring \(\Laz\) called the Lazard ring.
Thus a smooth projective variety \(X\) has a cobordism class \(\lc X \rc \in \Laz\), which exactly encodes its Chern numbers.

A complete determination of the fixed-point theorems involving Chern numbers is equivalent to the computation of the subgroup \(\JK(G) \subset \Laz\) generated by the cobordism classes of varieties admitting a fixed-point-free action of a given algebraic group \(G\).

In this paper we consider finite type diagonalizable groups \(G\), which means, given our assumptions on the field \(k\), that \(G\) is the product of a split torus and a (constant) finite abelian group.
Borel's fixed-point theorem implies that \(\JK(G \times \Gm)=\JK(G)\), so we are reduced to considering finite abelian groups \(G\).
For such groups, the only case of real interest is when \(G\) is a finite abelian \(p\)-group for some prime \(p\): indeed it turns out that otherwise \(\JK(G) = \Laz\) (see (\ref{prop:J_mu_pq})).

Our first result, proved as (\ref{th:J_Ipn}), identifies the ideal \(\JK(G)\) in this case, providing an algebraic analogue of a theorem of tom Dieck in topology \cite{tomDieck-Actions-finite-abelian}.
By the \emph{rank} of a finite abelian group, we mean its number of cyclic factors.
\begin{thm}
	\label{th:no_fixed}
	Let \(G\) be a finite abelian \(p\)-group of rank \(r\).
	Then \(\JK(G)\) is the subgroup \(\Ipn{r}\) of \(\Laz\) generated by the classes of smooth projective varieties of dimension \(< p^{r-1}\) having all Chern numbers divisible by \(p\).
\end{thm}
The subgroup \(\Ipn{r}\) is the \(r\)-th \emph{Landweber ideal}, which plays a fundamental role in many aspects of cobordism theory.
As an ideal, it is generated by the classes of \(r\) explicit varieties, which turn out to admit fixed-point-free \(G\)-actions.
The first of these is the disjoint union of \(p\) points, and so \(\Ipn{1}=p\Laz\).
When \(r>1\), the ideal \(\Ipn{r}\) can be thought of as a higher version of the ideal \(p\Laz\).
For instance, the ideal \(I_2(2)\) is generated by \(2\) and \(\lc \Pp^1 \rc\),
and a fixed-point-free action of \(\Zz/2 \times \Zz/2\) on \(\Pp^1\) is given by the commuting involutions \(x \mapsto -x\) and \(x\mapsto x^{-1}\).

Concretely, this result allows for the detection of fixed points: any \(G\)-action on a variety whose cobordism class does not lie in \(\Ipn{r}\) must have a fixed point.
The theorem additionally expresses the fact that no other (nonequivariant) cobordism invariant can be used to detect fixed points.

Observe that Theorem~\ref{th:no_fixed} implies that the action of an abelian \(p\)-group on a variety having a Chern number prime to \(p\) must fix a point, a result previously obtained in \cite[(1.1.2)]{fpt}.
Note that the structure of the abelian \(p\)-group \(G\) (including its rank \(r\)) played no role in that particular result.
Theorem~\ref{th:no_fixed} also recovers, for \(r=1\), the fixed-point theorem for actions of cyclic \(p\)-groups involving the arithmetic genus proved in \cite[(1.2.2)]{fpt}.
Theorem~\ref{th:no_fixed} can thus be thought of as interpolating between those two extreme special cases.

Finally, note that Theorem~\ref{th:no_fixed} provides a new definition of the Landweber ideal \(\Ipn{r}\), in terms of fixed-point-free actions, instead of Chern numbers.

\subsection{Fixed locus dimension}
Beyond detecting fixed points, one may attempt to control the dimension of the fixed locus.
To this end, for each integer \(d\), we denote by \(\Delta^d(G)\subset \Laz\) the subgroup of the Lazard ring generated by the classes of varieties admitting a \(G\)-action with a fixed locus of dimension \(d\).
We observe that \(\Delta^d(G)\) contains the classes of all varieties admitting a \(G\)-action with a fixed locus of dimension \(\le d\) (see (\ref{lemm:only_lower_bounds})).
Consequently, the cobordism class of the ambient variety imposes only \emph{lower bounds} on the fixed locus dimension.

We prove in (\ref{prop:Delta_Gm}) that \(\Delta^d(G)=\Laz\) when \(G\) contains the multiplicative group \(\Gm\).
As \(\JK(G) \subset \Delta^d(G)\), the considerations of \S\ref{sect:fixed_points} imply that \(\Delta^d(G)=\Laz\) when \(G\) is not a finite abelian \(p\)-group for some prime \(p\) (as above, for \(G\) diagonalizable).

To complete the picture, let us therefore assume that \(G\) is a finite abelian \(p\)-group.
The next theorem, proved as (\ref{th:Delta^d}), determines the group \(\Delta^d(G)\) in terms of the cardinality \(q\) of \(G\), and its rank \(r\) (the number of cyclic factors in \(G\)).
For \(G=\Zz/2\) it was proved in \cite[(8.1.2)]{ciequ}.

\begin{thm}
	\label{th:main}
	Consider the filtration by subgroups
	\[
		F^0_q\Laz \subset \ldots \subset F^d_q\Laz \subset F^{d+1}_q\Laz \subset \ldots\subset \Laz
	\]
	generated by the following conditions:
	\begin{enumerate}[label=(\roman*), ref=\roman*]
		\item \((F^d_q\Laz) \cdot (F^e_q\Laz) \subset F^{d+e}_q\Laz\),
		\item \(F^{\lfloor i/q \rfloor}_q\Laz\) contains the classes of all \(i\)-dimensional varieties, for \(i\ge 0\).
	\end{enumerate}
	Then \(\Delta^d(G)=F^d_q\Laz+\Ipn{r}\) for every \(d\).
\end{thm}
To illustrate this filtration, a typical element of \(F^d_2\Laz\) is the class of the variety \(P=\Pp^{n_1}\times \cdots \times \Pp^{n_s}\), where \(d=\lfloor n_1/2 \rfloor + \ldots + \lfloor n_s/2 \rfloor\).
The variety \(P\) indeed carries a \(\Zz/2\)-action with a \(d\)-dimensional fixed locus, given by the involution multiplying half of the coordinates of each projective space by \(-1\).\\

In Theorem~\ref{th:main}, the summand \(\Ipn{r}\) originates from Theorem~\ref{th:no_fixed}, and the first condition follows from the fact that fixed loci commute with products.
The necessity of the second condition can be observed by constructing explicit \(G\)-actions with low-dimensional fixed locus on varieties representing cobordism generators.
Theorem~\ref{th:main} establishes that these natural conditions are, in fact, sufficient.
The majority of this work is devoted to proving this result.

Recall that the Lazard ring \(\Laz\) is polynomial, with one generator in each dimension (noncanonically), so that the filtration of Theorem~\ref{th:main} can be explicitly determined.
Consequently, one obtains a lower bound for the fixed locus of a \(G\)-action on a variety \(X\) by considering the monomials present in the polynomial associated with its cobordism class \(\lc X \rc\).
More precisely, assume that \(G\) is nontrivial (\(r\ge 1\)).
Then the quotient \(\Laz/\Ipn{r}\) is a polynomial \(\Fp\)-algebra, with one generator in each dimension \(i\not \in \{p-1,\ldots,p^{r-1}-1\}\).
By assigning the degree \(\lfloor i/q \rfloor\) to the \(i\)-dimensional generator, we obtain a grading of that algebra (which depends on \(q\) and differs from the usual grading by dimension).
Theorem~\ref{th:main} asserts that the dimension of the fixed locus \(X^G\) is no less than the degree of the class of \(\lc X \rc\) modulo \(\Ipn{r}\) with respect to that grading.

Note that this bound depends only on the ambient variety \(X\) and not on the particular \(G\)-action, in the spirit of the fixed-point theorems of the previous section. Additionally, Theorem~\ref{th:main} asserts that no further conditions on the fixed locus dimension can be derived from the subgroup generated by the cobordism class of the ambient variety.\\

Let us illustrate Theorem~\ref{th:main} with a few concrete corollaries.
We fix a finite abelian \(p\)-group \(G\) of cardinality \(q\) and rank \(r\), and a smooth projective variety \(X\) with a \(G\)-action.
Some of the monomials appearing in a cobordism class can be detected by Chern numbers modulo \(p\), which yields the following corollary.
The same bound was obtained in topology by Kosniowski--Stong, for actions of the groups \(G=(\Zz/2)^r\) on manifolds (see \cite[p. 314]{KS} and \cite[p. 737]{Kosniowski-Stong-Z2k}).
\begin{cor}
	\label{cor:Chern_numbers}
	If \(\alpha=(\alpha_1,\ldots,\alpha_m)\) is a partition such that the Chern number \(c_\alpha(X)=\deg c_{\alpha}(-\Tan_X)\) is not divisible by \(p\), then
	\begin{equation}
		\label{eq:lower_bound_partition}
		\dim X^G \ge \big \lfloor \frac{\alpha_1}{q} \big \rfloor + \cdots + \big \lfloor \frac{\alpha_m}{q} \big \rfloor.
	\end{equation}
\end{cor}
For instance, we obtain that if \(X\) is a smooth hypersurface in \(\Pp^{n+1}\) of degree prime to \(p\), where \(n+2\) is divisible by \(p\), then \(\dim X^G \ge \lfloor n/q \rfloor\).

Note that the corollary involves only the cardinality \(q\) of the \(p\)-group \(G\), and not its rank \(r\).
This reinforces the observation made above that the group structure of \(G\) is irrelevant as far as Chern numbers modulo \(p\) are concerned.

But the cobordism class of the ambient variety in the Lazard ring carries substantially more \(p\)-primary information.
Indeed, it turns out that for certain partitions \(\alpha=(\alpha_1,\ldots,\alpha_m)\), the associated Chern number \(c_\alpha\) only ever takes values divisible by a certain power of \(p\);
for instance the Euler number of a smooth projective curve is always even.
Theorem~\ref{th:main} asserts that we may divide the Chern number by that power of \(p\), and use the result modulo \(p\) to obtain the lower bound \eqref{eq:lower_bound_partition}, as long as the partition \(\alpha\) is not a refinement of a partition containing one of the \(r-1\) integers \(p-1,\ldots,p^{r-1}-1\) (see (\ref{cor:dim_pi})).
This last condition is vacuous when \(r=1\);
in general it is satisfied for instance when \(\alpha_m \geq p^{r-1}\).\\

The Chern numbers associated with the partitions \(\alpha=(n)\), of length one, detect indecomposability in the Lazard ring, which yields the following corollary.
Note that when \(n+1\) is a power of \(p\), these numbers are always divisible by \(p\), and so Corollary~\ref{cor:Chern_numbers} does not apply.
\begin{cor}
	\label{cor:indec}
	If \(X\) has pure dimension \(n\not \in \{p-1,\ldots,p^{r-1}-1\}\) and its class in the graded ring \(\Laz/p\) is indecomposable, then \(\dim X^G \ge \lfloor n/q \rfloor\).
\end{cor}

Another extreme case is the description of the ring \(F^0_q\Laz\):
\begin{cor}
	\label{cor:isolated}
	In the quotient ring \(\Laz/\Ipn{r}\), the set of classes of varieties admitting a \(G\)-action with isolated fixed points is the subring generated by the classes of varieties of dimension \(<q\).
\end{cor}

Tom Dieck proved in \cite[Satz~3~(a)]{tomDieck-Periodische} a weaker, ``multiplicative'' variant of Theorem~\ref{th:main} (in the special case \(r=1\), i.e.\ when \(G\) is a cyclic \(p\)-group).
It does not exactly describe the topological version of the subgroup \(F^d_q\Laz\), but rather the subring generated by it.
This suffices to derive the analogue of Corollary~\ref{cor:indec} (in \cite[Satz~3~(b)]{tomDieck-Periodische}), but in general yields a much coarser dimensional bound than Theorem~\ref{th:main} does:
for instance in Corollary~\ref{cor:Chern_numbers} the lower bound \(\lfloor \alpha_1/q \rfloor + \cdots + \lfloor \alpha_n/q \rfloor\) would be weakened to \(\lfloor \alpha_1/q \rfloor\).
We are not aware of any analogue of Theorem~\ref{th:main} in topology, although such a result likely holds.

\subsection{Remarks}
\begin{para}
	We emphasize that the above results describe all the possible restrictions on the dimension of the fixed locus arising linearly from the Chern numbers of the ambient variety, for the action of a diagonalizable group \(G\). Explicitly, setting \(\Delta^{-\infty}_q(G)=\JK(G)\),
	\[
		\Delta^d(G) =
		\left\{
			\begin{array}{@{}l@{\quad}p{0.55\textwidth}@{}}
				\Ipn{r} & if \(d=-\infty\), and \(G\) is the product of a torus and an abelian \(p\)-group of rank \(r\),\\[3ex]
				F^p_q\Laz + \Ipn{r} & if \(G\) is an abelian \(p\)-group of rank \(r\) and order \(q\), \\[1ex]
				\Laz & otherwise.
			\end{array}
		\right.
	\]
	It is worth noting that all such restrictions turn out to come from relatively simple examples, via three basic constructions:
	\begin{itemize}
		\item[---]
			trivial actions,
		\item[---]
			linear actions on projective spaces or on projective bundles over them,
		\item[---]
			restrictions of those to degree \(p\) hypersurfaces in projective spaces.
	\end{itemize}
\end{para}
\begin{para}
	The above results are ``additive'' in nature: the conditions derived apply to the subgroup generated by the cobordism class of the ambient variety rather than the class itself.
	The difference between these two types of conditions is investigated in Appendix~\ref{app:eff}, where we discuss effective versions of the groups \(\JK(G)\) and \(\Delta^d(G)\), namely the subsets \(\JKe(G)\) and \(\Delta^d_{\eff}(G)\) of \(\Laz\), consisting of the classes of varieties admitting a \(G\)-action without fixed points, and with a fixed locus of dimension \(d\).
	We are able to determine completely the set \(\JKe(G)\) for all finite type diagonalizable groups \(G\), as well as the set \(\Delta^d_{\eff}(G)\) when \(G\) is finite abelian.
	Those sets essentially agree with their noneffective counterparts, up to their components of dimension zero, which are easily determined.
	By contrast, we observe that \(\Delta^0_{\eff}(\Gm)\) differs significantly from \(\Delta^0(\Gm)\).
\end{para}

\begin{para}
	As mentioned above, the special case \(G=\Zz/2\) in Theorem~\ref{th:main} was considered in \cite{ciequ}.
	That approach relied fundamentally on the exceptional properties of \(\Zz/2\)-actions, i.e.\ involutions, and does not extend even to the prime cyclic case \(G=\Zz/p\) for \(p\) odd.
	One crucial feature of the case \(G=\Zz/2\) is that this group always acts through the same character on the normal bundle to the fixed locus.
	For instance, this property enabled the geometric construction of Vishik's symmetric operations for the prime \(2\) in \cite{Vishik-Sym}, while other primes were treated only much later in \cite{Vishik-Sym-p} using a completely different method.

	The structure of the equivariant cobordism ring appears to be substantially more intricate when \(|G|>2\),
	and it is rather unclear whether the structural results of \cite{ciequ} for \(G=\Zz/2\) have direct analogues for other abelian \(p\)-groups, or if those are purely specific to involutions.
\end{para}

\begin{para}
	Our results are obtained using algebraic cobordism as a cohomology theory of algebraic varieties.
	The coefficient ring of this theory is precisely the Lazard ring, making it a natural choice for the study of Chern numbers.
	Other cohomology theories, such as Chow groups or \(K\)-theory, have been used to obtain fixed-point theorems and related results involving Chern numbers \cite{fpt,inv,ciequ}.
	The coefficient rings of those theories turn out to be too small for the purposes of this paper.
	For instance, while \(K\)-theory is suitable to study the actions of cyclic \(p\)-groups \cite[\S2.2]{ciequ}, it probably is not for \(p\)-groups of higher rank, as suggested by results in topology \cite{Bix-tomDieck}.
	By contrast, a key property of the Lazard ring for our inductive computation of the equivariant ring of the point is the fact that it is \emph{not} a noetherian ring.
\end{para}

\begin{para}
	We in fact prove all results over arbitrary fields of characteristic different from \(p\).
	In this situation, finite abelian groups are naturally replaced by finite diagonalizable groups, which are isomorphic to products of the groups \(\mu_n\), and the presence of roots of unity in the base field plays no role.

	Working over fields of positive characteristic forces us to use Voevodsky's cobordism spectrum \(\MGL\) instead of Levine--Morel's version of algebraic cobordism.
	This leads us to localize all cobordism rings away from the characteristic, which does not affect \(p\)-primary properties in the Lazard ring, as long as the prime \(p\) is distinct from the characteristic.

	The case of base fields of characteristic \(p\) remains largely open due to several obstacles:
	\begin{itemize}
		\item[---] there is currently no well-behaved \(p\)-primary theory of algebraic cobordism available,
		\item[---] the usual equivariant localization techniques (concentration theorem) are not available for constant finite abelian \(p\)-groups \(G\), as those are not linearly reductive,
		\item[---] the classifying space \(BG\) is \(\mathbb{A}^1\)-contractible, so no nonequivariant information will be carried by the equivariant Chow/\(K\)-theory/cobordism ring of the point.
	\end{itemize}
\end{para}

\subsection{Strategy of the proof}

Our starting point is the concentration theorem \cite[(4.5.4)]{conc} applied to the algebraic cobordism theory of Levine and Morel \cite{LM-Al-07}.
This result identifies the equivariant cobordism of a variety with a group action with that of its fixed locus, upon localization with respect to a certain multiplicative subset \(S\) of the equivariant cobordism of the point.

Assume that \(G\) is a diagonalizable group whose character group is a \(p\)-group (such groups are referred to as \emph{diagonalizable \(p\)-groups} in this paper).
We compute in \S\ref{sect:omega_G} the \(G\)-equivariant cobordism ring of the point \(\Omega_G(\Speck)\).
We then describe the kernel of the localization \(\Omega_G(\Speck) \to S^{-1}\Omega_G(\Speck)\), an ideal of \(\Omega_G(\Speck)\) which turns out to be mapped in the Lazard ring to the ideal \(\Ipn{r}\), under the forgetful morphism \(\Omega_G(\Speck) \to \Omega(\Speck)=\Laz\).
Consequently, the quotient \(\Laz/\Ipn{r}\) captures the nonequivariant trace of the concentration theorem.
This implies that the existence of fixed points is detected by the cobordism class modulo \(\Ipn{r}\), which is one inclusion in Theorem~\ref{th:no_fixed}.
The other inclusion is proved in \S\ref{sect:optimal} by exhibiting explicit fixed-point-free actions on varieties whose cobordism classes generate the ideal \(\Ipn{r}\).\\

The concentration theorem is somewhat more precise.
Within the localization \(S^{-1}\Omega_G(\Speck)\) of the equivariant cobordism ring of the point, it permits us to express the class of a variety with a \(G\)-action in terms of the normal bundle to its fixed locus, to which we will refer as the \emph{fixed data} for brevity.
Accordingly, we introduce in \S\ref{sect:Mcc} the cobordism ring \(\Mcc\) of vector bundles graded by the nontrivial characters of \(G\), where the class of the fixed data naturally resides.\\

In \S\ref{sect:optimal}, we proceed to endow a family of generators of the Lazard ring with \(G\)-actions, in such a way as to minimize the dimension of the fixed locus. This allows us in particular to verify that the group \(F^d\) of Theorem~\ref{th:main} is generated by a set of classes of varieties carrying a \(G\)-action with a fixed locus of dimension \(\le d\).\\

In \S\ref{sect:alg_indep}, we describe a method to produce algebraically independent elements in the ring \(S^{-1}\Omega_G(\Speck)\), starting from elements of \(\Omega_G(\Speck)\) whose nonequivariant images in \(\Laz/\Ipn{r}\) are algebraically independent over \(\Fp\) (when \(r \ge 1\)).\\

We then define in \S\ref{sect:Mcc_e} a morphism of \(\Laz\)-algebras \(\Mcc \to S^{-1}\Omega_G(\Speck)\), which sends the class of the fixed data to the equivariant class of the variety.
We prove in (\ref{th:inj_M}) that this morphism is injective; note that its source is a polynomial algebra in infinitely many variables, and its target the localization of a quotient of a power series algebra in \(r\) variables.
This theorem implies in particular a form of converse to the concentration theorem: the class of the fixed data in \(\Mcc\) is determined by the equivariant class of the variety.

We prove that the ring \(\Mcc\) is algebraic over the subring generated by the classes of fixed data, and \(r-1\) additional explicit elements.
In fact it suffices to use the fixed data for the explicit actions constructed in \S\ref{sect:optimal}.
This observation leads to the proof of the injectivity theorem.\\

The next step is to consider, for a given \(d \in \Nn\), the subgroup \(\Mcd{d}\) generated in \(\Mcc\) by classes of graded vector bundles over bases of dimension \(\leq d\), and to view it as a module over the ring \(\Mcd{0}\).
As such it is free of finite rank.
The explicit actions on varieties representing cobordism generators constructed in \S\ref{sect:optimal} allow us to build an explicit family \(\Fc\) of varieties carrying a \(G\)-action with a fixed locus of dimension \(\le d\). The classes of the corresponding fixed data form a linearly independent family of elements in the \(\Mcd{0}\)-module \(\Mcd{d}\), which is maximal, as its cardinality is precisely the rank of the module.

We then prove in (\ref{lemm:lin_indep:base}) that relations of linear dependence over \(\Mcd{0}\) between classes of fixed data in \(\Mcc\) imply relations of linear dependence in \(\Laz/\Ipn{r}\) for the nonequivariant classes of the corresponding varieties, over the subring generated by elements of dimension \(1, \ldots, q-1\).
The nonequivariant cobordism classes of our explicit varieties in \(\Fc\) generate, as a module over that subring, a direct summand of \(\Laz/\Ipn{r}\), which is precisely the image modulo \(\Ipn{r}\) of the group \(F^d_q\Laz\) of Theorem~\ref{th:main}.
Combining these facts, we deduce the following in (\ref{prop:main}): if a variety carries a \(G\)-action, and its nonequivariant cobordism class in the Lazard ring does not belong to \(F^d_q\Laz + \Ipn{r}\), then the class of its fixed data in \(\Mcc\) is linearly independent from the classes of the fixed data of our explicit family \(\Fc\).
By the aforementioned maximality property, this class cannot belong to \(\Mcd{d}\).
This implies that the fixed locus of the variety must have dimension \(>d\), and concludes the proof of Theorem~\ref{th:main}.

\subsection{Organization of the paper}
In \S\ref{sect:Lazard} we fix the notation, and recall known facts used throughout the paper.

In \S\ref{sect:optimal} we perform explicit constructions of actions with low-dimensional fixed locus on varieties whose classes generate the Lazard ring.
We also introduce the ``\(q\)-filtration'', which is used in the statement of the main theorem.

In \S\ref{sect:oct}, we recall the required notions on algebraic cobordism, and introduce the cobordism ring \(\Mcc\) of graded vector bundles.

In \S\ref{sect:omega_G}, we compute the \(G\)-equivariant cobordism ring of the point, when the group \(G\) is diagonalizable.

In \S\ref{sect:alg_indep}, that ring is investigated from a purely algebraic point of view.
We give a method to produce algebraically independent elements in it, and describe which nonequivariant information is lost in the concentration theorem.

In \S\ref{sect:fixed_locus} we combine the results of \S\S\ref{sect:oct}--\ref{sect:alg_indep} to prove the main theorems on actions of diagonalizable \(p\)-groups.

In \S\ref{sect:conclusion}, we conclude by stating our results for general diagonalizable groups, present concrete consequences in terms of Chern numbers, and provide a few examples.

In Appendix~\ref{app:eff} we discuss effectivity aspects, including ``nonadditive'' versions of our main results.

In Appendix~\ref{app:glossary}, to help the reader navigate the paper, we provide a list of the symbols used in the paper, focusing on those used across sections, together with references and brief descriptions.

\section{The Lazard ring}
\renewcommand{\theequation}{\thetheorem.{\alph{equation}}}
\label{sect:Lazard}
All rings are assumed to be commutative, unital, and associative.\\

In this section we fix the notation and recall a few well-known facts that will be used throughout the paper.

\subsection{The algebraic Hurewicz morphism}

\begin{para}
	\label{p:graded_power_series}
	When \(R\) is a graded (i.e.\ \(\Zz\)-graded) ring, we denote by \(R[[t]]\) the \emph{graded power series algebra}, defined as the limit of the graded \(R\)-algebras \(R[t]/t^m\) for \(m\in \Nn\), where \(t\) is homogeneous of degree \(1\).
	Thus \(R[[t]]\) is additively generated by the power series of the form \(\sum_{j \in \Nn} r_{i-j} t^j\) for \(i\in \Zz\), where each \(r_k \in R\) is homogeneous of degree \(k\in \Zz\) (it is a subset of the usual power series ring).
	More generally we consider the graded power series algebra in \(r\) variables \(R[[x_1,\ldots,x_r]]\), or when \(E\) is a set, the graded power series algebra \(R[[E]]\) in the variables \(X_e\) for \(e\in E\), where the variables are homogeneous of degree \(1\).
\end{para}

\begin{para}
	\label{p:Laz}
	We denote by \(\Laz\) the Lazard ring, defined as the coefficient ring of the universal commutative one-dimensional formal group law \cite[I, \S5]{Adams-Stable}.
	It is a graded ring;
	we denote by \(\Laz^n\) its homogeneous component of degree \(n \in \Zz\), which vanishes when \(n>0\).
\end{para}

\begin{para}
	\label{p:fgl}
	The formal group law of \(\Laz\) is denoted by
	\begin{equation}
		\label{eq:fgl}
		x \fgl y \in \Laz[[x,y]].
	\end{equation}
	Formal multiplication by an integer \(n\in \Zz\) is denoted by
	\begin{equation}
		\label{eq:fgl_n}
		\fglr{n}(t) \in \Laz[[t]].
	\end{equation}
	For \(n\ge 0\), this is defined by setting \(\fglr{0}(t)=0\) and recursively \(\fglr{n+1}(t)=(\fglr{n}(t)) \fgl t\).
	We let \(\fglr{-1}(t)\) denote the formal inverse, and for \(n <0\) we set \(\fglr{n}(t)=\fglr{-1}(\fglr{-n}(t))\).

	The power series \eqref{eq:fgl} and \eqref{eq:fgl_n} are homogeneous of degree \(1\).
	If \(R\) is a \(\Laz\)-algebra which is complete with respect to the ideal generated by some elements \(a,b\in R\), we denote by \(a \fgl b\) the image of \(x \fgl y\) under the morphism \(\Laz[[x,y]] \to R\) given by \(x \mapsto a, y \mapsto b\).
	Similarly if \(R\) is complete with respect to the ideal generated by \(s\), we denote by \(\fglr{n}(s)\) the image of \(\fglr{n}(t)\) under the morphism \(\Laz[[t]] \to R\) given by \(t \mapsto s\).
\end{para}

\begin{para}
	\label{p:L_in_Zb}
	When \(R\) is a graded ring, we consider the graded polynomial \(R\)-algebra
	\[
		R[\bb]=R[b_i, i \in \Nn\smallsetminus \{0\}],
	\]
	where \(b_i\) is homogeneous of degree \(-i\). We also set \(b_0=1\).
	The power series
	\begin{equation}
		\label{eq:exp}
		\exp(t)=\sum_{i\in \Nn} b_i t^{i+1} \in R[\bb][[t]]
	\end{equation}
	admits a compositional inverse \(\exp^{-1}(t)\).
	We consider the ring morphism
	\begin{equation}
		\label{eq:Laz_Zb}
		\Laz \to \Zz[\bb]
	\end{equation}
	classifying the formal group law
	\begin{equation}
		\label{eq:fglr_zb}
		(x,y) \mapsto \exp(\exp^{-1}(x)+\exp^{-1}(y)).
	\end{equation}
	This morphism is injective \cite[II, Theorem~7.8]{Adams-Stable}.
\end{para}

\begin{para}
	\label{p:flgr_exp}
	For every \(n\in \Nn\), the morphism \(\Laz[[t]] \to (\Zz[\bb])[[t]]\) induced by \eqref{eq:Laz_Zb} sends the power series \(\fglr{n}(t)\) to \(\exp(n \exp^{-1}(t))\).
\end{para}

\begin{para}
	\label{p:partitions}
	A \emph{partition} \(\alpha\) is a (possibly empty) sequence of integers \((\alpha_1,\ldots, \alpha_n)\) verifying \(\alpha_1 \ge \cdots \ge \alpha_n >0\).
	The integer \(n \in \Nn\) is the \emph{length} of the partition \(\alpha\), denoted by \(\ell(\alpha)\in \Nn\).
	The \emph{weight} of \(\alpha\) is the integer \(|\alpha|=\alpha_1+\ldots+\alpha_n\).
	To the partition \(\alpha\) corresponds the monomial \(b_{\alpha}=b_{\alpha_1}\cdots b_{\alpha_n}\in \Zz[\bb]\).
	Applying the morphism \eqref{eq:Laz_Zb} and taking the \(b_{\alpha}\)-coefficient yields a morphism of abelian groups:
	\begin{equation}
		\label{eq:c_alpha}
		c_{\alpha} \colon \Laz \to \Zz.
	\end{equation}
\end{para}

\begin{para}
	If \(\alpha=(\alpha_1,\dots,\alpha_n)\) and \(\beta=(\beta_1,\dots,\beta_m)\) are partitions, we denote by \(\alpha \cup \beta\) the partition obtained by reordering the tuple \((\alpha_1,\dots,\alpha_n,\beta_1,\dots,\beta_m)\).
\end{para}

\begin{para}
	\label{p:c_product}
	If \(\alpha\) is a partition, we have for every \(x,y \in \Laz\)
	\[
		c_\alpha(xy) = \sum_{\beta \cup \gamma = \alpha} c_\beta(x) c_\gamma(y).
	\]
\end{para}

\begin{para}
	\label{p:decomposable}
	An element of a graded ring is called \emph{decomposable} if it belongs to \(I^2\), where \(I\) is the ideal generated by homogeneous elements of nonzero degree.
	An element which is not decomposable is called \emph{indecomposable}.
\end{para}

\begin{para}
	\label{p:decomp_zb}
	In the graded ring \(\Zz[\bb]\), the ideal of decomposable elements is generated by the elements \(b_ib_j\) for \(i,j>0\).
	It follows that, for every \(i\in \Nn\), the map \(c_{(i)}\colon \Laz \to \Zz\) vanishes on elements having decomposable image in \(\Zz[\bb]\).
\end{para}

\begin{para}[{See \cite[II, \S 7.8]{Adams-Stable}}]
	\label{p:pol_gen_laz}
	The \(\Zz\)-algebra \(\Laz\) is polynomial, with one generator in each degree \(-i\), for \(i \in \Nn \smallsetminus \{0\}\).
	A family \(\ell_i \in \Laz^{-i}\) for \(i \in \Nn \smallsetminus \{0\}\) is a set of polynomial generators if and only if for every \(i\) the integer \(c_{(i)}(\ell_i)\) generates the group \(c_{(i)}(\Laz^{-i})\), or equivalently
	\begin{equation}
		\label{eq:ci_li}
		c_{(i)}(\ell_i)=
		\begin{cases}
			\pm 1 & \text{if \(i+1\) is not a power of a prime number}, \\
			\pm p & \text{if \(i+1\) is a power of the prime number \(p\)}.
		\end{cases}
	\end{equation}
	Given a partition \(\alpha=(\alpha_1,\ldots,\alpha_n)\), we write \(\ell_\alpha = \ell_{\alpha_1}\cdots \ell_{\alpha_n}\).
\end{para}

\begin{para}
	\label{p:Ld_pol}
	For \(d\in \Nn\), denote by \(\Laz(d)\) the subring of \(\Laz\) generated by the homogeneous elements of degree \(\ge -d\). If the elements \(\ell_i\in \Laz^{-i}\) form a family of polynomial generators of the \(\Zz\)-algebra \(\Laz\), then the \(\Zz\)-algebra \(\Laz(d)\) is polynomial in the elements \(\ell_1,\ldots,\ell_d\).
\end{para}

\begin{para}
	\label{p:cn_dec}
	Let \(y \in \Laz^{-n}\) with \(n >0\). It follow from (\ref{p:pol_gen_laz}) that \(y\) is decomposable in \(\Laz\) if and only if \(c_{(n)}(y)=0\).
\end{para}

\begin{para}
	\label{p:gen_indec}
	Let \(p\) be a prime number.
	Then a family \(y_i \in \Laz^{-i}/p\) is a set of polynomial generators of the \(\Fp\)-algebra \(\Laz/p\) if and only if each \(y_i\) is indecomposable in the graded ring \(\Laz/p\).
\end{para}

\begin{para}
	\label{p:indec_cn}
	Let \(y \in \Laz^{-n}\) with \(n >0\), and let \(p\) be a prime number.
	It follows from (\ref{p:pol_gen_laz}) that an element \(y\) reduces modulo \(p\) to an indecomposable element of \(\Laz/p\) if and only if the integer \(c_{(n)}(y)\) is not divisible by \(p\) when \(n+1\) is not a power of \(p\), and by \(p^2\) when \(n+1\) is a power of \(p\) (see also \cite[(7.3.2.iii)]{inv}).
\end{para}

\begin{para}
	\label{p:def_succ}
	Let \(\alpha,\beta\) be partitions.
	Write \(\beta = (\beta_1,\ldots,\beta_m)\).
	We say that \(\alpha\) is a refinement of \(\beta\), and write \(\alpha \succeq \beta\), when there exist partitions \(\alpha^1,\ldots,\alpha^m\) such that \(\beta_i = |\alpha^i|\) for all \(i=1,\ldots,m\) and \(\alpha = \alpha^1 \cup \cdots \cup \alpha^m\).
\end{para}

\begin{para}
	\label{p:succ_length}
	If \(\alpha \succeq \beta\), then \(\ell(\alpha) \ge \ell(\beta)\), with equality only when \(\alpha=\beta\).
\end{para}

\begin{lemma}
	\label{lemm:succ}
	Let \(\ell_i \in \Laz^{-i}\) be a family of polynomial generators of \(\Laz\). If \(\alpha,\beta\) are partitions such that \(c_\alpha(\ell_\beta) \neq 0\) then \(\alpha \succeq \beta\).
\end{lemma}
\begin{proof}
	Write \(\beta = (\beta_1,\dots,\beta_m)\). Then by (\ref{p:c_product}) we have
	\[
		c_\alpha(\ell_\beta) = c_\alpha(\ell_{\beta_1} \cdots \ell_{\beta_m}) = \sum_{\alpha^1 \cup \dots \cup \alpha^m=\alpha} c_{\alpha^1}(\ell_{\beta_1}) \cdots c_{\alpha^m}(\ell_{\beta_m}).
	\]
	Thus there exist partitions \(\alpha^1,\dots,\alpha^m\) such that \(\alpha^1 \cup \dots \cup \alpha^m=\alpha\) and \(c_{\alpha^i}(\ell_{\beta_i}) \neq 0\) for \(i=1,\dots,m\). By degree reasons, we must have \(|\alpha^i| = \beta_i\) for \(i=1,\dots,m\), which proves that \(\alpha \succeq \beta\).
\end{proof}

\subsection{Cobordism classes of varieties}
In this section, we fix a base field \(k\), and recall how the Lazard ring can be interpreted in terms of Chern numbers of smooth projective \(k\)-varieties.
\begin{para}
	\label{p:CF}
	(See e.g.\ \cite[(3.1.2), (3.1.4)]{inv}.)
	Every line bundle \(L\) over a smooth \(k\)-variety \(X\) admits a first Chern class \(c_1(L) \in \CH(X)\) with values in the Chow ring, which is nilpotent.
	There is a unique way to define for every smooth projective \(k\)-variety \(X\) a map \(P \colon K_0(X) \to \CH(X)[\bb]\) such that (here \(K_0\) refers to the Grothendieck group of vector bundles on \(X\)):
	\begin{enumerate}[label=(\roman*), ref=\roman*]
		\item \(f^*P(E) = P(f^*E)\) for any morphism \(f\colon Y \to X\) and \(E\in K_0(X)\),

		\item \(P(L) = \displaystyle{\sum_{i\in \Nn} c_1(L)^i b_i}\) when \(L\) is the class of a line bundle over \(X\),
		\item \(P(E + F) = P(E) P(F)\) for any \(E,F \in K_0(X)\).
	\end{enumerate}
	Given \(E \in K_0(X)\), the \emph{Conner--Floyd Chern classes} \(c_\alpha(E) \in \CH(X)\), where \(\alpha\) runs over the partitions, are defined by the formula
	\begin{equation}
		\label{eq:P_CF}
		P(E) = \sum_\alpha c_\alpha(E) b_\alpha \in \CH(X)[\bb].
	\end{equation}
	Note that \(c_{\alpha}(E)\) is nonzero only for finitely many partitions \(\alpha\).
\end{para}

\begin{definition}
	\label{def:lc_rc}
	Let \(X\) be a smooth projective \(k\)-variety.
	Using (\ref{p:CF}), to a partition \(\alpha\) one associates the \emph{Chern number}
	\[
		c_\alpha(X) = \deg c_{\alpha}(-\Tan_X) \in \Zz.
	\]
	We set
	\begin{equation}
		\label{eq:Chern_numbers}
		\lc X \rc = \sum_{\alpha} c_\alpha(X) b_\alpha \in \Zz[\bb],
	\end{equation}
	where \(\alpha\) runs over the partitions. In other words, using the notation of (\ref{eq:c_alpha}) we have \(c_{\alpha}(\lc X \rc)= c_\alpha(X)\) for every partition \(\alpha\).	
\end{definition}

\begin{para}
	\label{p:L_Zb}
	Recall from (\ref{p:L_in_Zb}) that we may view \(\Laz\) as a graded subring of \(\Zz[\bb]\).
	It is proved in \cite[Theorem~8.2]{Mer-Ori} that the subgroup of \(\Zz[\bb]\) generated by the classes \(\lc X \rc\), where \(X\) runs over the smooth projective \(k\)-varieties, is \(\Laz\).
\end{para}

\begin{lemma}
		\label{lemm:c_n_hyp}
	Let \(X\) be a smooth hypersurface of degree \(d\) in \(\Pp^{n+1}\), with \(n\ge 1\).
	Then
	\[
		c_{(n)}(X) = d(d^n -n -2).
	\]
\end{lemma}
\begin{proof}
	Let \(i\colon X \to \Pp^{n+1}\) be the closed immersion.
	The class of the tangent bundle of \(X\) in \(K_0(X)\) is the restriction of \(-\Oc(d) + \Oc(1)^{\oplus (n+2)}-1 \in K_0(\Pp^{n+1})\).
	Using the additivity of the Conner--Floyd Chern class \(c_{(n)}\), we compute in \(\CH_0(X)\)
	\[	
		c_{(n)}(-\Tan_X) = i^*\big( c_1(\Oc(d))^n - (n+2)c_1(\Oc(1))\big) = (d^n- n-2) \cdot i^*c_1(\Oc(1))^n.
	\]
	We conclude by taking the degree, using the projection formula and the fact that \(i_*(1)= d\cdot c_1(\Oc(1))\).
\end{proof}

\begin{example}
	\label{ex:c_n_hypersurface}
	Let \(X\) be a smooth hypersurface of degree \(d\) in \(\Pp^{n+1}\), with \(n\ge 1\).
	If \(p\) is a prime number, it follows from (\ref{p:indec_cn}) and (\ref{lemm:c_n_hyp}) that the class \(\lc X \rc\) is indecomposable in the graded ring \(\Laz/p\) under any of the following assumptions:
	\begin{enumerate}[label=(\roman*), ref=\roman*]
		\item
			\(n=p-1\), and \(d\) prime to \(p\),
		\item
			\label{ex:c_n_hypersurface:2}
			\(n+1\) a power of \(p\), and \(d\in p \Zz\smallsetminus p^2 \Zz\),
		\item
			\(n = -2 \mod p\), and \(d\) prime to \(p\).
	\end{enumerate}
\end{example}

\subsection{Landweber ideals}
In this section, we fix a prime number \(p\).
\begin{definition}
	\label{def:Ipn}
	Denote by \(\Ipn{\infty} \subset \Laz\) the kernel of the composite
	\[
		\Laz \to \Zz[\bb] \to \Fp[\bb],
	\]
	where the first map is \eqref{eq:Laz_Zb}, and the second is induced by the ring morphism \(\Zz\to \Fp\).
	This is a homogeneous ideal of \(\Laz\).
	For \(n\in \Nn \smallsetminus \{0\}\), we denote by \(\Ipn{n}\) the ideal of \(\Laz\) generated by the homogeneous elements in \(\Ipn{\infty}\) of degree \(\ge 1-p^{n-1}\).
	So, in particular \(\Ipn{1} = p \Laz\).
	We also set \(\Ipn{0} = 0 \subset \Laz\).
\end{definition}

\begin{para}
	\label{rem:Ipn_Chern_numbers}
	As mentioned in (\ref{def:lc_rc}) and (\ref{p:L_Zb}), fixing a base field \(k\), every smooth projective \(k\)-variety has a class \(\lc X \rc \in \Laz\).
	Then \(\Ipn{n}\) is the subgroup of \(\Laz\) generated by the classes \(\lc X \rc\), where \(X\) runs over the smooth projective \(k\)-varieties of dimension \(\le p^{n-1}-1\) having all Chern numbers divisible by \(p\).
\end{para}

\begin{definition}
	\label{def:u_n}
	Consider the elements \(u_m \in \Laz^{-m}\) for \(m\in \Nn\) such that
	\[
		\fglr{p}(t) = \sum_{m \in \Nn} u_m t^{m+1} \in \Laz[[t]].
	\]
	For \(n\in \Nn\), set \(v_n = u_{p^n-1} \in \Laz^{1-p^n}\).
\end{definition}

\begin{para}
	Note that \(v_0=p\).
\end{para}

\begin{lemma}
	\label{lemm:u_in_I}
	We have \(u_m \in \Ipn{\infty}\) for every \(m\in \Nn\).
\end{lemma}
\begin{proof}
	It follows from (\ref{p:flgr_exp}) that \(\fglr{p}(t) \in \Laz[[t]]\) has vanishing image in \(\Fp[\bb][[t]]\).
\end{proof}

\begin{lemma}
	\label{lemm:u_indec}
	For \(n\ge 1\), the element \(v_n\) has indecomposable image in \(\Laz/p\).
\end{lemma}
\begin{proof}
	Denote by \(D\subset \Zz[\bb]\) the ideal of decomposable elements (in \(\Zz[\bb]\)).
	Let \(\alpha=\exp^{-1}(t) \in \Laz[[t]]\) (see (\ref{eq:exp})).
	Then in \((\Zz[\bb]/D)[[t]]\), we have for any \(j>0\)
	\begin{equation}
		\label{eq:tb_ab}
		tb_j=\exp(\alpha) b_j = \alpha b_j + \sum_{i >0} \alpha^{i+1} b_i b_j = \alpha b_j,
	\end{equation}
	since \(b_ib_j\in D\).
	Therefore, denoting by \(\delta\colon \Laz \to \Zz[\bb]/D\) the natural morphism (as well as the morphism induced on the graded power series rings), we have in \((\Zz[\bb]/D)[[t]]\)
	\begin{align*}
		\delta(\fglr{p}(t))=\exp(p\alpha)
		& =\sum_{i\ge 0} p^{i+1}\alpha^{i+1}b_i
		& \text{by \eqref{p:flgr_exp}},\\
		& =p\sum_{i\ge 0} \alpha^{i+1}b_i+ \sum_{i>0} p(p^i-1)\alpha^{i+1}b_i\\
		& =pt + \sum_{i>0} p(p^i-1)\alpha^{i+1}b_i\\
		& =pt + \sum_{i>0} p(p^i-1)t^{i+1}b_i
		& \text{by \eqref{eq:tb_ab}}.
	\end{align*}
	This means that, for each \(i>0\) we have \(\delta(u_i)=p(p^i-1)b_i\) in \(\Zz[\bb]/D\).
	By (\ref{p:decomp_zb}), we deduce that \(c_{(i)}(u_i)=p(p^i-1)\).
	Taking \(i=p^n-1\), we conclude using (\ref{p:indec_cn}).
\end{proof}

\begin{lemma}
	\label{lemm:w_I}
	Let \(w_j\in \Laz^{1-p^j}\) for \(j\in \Nn \smallsetminus \{0\}\), and set \(w_0=p\).
	For each \(j >0\), assume that \(w_j\) belongs to \(\Ipn{\infty}\) and has indecomposable image in \(\Laz/p\).

	Then for each \(n\in \Nn\), the ideal \(\Ipn{n} \subset \Laz\) is generated by the elements \(w_0,\ldots,w_{n-1}\).
	In addition, this ideal is prime and does not contain \(w_n\).
\end{lemma}
\begin{proof}	
	The case \(n=0\) is clear, as \(\Ipn{0}=0\), and \(\Laz\) is a domain where \(p\ne 0\) (the \(\Zz\)-algebra \(\Laz\) being polynomial).

	So we assume that \(n>0\).
	Let \(J(\infty)\) be the ideal of \(\Laz\) generated by the elements \(w_j\) for \(j\in \Nn\), and let \(J(n)\subset \Laz\) be the ideal generated by the homogeneous elements of degree \(\ge 1-p^{n-1}\) of \(J(\infty)\).
	As \(\Laz\) has no nonzero homogeneous elements of positive degree, it follows that \(J(n)\) is generated by \(w_0,\ldots,w_{n-1}\).

	Let \(\ell_i \in \Laz^{-i}\) for \(i\in \Nn\smallsetminus \{0\}\) be elements reducing modulo \(p\) to polynomial generators of the \(\Fp\)-algebra \(\Laz/p\).
	As the elements \(w_i\) are indecomposable in \(\Laz/p\), by (\ref{p:gen_indec}) we may assume that \(\ell_{p^j-1}=w_j\) for each \(j \in \Nn\smallsetminus \{0\}\).
	Then \(\Laz/J(n)\), resp.\ \(\Laz/J(\infty)\), is a polynomial \(\Fp\)-algebra in the variables \(\ell_i\) for \(i \ne p^s-1 \) for all \(s\in \{0,\ldots,n-1\}\), resp.\ \(s\in \Nn\).
	In particular the ideal \(J(n)\) of \(\Laz\) is prime.

	Since \(J(\infty)\subset \Ipn{\infty}\) by assumption, we have an induced morphism \(\psi \colon \Laz/J(\infty) \to \Fp[\bb]\), which by \eqref{eq:ci_li} and (\ref{p:cn_dec}) maps \(\ell_i\) to a nonzero multiple of \(b_i\) modulo decomposable elements, when \(i+1\) is not a power of \(p\).
	Therefore the elements \(\psi(\ell_i) \in \Fp[\bb]\), for \(i\in \Nn\smallsetminus \{0\}\) such that \(i+1\) is not a power of \(p\), are algebraically independent over \(\Fp\).
	Given the above description of \(\Laz/J(\infty)\), this implies that the morphism \(\psi\) is injective, which means that \(J(\infty)=\Ipn{\infty }\).
	Therefore \(\Ipn{n}=J(n)\) is generated by \(w_0,\ldots,w_{n-1}\).

	Finally, the ideal \(\Ipn{n}\) is generated by \(w_0=p\), together with elements which are homogeneous of degree \(k\) with \(1-p^n <k <0\).
	Therefore \(\Ipn{n}\) contains no homogeneous element of degree \(p^n-1\) having indecomposable image in \(\Laz/p\), and so in particular it does not contain \(w_n\).
\end{proof}

\begin{proposition}
	Let \(n\in \Nn\). Then the ideal \(\Ipn{n}\) of \(\Laz\)
	\label{prop:coeff_I}
	\begin{enumerate}[label=(\roman*), ref=\roman*]
		\item \label{prop:coeff_I:prime}
			is prime,
		\item \label{prop:coeff_I:gen}
			is generated by the elements \(v_i\) for \(i \in \{0,\ldots,n-1\}\),

		\item
			\label{prop:coeff_I:v_n}
			does not contain \(v_n\),
		\item
			\label{prop:coeff_I:u_m}
			contains \(u_m\) for \(m<p^n-1\).
	\end{enumerate}
\end{proposition}
\begin{proof}
	The statements \eqref{prop:coeff_I:prime}, \eqref{prop:coeff_I:gen}, \eqref{prop:coeff_I:v_n} follow from (\ref{lemm:w_I}), given (\ref{lemm:u_in_I}) and (\ref{lemm:u_indec}).

	Now \eqref{prop:coeff_I:gen} implies that \(\Ipn{\infty}\) is generated by the elements \(v_i\) for \(i\in \Nn\), and thus by (\ref{lemm:u_in_I}), for each \(m\in \Nn\) the element \(u_m\) is a \(\Laz\)-linear combination of the elements \(v_i\).
	Since \(\Laz\) contains no nonzero homogeneous elements of positive degree, such combination only involves the \(v_i\)'s for \(p^i-1 \le m\).
	Together with \eqref{prop:coeff_I:gen}, this implies \eqref{prop:coeff_I:u_m}.
\end{proof}

\begin{para}
	\label{p:su_n}
	When \(M\) is an \(\Laz\)-module, we write for every \(n\in \Nn\)
	\[
		\su{M}{n}=M/(\Ipn{n}M)=M \otimes_{\Laz} (\Laz/\Ipn{n}).
	\]
	We will often use the same notation for an element of \(M\) and its image in the quotient \(\su{M}{n}\).
\end{para}

\begin{para}
	\label{p:Laz_La}
	Let \(\La=\Laz[c^{-1}]\) for some \(c\in \Zz\smallsetminus p \Zz\), and \(n\in \Nn\smallsetminus \{0\}\).
	As \(p=v_0 \in \Ipn{n}\), it follows that the morphism \(\su{\Laz}{n} \to \su{\La}{n}\) is bijective.
\end{para}

\begin{para}
	\label{p:Laz_n_d}
	For \(n,d \in \Nn\), we denote by \(\su{\Laz}{n}(d)\) the subring of \(\su{\Laz}{n}\) generated by the homogeneous elements of degrees \(\ge -d\).
\end{para}

\begin{para}
	\label{p:N}
	For \(n \in \Nn\), we consider the subset of \(\Nn\smallsetminus \{0\}\)
	\[
		\su{N}{n}=\{\text{\(i\in \Nn \smallsetminus \{0\}\) such that \(i \ne p^m -1 \) for \(m=1,\ldots ,n-1\)}\}.
	\]
	Note that \(\su{N}{0}=\su{N}{1}=\Nn\smallsetminus \{0\}\).
	If \(d\in \Nn\), we set
	\[
		\su{N}{n}(d)=\{\text{\(i\in \{1,\ldots , d\} \) such that \(i \ne p^m -1 \) for \(m=1,\ldots ,n-1\)} \}.
	\]
\end{para}

\begin{para}
	\label{p:gen_in_I}
	Let \(\ell_i \in \Laz^{-i}\) be a family of polynomial generators of \(\Laz\).
	Let \(n \in \Nn\smallsetminus \{0\}\).
	Then it follows from (\ref{lemm:u_indec}) that for each \(i=1,\ldots,n-1\) there exists \(a_i \in \Laz\) such that \(v_i + pa_i -\ell_{p^i-1}\) is decomposable in \(\Laz\).
	By (\ref{prop:coeff_I}.\ref{prop:coeff_I:gen}), replacing \(\ell_{p^i-1}\) with \(v_i + pa_i\) (which is still in \(\Ipn{n}\)), we may thus arrange that the ideal \(\Ipn{n}\) is generated by \(p\) and the elements \(\ell_{p^i-1}\) for \(i=1,\ldots,n-1\).
\end{para}

\begin{lemma}
	\label{lemm:pol_gen_L}
	Let \(\ell_i \in \Laz^{-i}\), for \(i\in \Nn\smallsetminus \{0\}\), be a family of polynomial generators of the ring \(\Laz\).
	Let \(n\in \Nn\smallsetminus \{0\}\).
	Then the \(\Fp\)-algebra \(\su{\Laz}{n}\) (resp.\ \(\su{\Laz}{n}(d)\) for \(d \in \Nn\)) is polynomial in the elements \(\ell_j\) for \(j\in \su{N}{n}\) (resp.\ \(j\in \su{N}{n}(d)\)).
\end{lemma}
\begin{proof}
	By (\ref{p:gen_in_I}), we may assume that the ideal \(\Ipn{n}\) is generated by \(p\) and \(\ell_j\) for \(j \not \in \su{N}{n}\).
	The statement then follows from (\ref{p:Ld_pol}).
\end{proof}

\begin{definition}
	\label{def:Landweber_exact}
	An \(\Laz\)-module \(M\) is called \emph{Landweber exact} (at the prime \(p\)) if for every \(n\in \Nn\) the element \(v_n\in \Laz\) is a nonzerodivisor in \(\su{M}{n}\).
\end{definition}

\begin{para}
	\label{p:Laz_exact}
	It follows from (\ref{prop:coeff_I}.\ref{prop:coeff_I:prime}) and (\ref{prop:coeff_I}.\ref{prop:coeff_I:v_n}) that the \(\Laz\)-module \(\Laz\) is Landweber exact, hence so is any localization of this module.
\end{para}

\section{Actions on cobordism generators}
\label{sect:optimal}
In this section we perform explicit constructions of varieties equipped with actions with ``minimal'' fixed loci.
This will serve a double objective: besides directly establishing the optimality of our lower bounds for the fixed locus dimension, this construction is essential to our method of obtaining those bounds in the sequel.\\

The letter \(k\) denotes the base field, and we will usually write again \(\Speck\) instead of \(\Spec k\).
By a \emph{\(k\)-variety} we mean a quasi-projective scheme over \(k\).

\subsection{Group actions on varieties}
We fix an algebraic group \(G\) over \(k\).
We refer to e.g.\ \cite[\S1.f]{Milne-AG} for the notion of \(G\)-action on a \(k\)-variety, and related basic concepts.
In particular, the variety \(\Speck\) always carries the trivial \(G\)-action.

\begin{para}
	\label{p:fixed_locus}
	Let \(X\) be a \(k\)-variety with a \(G\)-action.
	The \emph{fixed locus} \(X^G\) is a closed subscheme of \(X\) characterized by the following property.
	For any \(k\)-variety \(T\), the subset \(X^G(T)\subset X(T)\) consists of those morphisms \(T\to X\) which are \(G\)-equivariant with respect to the trivial \(G\)-action on \(T\).
	Its existence is proved e.g.\ in \cite[\S7.b]{Milne-AG}.

	We say that the \(G\)-action on \(X\) is \emph{fixed-point-free} when \(X^G=\varnothing\).
\end{para}

\begin{para}
	Let \(X\) be a \(k\)-variety with a \(G\)-action.
	A \emph{\(G\)-equivariant vector bundle} is a vector bundle \(E\to X\) equipped with a \(G\)-equivariant structure on its \(\Oc_X\)-module of sections, in the sense of \cite[I, \S6.5]{SGA3-1}.
	Such a structure induces a \(G\)-action on the \(k\)-variety \(E\) such that the morphism \(E\to X\) is \(G\)-equivariant, and in addition the action is ``linear'' on the fibers, see \cite[I, Remarque~6.5.3]{SGA3-1}.
\end{para}

\begin{proposition}
	\label{prop:fixed_smooth}
	Let \(X\) be a smooth \(k\)-variety equipped with an action of a diagonalizable group \(G\).
	Then the \(k\)-variety \(X^G\) is smooth.
	In addition the normal bundle \(N\) to immersion \(X^G \to X\) is naturally \(G\)-equivariant, and \(N^G=0\).
\end{proposition}
\begin{proof}
	See \cite[(3.5.2),(3.5.3)]{isol} or \cite[(1.4.7)]{conc}.
\end{proof}

\begin{para}
	\label{p:L_g}
	We denote by \(\widehat{G}\) the group of characters of \(G\).
	Each character \(g\in \widehat{G}\) corresponds to a \(G\)-equivariant line bundle over \(\Speck\) that we denote by \(\character{g}\).
	We will use the same notation for its pull-back to any \(k\)-variety with trivial \(G\)-action.
\end{para}

\begin{para}
	\label{p:character}
	Assume that the algebraic group \(G\) is diagonalizable.
	Let \(X\) be a \(k\)-variety with trivial \(G\)-action, and let \(E\to X\) be a \(G\)-equivariant vector bundle.
	There are unique \(G\)-equivariant vector bundles \(\wgt{E}{g} \to X\) with trivial \(G\)-action, for \(g\in \widehat{G}\), such that
	\begin{equation}
		\label{eq:decomp_weights}
		E = \bigoplus_{g \in \widehat{G}} \wgt{E}{g} \otimes \character{g}.
	\end{equation}
	This is the weight decomposition with respect to the \(G\)-action.
	Note that \(\wgt{E}{g}\) is nonzero only for finitely many characters \(g\in \widehat{G}\), and that \(E^G= \wgt{E}{0}\).
\end{para}

\subsection{Generators of the Landweber ideals}
The next proposition is an algebraic version of a result of Floyd \cite[Theorem~2.1]{Floyd-stationary} in topology.

\begin{proposition}
	\label{prop:no_fixed_on_gen}
	Let \(p\) be a prime number different from the characteristic of \(k\).
	Then there exist smooth projective \(k\)-varieties \(Y_s\) of pure dimension \(p^s-1\), for all \(s \in \Nn\), satisfying the following conditions.
	\begin{enumerate}[label=(\roman*), ref=\roman*]
		\item
			For every \(n\in \Nn\), the ideal \(\Ipn{n}\) is generated by the classes \(\lc Y_s \rc\) for \(s \in \{0,\ldots,n-1\}\).
		\item
			\label{prop:no_fixed_on_gen:2} Let \(s \in \Nn\).
			Let \(G\) be an algebraic group over \(k\) whose character group \(\widehat{G}\) contains at least \(p^s+1\) elements of order dividing \(p\).
			Then the variety \(Y_s\) admits a fixed-point-free \(G\)-action.
	\end{enumerate}
\end{proposition}
\begin{proof}
	Let \(s\in \Nn\).
	Pick a subset \(C\subset \widehat{G}\) such that \(|C|=p^s+1\), and \(pC=0\).
	Using the notation of (\ref{p:L_g}), let
	\[
		V= \bigoplus_{c\in C} \character{c}.
	\]
	The group \(G\) acts on the projective space \(\Pp(V)\), with \(\Pp(V)^G\) consisting of the \(p^s+1\) rational points \(\Pp(\character{c})\) for \(c\in C\). For each \(c\in C\), the projection \(V\to \character{c}\) induces a \(G\)-equivariant section \(x_c\) of the line bundle \(\Oc(1)\otimes \character{c}\) over \(\Pp(V)\). Using the canonical isomorphisms \((\character{c})^{\otimes p}\simeq \character{pc}=\character{0}\simeq 1\), we construct a \(G\)-equivariant section
	\[
		\sum_{c\in C} (x_c)^{\otimes p}
	\]
	of the \(G\)-equivariant line bundle \(\Oc(p)\) over \(\Pp(V)\).
	It is transverse to the zero section, and its zero locus is a \(G\)-invariant closed subscheme \(Y_s \subset \Pp(V)\), of pure dimension \(p^s-1\), which is smooth over \(k\) by the characteristic assumption.
	In addition \(Y_s\) contains none of the points \(\Pp(\character{c})\) for \(c\in C\), and thus \((Y_s)^G=\varnothing\), proving \eqref{prop:no_fixed_on_gen:2}.

	As the variety \(Y_s\) is a degree \(p\) hypersurface in \(\Pp^{p^s}\), it follows easily that every Chern number of \(Y_s\) is divisible by \(p\) (as \(\Tan_{Y_s}\) is the restriction of a virtual vector bundle on \(\Pp(V)\), each of its Conner--Floyd Chern classes is the pullback of an element of \(\CH(\Pp(V))\), hence has degree divisible by \(p\), by the projection formula).
	Therefore \(\lc Y_s\rc\in \Ipn{s+1}\) by (\ref{rem:Ipn_Chern_numbers}).
	Since \(Y_0\) is a degree \(p\) hypersurface in \(\Pp^1\) (explicitly \(Y_0 \simeq \mu_p\)), it follows that \(\lc Y_0 \rc = p \in \Laz\).
	If \(s\ge 1\) then by (\ref{ex:c_n_hypersurface}.\ref{ex:c_n_hypersurface:2}) the class \(\lc Y_s \rc\) is indecomposable in \(\Laz/p\).
	The statement then follows from (\ref{lemm:w_I}).
\end{proof}

\begin{example}
	The condition on \(G\) in (\ref{prop:no_fixed_on_gen:2}) is verified when \(G=(\mu_p)^{s+1}\), or more generally when \(G\) is a diagonalizable \(p\)-group of rank \(r\), with \(r\ge s+1\) (see (\ref{p:p_rank}) below).
\end{example}

\begin{remark}
	Assume that the field \(k\) contains a root of unity of order \(p\).
	\begin{enumerate}[label=(\roman*), ref=\roman*]
		\item
			The variety \(Y_0\) constructed in the proof of (\ref{prop:no_fixed_on_gen}) is the union of \(p\) rational points.
		\item
			For \(s=1\), another possible choice would have been \(Y_1 = \Pp^{p-1}\) (see \cite[(5.1.5)]{fpt}).
	\end{enumerate}
\end{remark}

\subsection{Milnor hypersurfaces}
In this section we endow varieties whose classes form a set of polynomial generators of the Lazard ring with an action of an algebraic group possessing a given number of characters, in such a way as to minimize the dimension of the fixed locus.

We fix an algebraic group \(G\) over \(k\), and a finite set of characters \(Q\subset \widehat{G}\) containing the trivial character \(0\). We set \(q=|Q|\).
\begin{para}
	\label{p:Milnor}
	Let \(0 \leq m \leq n\) be integers such that \(n\geq 1\).
	Let us consider the vector bundle \(U\) over \(\Pp^m\) defined as the kernel of the epimorphism \(1^{\oplus n+1} \to 1^{\oplus m+1} \to \Oc(1)\), where the first map is given by forgetting the \(n-m\) last coordinates, and the second is the canonical epimorphism.
	We define the variety \(H_{m,n} = \Pp_{\Pp^m}(U)\), a smooth projective connected \(k\)-variety of dimension \(m+n-1\).

	The inclusion \(U \subset 1^{\oplus n+1}\) makes \(H_{m,n}\) a hypersurface in \(\Pp_{\Pp^m}(1^{\oplus n+1})=\Pp^m \times \Pp^n\), the \emph{Milnor hypersurface}.
	Observe that \(H_{0,n} = \Pp^{n-1}\).
\end{para}

\begin{para}
	\label{p:c_H_m_n}
	Let \(0\leq m \leq n\) with \(n\ge 1\).
	Then using the notation of (\ref{def:lc_rc}), we have (see e.g.\ \cite[(5.1.2),(5.1.5)]{ciequ})
	\[
		c_{(m+n-1)}(H_{m,n}) =
		\begin{cases}
			-n & \text{ if \(m=0\)},\\
			\displaystyle{\binom{m+n}{m}} & \text{ if \(m \geq 2\)}. \\
		\end{cases}
	\]
\end{para}

\begin{proposition}
	\label{prop:action_on_Hnm}
	Let \(0\leq m \leq n\) with \(n\ge 1\).
	The \(k\)-variety \(H_{m,n}\) admits a \(G\)-action with a fixed locus of dimension \(d \leq \lfloor (m+n-1)/q \rfloor\). More precisely
	\[
		d=
		\begin{cases}
			\lfloor (m+n-1)/q \rfloor & \text{if \(m\) and \(n\) are both divisible by \(q\)}, \\
			\lfloor m/q \rfloor + \lfloor n/q \rfloor & \text{otherwise},
		\end{cases}
	\]
\end{proposition}
\begin{proof}
	The rough idea is to distribute the characters of \(G\) as evenly as possible in each of the projective spaces \(\Pp^m,\Pp^n\), in order to control the dimension of the fixed locus.

	Write \(m+1=qa +r\), with \(a,r \in \Nn\) and \(1 \le r \le q\).
	Pick a subset \(R \subset Q\) such that \(|R|=r\) and \(0 \in R\).
	Consider the \(G\)-equivariant vector bundle over \(\Speck\) (we use the notation of (\ref{p:L_g}))
	\[
		M = \Big(\bigoplus_{g \in R} (\character{g})^{\oplus a+1}\Big) \oplus \Big(\bigoplus_{g \in Q \smallsetminus R} (\character{g})^{\oplus a}\Big),
	\]
	so that, using the notation of (\ref{p:character})
	\begin{equation}
		\label{eq:M_g}
		\rank(\wgt{M}{g}) =
		\begin{cases}
			a+1 & \text{ if \(g \in R\)},\\
			a & \text{ if \(g \not \in R\)}.
		\end{cases}
	\end{equation}
	The scheme underlying \(P=\Pp(M)\) is \(\Pp^m\).
	Writing \(P_g = \Pp(\wgt{M}{g})\), we have
	\begin{equation}
		\label{eq:P^G}
		P^G = \coprod_{ g \in Q} P_g.
	\end{equation}

	Next, write \(n+1=qb +s\) with \(b,s \in \Nn\) and \(1 \le s \leq q\).
	Pick a subset \(S \subset Q\) such that \(|S|=s\) and \(0\in S\).
	In case \(r \le s\), we additionally arrange that \(R \subset S\).
	Consider the \(G\)-equivariant vector bundle over \(\Speck\)
	\[
		N = \Big(\bigoplus_{g \in S} (\character{g})^{\oplus b+1}\Big) \oplus \Big(\bigoplus_{g \in Q \smallsetminus S} (\character{g})^{\oplus b}\Big),
	\]
	so that
	\begin{equation}
		\label{eq:N_g}
		\rank(\wgt{N}{g}) =
		\begin{cases}
			b+1 & \text{ if \(g \in S\)}, \\
			b & \text{ if \(g \not \in S\)}.
		\end{cases}
	\end{equation}
	We fix an inclusion \(M\subset N\) (our choice of \(S\) made this possible; note that if \(r>s\) then we must have \(a+1 \le b\)).
	We denote again by \(M,N\) the pull-backs of the \(G\)-equivariant vector bundles \(M,N\) to any \(k\)-variety with a \(G\)-action.

	Let \(U\) be the kernel of the composite epimorphism of \(G\)-equivariant vector bundles \(N^{\vee} \to M^{\vee} \to \Oc(1)\) over \(P\), and set \(H = \Pp_P(U)\). The scheme underlying \(H\) is \(H_{m,n}\).
	Then by \eqref{eq:P^G} we have
	\begin{equation}
		\label{eq:H_G}
		H^G = \big(\Pp_{P^G}(U|_{P^G})\big)^G = \coprod_{g \in Q} \big(\Pp_{P_g}(U|_{P_g})\big)^G= \coprod_{g,h\in Q} \Pp_{P_g}(\wgt{U|_{P_g}}{h}).
	\end{equation}
	Let \(g \in Q\).
	As \(G\) acts via the character \(-g\) on the restriction of \(\Oc(1)\) to \(P_g\), we have, for every \(h\in Q\)
	\begin{equation}
		\label{eq:rank_Q+U}
		\rank (\wgt{U|_{P_g}}{h}) = \begin{cases}
			\rank (\wgt{N}{h}) & \text{if \(h\neq -g\)},\\
			\rank (\wgt{N}{h})-1 & \text{if \(h=-g\)}.
		\end{cases}
	\end{equation}

	Let us first consider the case when \(r >1\) or \(s>1\).
	Then by \eqref{eq:rank_Q+U} and \eqref{eq:N_g}, for any \(g,h \in Q\), we have
	\[
		\rank (\wgt{U|_{P_g}}{h}) \leq \rank (\wgt{N}{h}) \leq b+1,
	\]
	with equality if \(h \neq -g\) and \(h \in S\).
	On the other hand (\ref{eq:M_g}) implies that \(\dim P_g \leq a\), with equality if \(g \in R\).
	Thus it follows from \eqref{eq:H_G} that \(\dim H^G \leq a+b\).
	As \(r>1\) or \(s>1\), we can find \(g \in R\) and \(h \in S\) such that \(h \neq -g\), so that
	\[
		\dim H^G = a+b = \lfloor m/q \rfloor + \lfloor n/q \rfloor,
	\]
	which concludes the proof in this case.\\

	Now assume instead that \(r=1\) and \(s=1\) (i.e.\ that \(m\) and \(n\) are divisible by \(q\)).
	Then \(R=S=\{0\}\).
	We have \(\dim P_0 = a\) by (\ref{eq:M_g}).
	By \eqref{eq:rank_Q+U} and \eqref{eq:N_g} we have
	\[
		\rank (\wgt{U|_{P_0}}{0}) = \rank (\wgt{N}{0}) -1 = b.
	\]
	For \(h \in Q \smallsetminus \{0\}\), we have by \eqref{eq:rank_Q+U} and \eqref{eq:N_g}
	\[
		\rank (\wgt{U|_{P_0}}{h}) = \rank (\wgt{N}{h}) = b.
	\]
	Thus \(\dim \Pp_{P_0}(\wgt{U|_{P_0}}{h}) = a+b-1\) for any \(h \in Q\).

	On the other hand, for \(g \in Q \smallsetminus \{0\}\) we have \(\dim P_g = a-1\) by (\ref{eq:M_g}), and for any \(h \in Q\) we have, by \eqref{eq:rank_Q+U} and \eqref{eq:N_g},
	\[
		\rank (\wgt{U|_{P_g}}{h}) \leq \rank(\wgt{N}{h}) \leq b+1.
	\]
	Therefore \(\dim \Pp_{P_g}(\wgt{U|_{P_g}}{h}) \le a+b-1\) for any \(g,h \in Q\) with \(g\ne 0\).
	Thus, in view of \eqref{eq:H_G}, we conclude that
	\[
		\dim H^G = a+b-1 = \Big\lfloor \frac{qa+qb-1}{q} \Big\rfloor = \Big\lfloor \frac{m+n-1}{q} \Big\rfloor.\qedhere
	\]
\end{proof}

\begin{proposition}
	\label{prop:actions}
	There exist, for \(i \in \Nn\smallsetminus \{0\}\), smooth projective \(k\)-varieties \(X_i^+,X_i^-\) of pure dimension \(i\) equipped with a \(G\)-action, and such that:
	\begin{enumerate}[label=(\roman*), ref=\roman*]
		\item
			\label{prop:actions:1}
			the elements \(\lc X_i^+ \rc - \lc X_i^- \rc\) are polynomial generators of the ring \(\Laz\),
		\item
			\label{prop:actions:2}
			\(\dim (X_i^+)^G \leq \lfloor i/q \rfloor\),
		\item
			\label{prop:actions:3}
			\(\dim (X_i^-)^G \leq \lfloor i/q \rfloor\).
	\end{enumerate}
\end{proposition}
\begin{proof}
	Let \(i \in \Nn\smallsetminus \{0\}\).
	An elementary computation shows that the integer of \eqref{eq:ci_li} is equal to
	\[
		\gcd_{1 \leq j\leq \lfloor (i+1)/2 \rfloor} \binom{i+1}{j}.
	\]
	The formula of (\ref{p:c_H_m_n}) shows that this integer is the greatest common divisor, and in particular a \(\Zz\)-linear combination, of the integers \(c_{(i)}(H_{m,n})\) for \(0\leq m \leq n\) and \(m+n-1=i\) with \(m\ne 1\).
	By (\ref{p:pol_gen_laz}) it follows that the ring \(\Laz\) is polynomial in elements \(\ell_i \in \Laz^{-i}\), for \(i\in \Nn\smallsetminus \{0\}\), which are \(\Zz\)-linear combinations of the classes \(\lc H_{m,n} \rc\) for \(0 \leq m \leq n\) and \(m+n-1=i\).
	Separating the negative and positive coefficients, we may write \(\ell_i = \lc X_i^+\rc - \lc X_i^- \rc\), where \(X_i^+,X_i^-\) are disjoint unions of varieties \(H_{m,n}\), for some \(m,n\) such that \(0 \leq m \leq n\) and \(m+n-1=i\). This proves \eqref{prop:actions:1}. The other statements follow from (\ref{prop:action_on_Hnm}).
\end{proof}

\subsection{The \texorpdfstring{\(q\)}{q}-filtration}

In this section we fix an integer \(q \in \Nn \smallsetminus \{0\}\).
We introduce a filtration of the Lazard ring which will play an important role.

\begin{definition}
	\label{def:F_q_L}
	Let \(R\) be a graded ring such that \(R^i=0\) for \(i>0\).
	We consider the filtration
	\[
		F_q^0 R\subset \cdots \subset F_q^iR \subset F_q^{i+1}R\subset \cdots \subset R
	\]
	generated by the conditions, for every \(i,j \in \Nn\)
	\begin{enumerate}[label=(\roman*), ref=\roman*]
		\item
			\(F_q^i R\) is a subgroup of \(R\),
		\item
			\label{def:F_q_L:1}
			for every \(j\) we have \(R^{-j} \subset F_q^{\lfloor j/q \rfloor} R\),
		\item \(F^i_q R \cdot F^j_q R \subset F^{i+j}_q R\).
	\end{enumerate}
\end{definition}

\begin{para}
	If \(R \to R'\) is a surjective graded ring morphism, then \(F_q^dR \to F_q^dR'\) is surjective for every \(d\).
\end{para}

\begin{para}
	\label{p:F_loc}
	If \(S\subset R\) is a multiplicatively closed subset consisting of homogeneous elements of degree \(0\), then the morphism \(S^{-1}(F^d_qR) \to F^d_q(S^{-1}R)\) is bijective for every \(d\).
\end{para}

\begin{para}
	\label{p:F0_L_q-1}
	Note that \(F^0_qR\) is the subring of \(R\) generated by the homogeneous elements of degree \(>-q\).
\end{para}

\begin{para}
	\label{p:grading_pol}
	Let \(A\) be a ring, and let \(N\subset \Nn\smallsetminus \{0\}\) be a subset.
	Consider the polynomial ring \(A[T_i, i \in N]\).
	We let the variable \(T_i\) be homogeneous of degree \(i\) (resp.\ \(\lfloor i/q \rfloor\)), and denote by \(\deg P\) (resp.\ \(\deg_q P\)) the degree of a polynomial \(P \) with respect to that grading.
\end{para}

\begin{para}
	\label{p:deg_q_deg}
	From the formula \(\lfloor a/q \rfloor + \lfloor b/q \rfloor \le \lfloor (a+b)/q \rfloor\)
	we deduce that \(\deg_q P \le \lfloor(\deg P)/q\rfloor\), for any polynomial \(P\) as in (\ref{p:grading_pol}).
\end{para}

\begin{proposition}
	\label{prop:L_q_grading}
	Let \(n,d\in \Nn\), and let \(\ell_i \in \Laz^{-i}\) be polynomial generators of \(\Laz\).
	Then letting the image of \(\ell_i\) in \(\su{\Laz}{n}\) be homogeneous of degree \(\lfloor i/q \rfloor\), for \(i\in \su{N}{n}\), induces a (new) grading of the ring \(\su{\Laz}{n}\), with respect to which the subgroup generated by the homogeneous elements of degree \(\le d\) is \(F_q^d\su{\Laz}{n}\).

	In other words, setting \(A=\Zz\) if \(n=0\), and \(A=\Fp\) if \(n>0\), we have
	\[
		F_q^d\su{\Laz}{n} = \{ P(\ell_1,\ldots) \text{ where \(P \in A[T_i, i \in \su{N}{n}]\) and \(\deg_q P \leq d\)}\}.
	\]
\end{proposition}
\begin{proof}
	Recall from (\ref{lemm:pol_gen_L}) and (\ref{p:Laz_La}) that the \(A\)-algebra \(\su{\Laz}{n}\) is polynomial in the variables \(\ell_i\) for \(i\in \su{N}{n}\).
	Since \(\ell_i \in F^{\lfloor i/q \rfloor}_q\su{\Laz}{n}\) for every \(i \in \su{N}{n}\), and since \(F^{\bullet}_q\su{\Laz}{n}\) is a ring filtration, it follows that any polynomial \(P \in A[T_i, i \in \su{N}{n}]\) verifies \(P(\ell_1,\ldots) \in F^{\deg_qP}_q\su{\Laz}{n}\).

	Conversely the group \(F^d_q\su{\Laz}{n}\) is generated by the products \(y_1 \cdots y_k\) with \(y_j \in \su{\Laz}{n}^{-d_j}\) for each \(j\), and such that \(\lfloor d_1/q \rfloor + \cdots + \lfloor d_k/q \rfloor \le d\).
Each \(y_j\) is of the form \(Q_j(\ell_1,\ldots)\), where \(Q_j \in A[T_i,i\in \su{N}{n}\}]\) is a polynomial such that \(\deg Q_j=d_j\).
Thus \(y_1 \cdots y_k = Q(\ell_1,\ldots)\), where \(Q=Q_1\cdots Q_k\) verifies, by (\ref{p:deg_q_deg}),
\[
	\deg_q Q = \deg_q Q_1 +\cdots +\deg_q Q_k \le \lfloor d_1/q \rfloor + \cdots + \lfloor d_k/q \rfloor \le d.\qedhere
\]
\end{proof}

\begin{remark}
	The grading of \(\su{\Laz}{n}\) in (\ref{prop:L_q_grading}) depends on the choice of the generators \(\ell_i\) (when \(q>1\)), while the filtration \(F^\bullet_q\su{\Laz}{n}\) does not.
\end{remark}

\begin{proposition}
	\label{prop:fixed_Fd}
	Let \(G\) be an algebraic group over \(k\), and \(q\) an integer verifying \(1 \leq q \leq |\widehat{G}|\).
	Then the group \(F^d_q\Laz\) is generated by a set of classes of smooth projective \(k\)-varieties \(X\) admitting a \(G\)-action such that \(\dim X^G \le d\).
\end{proposition}
\begin{proof}
	Let us consider the varieties \(X_i^+,X_i^-\) of (\ref{prop:actions}).
	By (\ref{prop:actions}.\ref{prop:actions:1}), we may take \(\ell_i = \lc X_i^+ \rc - \lc X_i^- \rc\) in (\ref{prop:L_q_grading}).
	We deduce that the group \(F^d_q\Laz\) is generated by the classes \(\lc X \rc\), where \(X=Z_1 \times \cdots \times Z_m\) with \(i_1, \ldots, i_m \in \Nn\smallsetminus \{0\}\) such that \(\lfloor i_1/q \rfloor + \cdots + \lfloor i_m/q \rfloor\le d\) and \(Z_j \in \{X_{i_j}^+,X_{i_j}^-\}\) for each \(j\in \{1, \ldots, m\}\).
	Now it follows from (\ref{prop:actions}.\ref{prop:actions:2}) and (\ref{prop:actions}.\ref{prop:actions:3}) that each \(Z_j\) admits a \(G\)-action such that \(\dim (Z_j)^G \le \lfloor i_j/q \rfloor\), and thus the variety \(X\) admits a \(G\)-action such that
	\[
		\dim X^G =\dim (Z_1)^G + \cdots + \dim (Z_m)^G \le \lfloor i_1/q \rfloor + \cdots + \lfloor i_m/q \rfloor\le d.\qedhere
	\]
\end{proof}

\section{Algebraic cobordism}
\label{sect:oct}

\subsection{Cobordism as a cohomology theory}
In this paper we will need to use a theory of algebraic cobordism.
When the base field \(k\) has characteristic zero, we could use the theory \(\Omega\) constructed by Levine and Morel in \cite{LM-Al-07}.
In order to allow for base fields of positive characteristic, we will use as a substitute the pure part of the cohomology theory represented by Voevodsky's motivic spectrum \(\MGL\) (introduced in \cite{Voe-A1}), with the characteristic exponent of \(k\) inverted.
In this section, we briefly review the facts that we will need.

\begin{para}
	As explained in \cite[2.9.7]{Panin-Oriented_II}, the spectrum \(\MGL\) induces a ``ring cohomology theory'', which admits a canonical orientation.
	This permits us to define the Chern classes \(c^{\MGL}_i(E) \in \MGL^{2i,i}(X)\) of a vector bundle \(E\to X\), when \(X\) is smooth \(k\)-variety, as well as push-forward maps \(f_*^{\MGL} \colon \MGL^{2*,*}(X) \to \MGL^{2*,*}(Y)\) for each morphism \(f\colon X \to Y\) between smooth \(k\)-varieties, where
	\[
		\MGL^{2*,*}(X) = \bigoplus_{n\in \Zz} \MGL^{2n,n}(X).
	\]
	The functor \(\MGL^{2*,*}\) satisfies the axioms of an oriented cohomology theory in the sense of \cite[Definition~1.1.2]{LM-Al-07} (which do not include the localization sequence).
\end{para}

\begin{para}
	\label{p:Omega}
	Let \(c\) be the characteristic exponent of \(k\).
	When \(X\) is a smooth \(k\)-variety, we define
	\[
		\Omega(X) = \MGL^{2*,*}(X)[c^{-1}].
	\]
	We will simply denote by \(c_i\) and \(f_*\) the localizations of \(c_i^{\MGL}\) and \(f_*^{\MGL}\).
\end{para}

\begin{proposition}	
	\label{prop:Chern_numbers}
	The morphism \(\Laz \to \Omega(\Speck)\) is injective,
	and maps the class \(\lc X \rc \in \Laz\) of (\ref{def:lc_rc}) to the class \(p_*(1) \in \Omega(\Speck)\), for every smooth projective \(k\)-variety \(X\) with structural morphism \(p\colon X\to \Speck\).
\end{proposition}
\begin{proof}
	Consider the Eilenberg--Mac Lane spectrum \(\Hb\) representing motivic cohomology with coefficients in \(\Zz[\bb]\).
	There is a unique morphism of ring spectra \(\varphi\colon\MGL \to \Hb\) mapping the first Chern class \(c_1^{\MGL}(L)\in \MGL^{2,1}(Y)\) of each line bundle \(L\) over a smooth \(k\)-variety \(Y\) to (the notation \(\exp\) is defined in \eqref{eq:exp})
	\[
		\exp(c_1^{\CH}(L))\in \CH^*(Y)[\bb]=\Hb^{2*,*}(Y).
	\]
	This yields an orientation of the ring cohomology theory induced by the spectrum \(\Hb\).
	Its formal group law is \eqref{eq:fglr_zb}, which implies that the morphism \eqref{eq:Laz_Zb} factors as
	\[
		\Laz \to \MGL^{2*,*}(\Speck) \xrightarrow{\varphi_*} \Hb^{2*,*}(\Speck) = \Zz[\bb].
	\]
	Since the morphism \eqref{eq:Laz_Zb} is injective, so is \(\Laz \to \MGL^{2*,*}(\Speck)\), which yields the stated injectivity.
	Thus it will suffice to verify that, for any smooth projective \(k\)-variety \(X\), the class \(p^{\Hb,\varphi}_*(1)\) is given in \(\Hb^{*,*}(\Speck) = \Zz[\bb]\) by the formula \eqref{eq:Chern_numbers},
	where \(p\colon X \to \Speck\) is the structural morphism, and \(p^{\Hb,\varphi}_* \colon \Hb^{2*,*}(X) \to \Hb^{2*,*}(\Speck)\) denotes the push-forward corresponding to the orientation \(\varphi\).

	To do so, consider the Eilenberg--Mac Lane spectrum \(\Hz\) representing motivic cohomology with coefficients in \(\Zz\), and the morphism of ring spectra \(\MGL \to \Hz\) mapping \(c_1^{\MGL}(L)\) to \(c_1^{\CH}(L)\), for every line bundle \(L\) over a smooth \(k\)-variety.
	The corresponding push-forward \(p_* \colon \Hz^{2*,*}(X) \to \Hz^{2*,*}(\Speck)\) is the degree map \(\deg \colon \CH(X) \to \Zz\).

	Next consider the composite \(\psi \colon \MGL\to \Hz \to \Hb\), where the second map is induced by the inclusion \(\Zz\subset \Zz[\bb]\), and denote by \(p^{\Hb,\psi}_* \colon \Hb^{2*,*}(X) \to \Hb^{2*,*}(\Speck)\) the push-forward corresponding to the orientation \(\psi\).
	It follows from \cite[Remark~2.16]{Panin-Oriented_II} (see also \cite[\S4]{Mer-Ori}) that
	\[
		p^{\Hb,\varphi}_*(1) = p^{\Hb,\psi}_*(P(-\Tan_X)),
	\]
	where \(P\) is the class defined in (\ref{p:CF}).
	We conclude using the formula \eqref{eq:P_CF}.
\end{proof}

\begin{para}
	\label{p:Laz_Omega_k}
	We write
	\[
		\La=\Laz[c^{-1}].
	\]
	It is proved in \cite[Proposition~8.2]{Hoyois-cobordism} that the canonical morphism \(\La \to \Omega(\Speck)\) is an isomorphism.
\end{para}

\begin{proposition}
	\label{prop:loc_seq}
	Let \(X\) be a smooth \(k\)-variety.
	Let \(Z\) be a closed subscheme of \(X\), and \(U\) its open complement.
	Assume that the \(k\)-variety \(Z\) is smooth.
	Denote by \(i\colon Z \to X\) and \(j\colon U \to X\) the immersions.
	Then we have an exact sequence of \(L\)-modules
	\[
		\Omega(Z) \xrightarrow{i_*} \Omega(X) \xrightarrow{j^*} \Omega(U) \to 0.
	\]
\end{proposition}
\begin{proof}
	The cofiber sequence of pointed motivic spaces
	\[
		\Sigma^{\infty}_+U \to \Sigma^{\infty}_+X \to \Sigma^{\infty}\Th(N),
	\]
	where \(\Th(N)\) refers to the Thom space of the normal bundle \(N\) to \(i\), together with the (nongraded) isomorphism \(\MGL^{2*,*}(\Th(N)) \simeq \MGL^{2*,*}(Z)\) induced by the orientation of \(\MGL\) yields an exact sequence
	\[
		\MGL^{2*,*}(Z) \xrightarrow{i_*} \MGL^{2*,*}(X) \xrightarrow{j^*} \MGL^{2*,*}(U) \to \MGL^{2*+1,*}(Z) .
	\]
	Here we have written, for \(Y\) a smooth projective \(k\)-variety,
	\[
		\MGL^{2*,*}(Y) = \bigoplus_{n \in \Zz}\MGL^{2n,n}(Y) \quad \text{ and } \quad \MGL^{2*+1,*}(Y) = \bigoplus_{n \in \Zz}\MGL^{2n+1,n}(Y).
	\]
	So to conclude it will suffice to verify that, for every \(n\in \Zz\)
	\begin{equation}
		\label{eq:mgl_2n+1_n}
		\MGL^{2n+1,n}(Z)[c^{-1}] =0.
	\end{equation}
	To do so we use the strongly convergent Hopkins--Hoyois--Morel spectral sequence \cite[(8.14)]{Hoyois-cobordism}, for fixed \(n\),
	\[
		L^q \otimes_{\Zz} H^{p-q,n-q}(Z,\Zz) \Rightarrow \MGL^{p+q,n}(Z)[c^{-1}]
	\]
	where \(H^{*,*}(-,\Zz)\) refers to motivic cohomology.
	For \(p+q=2n+1\) we have
	\[
		H^{p-q,n-q}(Z,\Zz) = H^{2(n-q)+1,n-q}(Z,\Zz) = 0,
	\]
	as follows from the identification with higher Chow groups \cite{Voevodsky-higher}.
	Thus the spectral sequence yields \eqref{eq:mgl_2n+1_n}, concluding the proof.
\end{proof}

\subsection{Cobordism of graded vector bundles}
\label{sect:Mcc}

In this section, we fix a diagonalizable group \(G\) over the field \(k\).

\begin{para}
	\label{p:[ag]}
	When \(R\) is a ring, we consider the polynomial \(R\)-algebra \(R[\ag]\) in the variables \(a_{i,g}\), for \(g\in \widehat{G}\smallsetminus \{0\}\) and \(i\in \Nn\).
\end{para}

\begin{para}
	\label{p:def_Q}
	For each \(G\)-equivariant vector bundle \(E \to X\), where \(X\) is a smooth \(k\)-variety with trivial \(G\)-action and \(E^G=0\), there exists a class \(\T(E) \in \Omega(X)[\ag]\) compatible with pull-backs, and such that:
	\begin{enumerate}[label=(\roman*), ref=\roman*]
		\item
			If \(0 \to E' \to E \to E'' \to 0\) is an exact sequence of \(G\)-equivariant vector bundles over \(X\) such that \(E^G=0\), then \(\T(E) =\T(E') \T(E'')\).

		\item
			\label{p:def_Q:2}
			If \(L \to X\) is a line bundle with trivial \(G\)-action and \(g \in \widehat{G}\smallsetminus \{0\}\), then
			\[
				\T(L \otimes \character{g}) = \sum_{i\in \Nn} c_1(L)^i a_{i,g}.
			\]
	\end{enumerate}
	This is proved with the splitting principle, similarly to (\ref{p:CF}), using additionally \eqref{eq:decomp_weights} to argue characterwise, and the fact that the first Chern classes of line bundles are nilpotent \cite[Lemma~1.1.3]{LM-Al-07}.
\end{para}

\begin{definition}
	\label{def:Mcc}
	Recall that \(c\) is the characteristic exponent of the field \(k\).
	We set
	\[
		\Mcc=\La[\ag]=(\Laz[c^{-1}])[\ag]=\Omega(\Speck)[\ag].
	\]
	When \(X\) is a smooth projective \(k\)-variety with trivial \(G\)-action, and \(E\to X\) a \(G\)-equivariant vector bundle such that \(E^G=0\), we define a class
	\[
		\lc E \to X \rc = p_*(\T(E)) \in \Mcc,
	\]
	where \(p \colon X \to \Speck\) is the structural morphism, and \(p_* \colon \Omega(X)[\ag] \to \La[\ag]=\Mcc\) is the map induced by push-forward morphism \(\Omega(X) \to \Omega(\Speck)=\La\) along \(p\).
\end{definition}

\begin{example}
	For every \(g\in \widehat{G}\smallsetminus \{0\}\) we have \(\lc \character{g} \to \Speck \rc = a_{0,g}\).
\end{example}

\begin{para}
	\label{p:grading_Mcc}
	The \(\La\)-algebra \(\Mcc=\La[\ag]\) is graded, by letting \(a_{i,g}\) be of degree \(-i\).
	Note that the class \(\lc E\to X\rc\) is homogeneous of degree \(-d\) when \(X\) has pure dimension \(d\).

	For \(d\in \Nn\), we denote by \(\Mcc(d)\) the subring of \(\Mcc\) generated by the homogeneous elements of degree \(\ge -d\).
\end{para}

\begin{para}
	\label{p:Mcc_1}
	If \(G=1\), then \(\Mcc=\La\), as graded \(\La\)-algebras.
\end{para}

\begin{para}
	\label{p:p_i_g}
	For \(g\in \widehat{G}\smallsetminus \{0\}\) and \(i \in \Nn\), we write
	\begin{equation}
		\label{eq:def_p_ig}
		p_{i,g} = \lc \Oc(1) \otimes \character{g} \to \Pp^i \rc \in \Mcc,
	\end{equation}
	where \(G\) acts trivially on \(\Oc(1)\).
\end{para}

\begin{proposition}
	\label{prop:M_is_poly}
	The \(\La\)-algebra \(\Mcc\) is polynomial in the elements \(p_{i,g}\) for \(i \in \Nn\) and \(g\in \widehat{G}\smallsetminus \{0\}\).
\end{proposition}
\begin{proof}
	For every \(n\in \Nn\), let \(q_n\colon \Pp^n \to \Speck\) be the structural morphism.
	Then we have, for every \(i\in \Nn\) and \(g\in \widehat{G}\smallsetminus \{0\}\)
	\[
		p_{i,g}= \sum_{j\in \Nn } (q_i)_*\big(c_1(\Oc(1))^j\big) a_{j,g}=\sum_{j=0}^i (q_{i-j})_*(1) a_{j,g}=\sum_{j=0}^i \lc \Pp^{i-j} \rc a_{j,g},
	\]
	in other words \(p_{i,g}=a_{i,g} + \lc \Pp^1 \rc a_{i-1,g}+ \cdots + \lc \Pp^i \rc a_{0,g}\).
	The implies that the elements \(p_{i,g}\) are polynomial generators of the \(\La\)-algebra \(\Mcc\).
\end{proof}

\begin{remark}
	One may deduce from (\ref{prop:M_is_poly}) that the \(\Zz[c^{-1}]\)-module \(\Mcc\) is generated by the elements \(\lc E\to X\rc\), where \(X\) is a smooth projective \(k\)-variety with trivial \(G\)-action, and \(E\to X\) a \(G\)-equivariant vector bundle such that \(E^G=0\).
\end{remark}

\begin{proposition}
	\label{prop:ell_pol}
	Let \(d\in \Nn\).
	Let \(\ell_i \in \Laz^{-i}\), for \(i\in \Nn\smallsetminus \{0\}\), be a family of polynomial generators of the \(\Zz\)-algebra \(\Laz\).
	Then the \(\Zz[c^{-1}]\)-algebra \(\Mcc(d)\) is polynomial in the elements
	\[
		\ell_1,\ldots, \ell_d,p_{0,g},\ldots , p_{d,g} \text{ for \(g\in \widehat{G}\smallsetminus \{0\}\)}.
	\]
\end{proposition}
\begin{proof}
	By (\ref{prop:M_is_poly}) the \(\Zz[c^{-1}]\)-algebra \(\Mcc\) is polynomial in the elements \(\ell_n\) for \(n \in \Nn\smallsetminus \{0\}\), and \(p_{i,g}\) for \(i \in \Nn\) and \(g\in \widehat{G}\smallsetminus \{0\}\).
	As \(\ell_n\), resp.\ \(p_{i,g}\), is homogeneous of degree \(-n\), resp.\ \(-i\), the statement follows.
\end{proof}

\section{Cobordism of diagonalizable groups}
\label{sect:omega_G}
In this section, we compute the \(G\)-equivariant cobordism ring of the point \(\Omega_G(\Speck)\), for \(G\) a diagonalizable group.
A similar computation was done in topology by Landweber in \cite{Landweber-coherence}, which inspired some of the arguments of this section.\\

We fix an element \(c \in \Zz\smallsetminus \{0\}\), and write \(\La=\Laz[c^{-1}]\).

\subsection{Algebraic preliminaries}
We begin by establishing the algebraic result (\ref{prop:limit_fglr}) that will be required for the determination of the ring \(\Omega_G(\Speck)\) in \S\ref{sect:omega_G:diag}.
It is the algebraic input used in the computation of \(\Omega_{G\times \mu_n}(\Speck)\) from \(\Omega_G(\Speck)\), allowing us to proceed inductively.

This result uses crucially the fact the Lazard ring is \emph{not} a noetherian ring.
It is the algebraic reason for using cobordism as a cohomology theory, as opposed to theories such as Chow theory or \(K\)-theory, which have noetherian coefficient rings.

\begin{lemma}[{\cite[Lemma 5]{Landweber-coherence}}]
	\label{lemm:fg_ideal_fgl}
	Let \(I\) be a finitely generated homogeneous ideal of \(\La\) such that \(I \ne \La\), and let \(m \in \Zz \smallsetminus \{0\} \).
	Then \(\fglr{m}(t) \neq 0\) in \(\La[[t]]/I\).
\end{lemma}
\begin{proof}
	The ideal \(I\cap \Zz[c^{-1}] \subset \Zz[c^{-1}]\) is proper, hence is contained in the ideal \(\ell \Zz[c^{-1}]\) for some prime number \(\ell\in \Nn\).
	Let \(s \in \Nn\) be such that \(I\) is generated by homogeneous elements of degrees \(\ge -s\).
	Then \(I\) is contained in the ideal \(J \subset \La\) generated by \(\ell\) and the homogeneous elements of degrees \(-1,\ldots ,-s\) in \(\La\).
	Note that the \(\mathbb{F}_{\ell}\)-algebra \(\La/J\) is polynomial by (\ref{p:pol_gen_laz}) (with one generator in each degree \(< -s\)), and therefore the ideal \(J\) is prime.
	Replacing \(I\) with \(J\), we may thus additionally assume that the ideal \(I\) is prime.

	Now if \(a,b \in \Nn\smallsetminus \{0\}\), then \(\fglr{ab}(t) = \fglr{a}(\fglr{b}(t))\) in \(\La[[t]]/I \simeq (\La/I)[[t]]\).
	Therefore if the lemma is valid for \(m=a\) and for \(m=b\), it is also valid for \(m=ab\) (here we use the fact that \(\La/I\) is a domain).
	Also, the lemma is true for \(m=-1\), since
	\[
		\fglr{-1}\big(\fglr{-1}(t)\big)=t \ne 0 \in (\La/I)[[t]].
	\]
	Thus, we are reduced to assuming that \(m=p\) is a prime number.
	As \(\fglr{p}(t) = 0\) in \((\La/I)[[t]]\), we have \(v_n \in I\) for every \(n \in \Nn\) (see (\ref{def:u_n})).
	In particular, when \(p^n >s\) the element \(v_n\) must be decomposable in \(\La\) (in the sense of (\ref{p:decomposable})).
	As \(p=v_0\) belongs to the proper ideal \(I\) of \(\La\), the number \(c\) is not divisible by \(p\).
	Thus \(\Laz \to \Laz/p\) factors through \(\La\).
	It follows that \(v_n\) is decomposable in \(\Laz/p\), which contradicts (\ref{lemm:u_indec}).
\end{proof}

\begin{para}
	Let us fix an integer \(n \in \Zz \smallsetminus \{0\} \).
	Let \(M\) be a finitely presented graded \(\La\)-module.
	For \(d \in \Nn\), we consider the graded \(\La[[t]]\)-module
	\[
		M_d = M \otimes_\La (\La[[t]]/t^d),
	\]
	as well as the graded \(\La\)-modules \(I_d\) and \(Q_{d+1}\) fitting into the exact sequence
	\begin{equation}
		\label{eq:I_M_M_Q}
		0 \to I_d \to M_d \xrightarrow{\cdot\fglr{n}(t)} M_{d+1} \to Q_{d+1} \to 0.
	\end{equation}
\end{para}

\begin{lemma}
	\label{lemm:lim_0_1}
	There exists \(m\in \Nn\) such that the natural map \(I_{d+m} \to I_d\) is zero for all \(d \in \Nn\).
	In particular,
	\[
		\lim_d I_d = 0 \quad \text{and} \quad {\lim_d}^1 I_d =0.
	\]
\end{lemma}
\begin{proof}
	As \(\fglr{n}(t)=nt\mod t^2\) in \(\La[[t]]\), and \(n\) is a nonzerodivisor in \(\La\), there exists for each \(d \in \Nn\smallsetminus \{0\}\) an \(\La\)-module \(T_d\) fitting into the exact sequence of \(\La\)-modules
	\[
		0 \to \La[[t]]/t^d \xrightarrow{\cdot\fglr{n}(t)} \La[[t]]/t^{d+1} \to T_d \to 0.
	\]
	Observe that \(I_d = \Tor_1^{\La}(M,T_d)\).

	Let us first consider the special case \(M= \La/\mathfrak{p}\), where \(\mathfrak{p}\) is a finitely generated homogeneous prime ideal of \(\La\).
	Then by (\ref{lemm:fg_ideal_fgl}) we have \(\fglr{n}(t) = at^m \mod t^{m+1}\) in \((\La/\mathfrak{p})[[t]]\) for some \(m \in \Nn \smallsetminus \{0\}\) and \(a\) nonzero in the domain \(\La/\pp\).
	This implies that \(I_d = t^{d+1-m} M_d\) for all \(d \in \Nn\).
	Then \(I_{d+m} \to I_d\) is zero, as required.

	Let us come back to the general case.
	For \(s\in \Nn\), write \(\La(s)=\Laz(s)[c^{-1}]\) (see (\ref{p:Laz_n_d})).
	We claim that there exists an integer \(s\in \Nn\) and a finitely generated graded \(L(s)\)-module \(N\) such that \(M=N\otimes_{\La(s)} \La\).
	Indeed, let us write \(M\) as the cokernel of a morphism of graded \(\La\)-modules \(\varphi \colon \La^{\oplus a} \to \La^{\oplus b}\), where \(a,b \in \Nn\).
	Such a morphism is given by a \((b \times a)\)-matrix with coefficients in \(\La\); this matrix has coefficients in the subring \(\La(s)\) for some \(s \in \Nn\).
	This gives a morphism of graded \(\La\)-modules \(\psi \colon L(s)^{\oplus a} \to L(s)^{\oplus b}\) such that \(\varphi = \psi \otimes_{\La(s)} \La\).
	Setting \(N = \coker \psi\) proves the claim.

	Now, as the ring \(\La(s)\) is noetherian, the graded module \(N\) admits a finite filtration with successive quotients of the form \(\La(s)/P\), where \(P\) is a homogeneous prime ideal of \(\La(s)\) (see e.g.\ \cite[IV, \S3, Proposition 2]{Bou-AC-1-4}).
	Since \(\La\) is a polynomial ring over \(\La(s)\), for \(P\) as above the ring \((\La(s)/P) \otimes_{\La(s)}\La \simeq \La/(P\La)\) is polynomial over \(\La(s)/P\), and in particular the ideal \(P\La\) of \(\La\) is prime.
	We thus obtain a finite filtration of \(M\) by finitely presented submodules, with successive quotients of the form \(\La/\mathfrak{p}\), where \(\mathfrak{p}\) is a finitely generated homogeneous prime ideal of \(\La\).

	Proceeding by induction on the length of the filtration, we may assume that we have an exact sequence of graded \(\La\)-modules
	\[
		0 \to M' \to M \to \La/\mathfrak{p} \to 0
	\]
	for some finitely generated homogeneous prime ideal \(\mathfrak{p}\) of \(\La\), and an integer \(r\) such that \(\Tor_1^{\La}(M',T_{d+r}) \to \Tor_1^{\La}(M',T_d)\) is zero for all \(d\).
	By the special case treated above, we may find an integer \(s\) such that \(\Tor_1^{\La}(\La/\mathfrak{p},T_{d+s}) \to \Tor_1^{\La}(\La/\mathfrak{p},T_d)\) is zero for all \(d\).
	In the commutative diagram with exact rows
	\[
		\begin{tikzcd}
			\Tor_1^{\La}(M',T_{d+r+s}) \ar[r] \ar[d] &
			\Tor_1^{\La}(M,T_{d+r+s}) \ar[r] \ar[d] &
			\Tor_1^{\La}(\La/\mathfrak{p},T_{d+r+s}) \ar[d, "0"] \\
			\Tor_1^{\La}(M',T_{d+r}) \ar[r] \ar[d, "0"'] &
			\Tor_1^{\La}(M,T_{d+r}) \ar[r] \ar[d] &
			\Tor_1^{\La}(\La/\mathfrak{p},T_{d+r}) \ar[d] \\
			\Tor_1^{\La}(M',T_d) \ar[r] &
			\Tor_1^{\La}(M,T_d) \ar[r] &
			\Tor_1^{\La}(\La/\mathfrak{p},T_d)
		\end{tikzcd}
	\]
	the middle vertical composite thus vanishes, so that we may take \(m=r+s\) to conclude the proof.
\end{proof}

\begin{para}
	The morphisms \(M_{d+1}/I_{d+1} \to M_d /I_d\) are surjective, so that \(\lim^1_d M_d/I_d=0\).
	Therefore the exact sequences \eqref{eq:I_M_M_Q}
	\[
		0 \to M_d/I_d \xrightarrow{\cdot\fglr{n}(t)} M_{d+1} \to Q_{d+1} \to 0
	\]
	yield an exact sequence
	\[
		0 \to \lim_d M_d/I_d \xrightarrow{\cdot\fglr{n}(t)} \lim_d M_d \to \lim_d Q_d \to 0
	\]
	Now, from the exact sequences
	\[
		0 \to I_d \to M_d \to M_d/I_d \to 0
	\]
	for \(d\in \Nn\), and from (\ref{lemm:lim_0_1}), we see that the morphism \(\lim_d M_d \to \lim_d M_d /I_d\) is bijective.
	We thus have an exact sequence
	\begin{equation}
		\label{eq:M_M_Q}
		0 \to \lim_d M_d \xrightarrow{\cdot\fglr{n}(t)} \lim_d M_d \to \lim_d Q_d \to 0.
	\end{equation}
\end{para}

\begin{proposition}
	\label{prop:limit_fglr}
	Let \(n \in \Zz\smallsetminus \{0\} \).
	Let \(R\) be a graded \(\La\)-algebra, finitely presented as a graded \(\La\)-module.
	Then we have an exact sequence of graded \(\La\)-modules
	\[
		0 \to R[[t]] \xrightarrow{\cdot \fglr{n}(t)} R[[t]] \to \lim_d \Big(R[[t]]/\big(t^d,\fglr{n}(t)\big)\Big) \to 0.
	\]
\end{proposition}
\begin{proof}
	Take \(M=R\) in \eqref{eq:M_M_Q}.
\end{proof}

\subsection{Equivariant cobordism}
\label{sect:omega_G:diag}
The purpose of this section is to compute the ring \(\Omega_G(\Speck)\) when \(G\) is a diagonalizable group of finite type.
In this section we work over a base field \(k\), and set \(\La=\Laz[c^{-1}]\), where \(c\) is the characteristic exponent of \(k\).

We first recall briefly the definition of the equivariant cobordism from \cite{Heller-Malagon-Lopez}, based on Borel's construction from \cite{Totaro-CHBG}.
We fix an affine algebraic group \(G\) over \(k\).

\begin{para}
	\label{p:good_system}
	A \emph{good system of representations} for \(G\) is a collection of triples \((V_m,U_m,W_m)\) indexed by \(m\in \Nn\) where
	\begin{itemize}
		\item[---] \(\cdots \subset V_m \subset V_{m+1} \subset \cdots\) is a chain of \(G\)-equivariant vector subbundles over \(k\),
		\item[---] \(U_m\) a \(G\)-invariant open subscheme of \(V_m\) for each \(m \in \Nn\),
		\item[---] \(W_m \subset V_m\) is a \(G\)-equivariant subbundle for each \(m \in \Nn\),
	\end{itemize}
	satisfying the following conditions:
	\begin{enumerate}[label=(\roman*), ref=\roman*]
		\item the \(G\)-action on \(U_m\) is free and \(U_m/G\) is a smooth \(k\)-variety,

		\item we have \(V_{m+1} = V_m \oplus W_m\),

		\item
			\label{def:good_system:3}
			the subscheme \(U_{m+1} \subset V_{m+1}\) contains \(U_m \times W_m\),

		\item the codimension of the complement of \(U_m\) in \(V_m\) is a strictly increasing function of \(m \in \Nn\).
	\end{enumerate}
\end{para}

\begin{para}
	It follows from \cite[Remark~1.4]{Totaro-CHBG} that such a system always exists.
\end{para}

\begin{definition}
	\label{def:Omega_G}
	Given a good system \((V_m,U_m,W_m)\) of representations for \(G\), the equivariant cobordism ring of a smooth \(k\)-variety \(X\) with a \(G\)-action is defined as
	\[
		\Omega_G(X) = \lim_m \Omega((X \times U_m)/G),
	\]
	where the limit is taken in the category of graded rings.
	As explained in \cite[\S3.1.1]{Heller-Malagon-Lopez} (the ``No-name Lemma'', which does not use the assumption that \(k\) has characteristic zero) this graded ring does not depend on the choice of the good system of representations for \(G\).

	This yields a functor \(\Omega_G\) from the category of smooth \(k\)-varieties with a \(G\)-action to the category of graded \(\La\)-modules.
	In addition, in the theory \(\Omega_G\) there are push-forwards along \(G\)-equivariant projective morphisms, and \(G\)-equivariant vector bundles admit Chern classes with values in that theory (see for instance \cite[\S4]{Heller-Malagon-Lopez}).
\end{definition}

\begin{definition}
	Let \(R\) be a graded \(\Laz\)-algebra, and \(H\) be an abelian group.
	We consider the graded power series algebra \(R[[H]]\) in the variables \(X_h\) for \(h \in H\) (see (\ref{p:graded_power_series})).
	We define a graded \(R\)-algebra \(\Sy{R}{H} = R[[H]]/I\), where \(I\) is the homogeneous ideal generated by the elements (we use the notation of (\ref{p:fgl}))
	\[
		X_{a+b} - (X_a\fgl X_b) \quad \text{for \(a,b \in H\)}.
	\]
\end{definition}

We will prove the following:
\begin{theorem}
	\label{th:Omega_diag}
	Let \(G\) be a diagonalizable group of finite type over \(\Speck\), with character group \(\widehat{G}\).
	Then mapping a character of \(G\) to the first Chern class of the associated \(G\)-equivariant line bundle over \(\Speck\) induces an isomorphism of graded \(\La\)-algebras
	\[
		\Sy{\La}{\widehat{G}} \simeq \Omega_G(\Speck).
	\]
\end{theorem}

Observe that, for \(n\in \Nn\) and abelian groups \(A,B\) we have
\[
	\Sy{R}{\Zz/n} = R[[t]]/\fglr{n}(t), \quad \Sy{R}{A \times B} = \Sy{\Sy{R}{A}}{B}.
\]
Therefore (\ref{th:Omega_diag}) can be restated as:
\begin{theorem}
	\label{cor:Omega_diag}
	Let \(q_1,\ldots,q_r \in \Nn\) and \(G = \mu_{q_1} \times \cdots \times \mu_{q_r}\), where we have written \(\mu_0 =\Gm\).
	Then mapping \(t_i\) to the first Chern class of the \(G\)-equivariant line bundle over \(\Speck\) associated with the character \(G \to \mu_{q_i} \subset \Gm\) yields an isomorphism of graded \(\La\)-algebras
	\[
		\La[[t_1,\ldots,t_r]]/(\fglr{q_1}(t_1),\ldots,\fglr{q_r}(t_r)) \simeq \Omega_G(\Speck).
	\]
\end{theorem}
\begin{proof}
	This follows by induction on \(r\) from (\ref{lemm:Omega_diag_induc}) below.
\end{proof}

\begin{lemma}
	\label{lemm:Omega_diag_induc}
	Let \(G'\) be an affine algebraic group over \(k\), and let \(q \in \Nn\).
	Let \(G=G' \times \mu_q\) and denote by \(g \colon G \to \mu_q \subset \Gm\) the induced character.
	Assume that there exists a good system of representations \((V'_m,U'_m,W'_m)\) for \(G'\) such that for each \(m \in \Nn\), the graded \(\La\)-module \(\Omega(U'_m/G)\) is finitely presented.
	Then
	\begin{enumerate}[label=(\roman*), ref=\roman*]
		\item
			\label{lemm:Omega_diag_induc:1}
			There exists a good system of representations \((V_m,U_m,W_m)\) for \(G\) such that for each \(m \in \Nn\), the graded \(\La\)-module \(\Omega(U_m/G)\) is finitely presented.

		\item
			\label{lemm:Omega_diag_induc:2}
			The mapping \(t \mapsto c_1(\character{g})\) induces an isomorphism of graded \(\La\)-algebras
			\[
				\Omega_{G'}(\Speck)[[t]]/\fglr{q}(t) \simeq \Omega_G(\Speck).
			\]
	\end{enumerate}
\end{lemma}
\begin{proof}
	Write \(C=\character{g}\).
	Set \(V_m = V'_m \oplus C^{\oplus m}\) and \(U_m = U'_m \times (C^{\oplus m} \smallsetminus 0)\).
	Let \(W_m = W'_m \oplus C \subset V_{m+1}\), where \(C \subset C^{\oplus (m+1)}\) is the inclusion of the last summand.
	Then \((V_m,U_m,W_m)\) is a good system of representations for \(G\).

	Write \(B_d= (C^{\oplus d}\smallsetminus 0)/\mu_q\) and \(U_{m,d} = U'_m \times (C^{\oplus d}\smallsetminus 0)\).
	Then
	\begin{equation}
		\label{eq:U_m_d}
		U_{m,d}/G=(U'_m/G') \times B_d.
	\end{equation}
	In addition \(U_{m,m}=U_m\) for every \(m\), and so
	\begin{equation}
		\label{eq:lim_lim}
		\lim_m \Omega(U_m/G) = \lim_m \lim_d \Omega(U_{m,d}/G).
	\end{equation}
	The variety \(B_d\) is isomorphic to the open complement of the zero section in the line bundle \(\Oc(q)\) over \(\Pp^{d-1}\) when \(q\ne 0\), and to \(\Pp^{d-1}\) when \(q=0\).
	From now on, we will assume that \(q\ne 0\), the case \(q=0\) being similar but substantially easier, and not used in the rest of the paper.

	The localization sequence (\ref{prop:loc_seq}) yields an exact sequence, taking \eqref{eq:U_m_d} into account,
	\[
		\Omega((U'_m/G') \times \Pp^{d-1}) \to \Omega((U'_m/G') \times \Oc(q)) \to \Omega(U_{m,d}/G) \to 0
	\]
	The middle term can be identified with \(\Omega((U'_m/G') \times \Pp^{d-1})\) by homotopy invariance, with the left arrow becoming multiplication by \(c_1(\Oc(q))=\fglr{q}(c_1(\Oc(1))\) under that identification.
	Recall that the projective bundle theorem \cite[Theorem~3.9]{Panin-Oriented_I} asserts that we have an isomorphism of \(\Omega(U'_m/G')\)-algebras
	\[
		\Omega(U'_m/G')[[t]]/t^d \xrightarrow{\sim} \Omega((U'_m/G') \times \Pp^{d-1}), \quad t \mapsto c_1(\Oc(1)).
	\]
	Now, the line bundle \(((C^{\otimes d} \smallsetminus 0) \times C)/\mu_q\) over \(B_d\) is the pull-back of \(\Oc(1)\) under the morphism \(B_d \to \Pp^{d-1}\).
	This yields an exact sequence, for each \(d\in \Nn\),
	\begin{equation}
		\label{eq:omega_m_d}
		\Omega(U'_m/G')[[t]]/t^d \xrightarrow{\cdot\fglr{q}(t)} \Omega(U'_m/G')[[t]]/t^d \to \Omega(U_{m,d}/G) \to 0
	\end{equation}
	where \(t\) is mapped to the first Chern class of the line bundle
	\[
		(U'_m/G') \times ((C^{\otimes d} \smallsetminus 0) \times C)/\mu_q= (U_{m,d} \times C)/G
	\]
	over \((U'_m/G')\times B_d=U_{m,d}/G\).
	Since the \(\La\)-module \(\Omega(U'_m/G')\) is finitely presented by assumption, so is \(\Omega(U'_m/G')[[t]]/t^d\), and the exact sequence above implies that the same is true for \(\Omega(U_{m,d}/G)\).
	In particular \(\Omega(U_m/G) = \Omega(U_{m,m}/G)\) is finitely presented, proving \eqref{lemm:Omega_diag_induc:1}.

	By (\ref{prop:limit_fglr}), the sequences (\ref{eq:omega_m_d}) yield an exact sequence of \(\La\)-modules
	\[
		0 \to \Omega(U'_m/G')[[t]] \xrightarrow{\cdot\fglr{q}(t)} \Omega(U'_m/G')[[t]] \to \lim_d \Omega(U_{m,d}/G) \to 0.
	\]
	Now, by homotopy invariance and the localization sequence (\ref{prop:loc_seq}), the condition \eqref{def:good_system:3} of (\ref{p:good_system}) implies that the morphisms \(\Omega(U'_{m+1}/G') \to \Omega(U'_m/G')\) are surjective, so that \(\displaystyle{{\lim_m}^1 \Omega(U'_m/G')[[t]]=0}\), hence we obtain an exact sequence
	\[
		0 \to \lim_m \big(\Omega(U'_m/G')[[t]]\big) \xrightarrow{\cdot\fglr{q}(t)} \lim_m \big(\Omega(U'_m/G')[[t]]\big) \to \lim_m \lim_d \Omega(U_{m,d}/G) \to 0,
	\]
	where \(t\) is mapped to the first Chern class of the \(G\)-equivariant line bundle \(C\) over \(\Speck\).
	By (\ref{eq:lim_lim}) and (\ref{lemm:ps_commute}) below, this yields the isomorphism \eqref{lemm:Omega_diag_induc:2}.
\end{proof}

\begin{lemma}
	\label{lemm:ps_commute}
	Let \(R_m\) be a direct system of graded \(\Laz\)-algebras.
	Then the morphism
	\[
		\big(\lim_m R_m\big)[[t]] \to \lim_m \big(R_m[[t]]\big)
	\]
	is bijective.
\end{lemma}
\begin{proof}
	When \(B\) is a ring, the \(B\)-module \(B[[t]]\) is isomorphic to a product of copies of \(B\) indexed by \(\Nn\).
	Thus the lemma follows from the fact that products and limits commute.
\end{proof}

\section{Algebraic independence}
\label{sect:alg_indep}
This purely algebraic section is motivated by the geometric computation of \S\ref{sect:omega_G} but is logically independent of it.
It will serve two purposes:
\begin{itemize}
	\item[---] describe which nonequivariant information is lost by performing the localization required by the concentration theorem.
	\item[---] give a method to produce algebraically independent elements in the equivariant cobordism ring of a diagonalizable \(p\)-group,
\end{itemize}

The inductive process is described in \S\ref{sect:formal_p}, which contains most of the technical arguments.
We use it in \S\ref{sect:H_j} to produce statements that will directly apply to the equivariant cobordism ring.

\subsection{Formal multiplication by a prime power}
\label{sect:formal_p}
Let us fix a prime number \(p\) and a power \(q>1\) of \(p\).
Recall that we defined elements \(v_n \in \Laz\) in (\ref{def:u_n}), and rings \(\su{\Laz}{n}\) in (\ref{p:Laz_n_d}).

\begin{lemma}
	\label{lemm:fglr_powers}
	For any \(n \in \Nn\), we have in \(\su{\Laz}{n}[[t]]\)
	\[
		\fglr{q}(t)= (v_n)^{\frac{q^n-1}{p^n-1}} t^{q^n} \mod t^{q^n+1}.
	\]
\end{lemma}
\begin{proof}
	Write \(v=v_n\), and \(q=p^a\) with \(a \in \Nn\smallsetminus \{0\} \).
	We proceed by induction on \(a\).
	The case \(a=1\) follows from (\ref{prop:coeff_I}.\ref{prop:coeff_I:u_m}), and implies that there exists \(y\in \su{\Laz}{n}[[t]]\) such that
	\begin{equation}
		\label{eq:fglr_u_y}
		\fglr{p}(t)=vt^{p^n} +t^{p^n+1}y \in \su{\Laz}{n}[[t]].
	\end{equation}
	If \(a \geq 2\), then using induction and \eqref{eq:fglr_u_y}, we obtain in \(\su{\Laz}{n}[[t]]\)
	\begin{align*}
		\fglr{p^a}(t) & = \fglr{p^{a-1}}(\fglr{p}(t)) & \\
									& = v^{\frac{p^{n(a-1)}-1}{p^n-1}} (vt^{p^n}+t^{p^n+1}y)^{p^{n(a-1)}} \mod (vt^{p^n} +t^{p^n+1}y)^{p^{n(a-1)}+1} & \\
									& = v^{\frac{p^{na}-1}{p^n-1}} t^{p^{na}} \mod t^{p^{na}+1}.\qedhere
	\end{align*}
\end{proof}

\begin{para}
	We let \(T\) be the multiplicative subset of \(\Laz[[t]]\) generated by the power series \(\fglr{m}(t)\), where \(m \in \Zz \smallsetminus q\Zz\).
\end{para}

\begin{para}
	\label{p:def_s}
	We consider the power series \(s= \fglr{q/p}(t) \in \Laz[[t]]\).
\end{para}

\begin{lemma}
	\label{lemm:invert_s}
	Let \(Q= \Laz[[t]]/\fglr{q}(t)\).
	Then the morphism \(Q[s^{-1}] \to T^{-1}Q\) is an isomorphism.
\end{lemma}
\begin{proof}
	Write \(q=p^a\) with \(a \in \Nn\smallsetminus \{0\} \).
	Let \(m \in \Zz\smallsetminus q\Zz\), and write \(m = p^br\) where \(b,r \in \Nn\), with \(r\) prime to \(p\).
	Note that \(b<a\).
	Pick \(u,v \in \Zz\) such that \(ur+vp=1\), and set \(w=up^{a-1-b}\).
	Then we have in \(Q\)
	\begin{align*}
		\fglr{mw}(t)
		& = \fglr{up^{a-1}r}(t) = \fglr{p^{a-1}-vq}(t) \\
		& = \big(\fglr{p^{a-1}}(t)\big) \fgl \big(\fglr{-v}(\fglr{q}(t))\big) \\
		& =\fglr{p^{a-1}}(t)= s.
	\end{align*}
	Since \(\fglr{mw}(t)=\fglr{w}(\fglr{m}(t))\) is divisible by \(\fglr{m}(t)\), it follows that \(\fglr{m}(t)\) is invertible in \(Q[s^{-1}]\), proving the statement.
\end{proof}

\begin{para}
	By (\ref{prop:coeff_I}.\ref{prop:coeff_I:u_m}) there exists a unique power series \(\pi_n \in \su{\Laz}{n}[[t]]\) such that \(\fglr{p}(t) = t^{p^n} \pi_n\), and in addition we have \(\pi_n = v_n \mod t\).
	Using the power series \(s\) of (\ref{p:def_s}), let us define the power series
	\[
		h_n=\pi_n(s) \in \su{\Laz}{n}[[t]].
	\]
	In other words, we have
	\begin{equation}
		\label{eq:s_h_q}
		s^{p^n}h_n=\fglr{q}(t) \in \su{\Laz}{n}[[t]].
	\end{equation}
	Let us record that
	\begin{equation}
		\label{eq:h_n_dom}
		h_n=v_n \mod t.
	\end{equation}
\end{para}

\begin{para}
	Let \(R\) be a graded \(\Laz\)-algebra.
	We assume that \(R\) is Landweber exact as an \(\Laz\)-module (see (\ref{def:Landweber_exact})).
	We set
	\[
		A=R[[t]]/\fglr{q}(t).
	\]
	Using the notation of (\ref{p:su_n}), we thus have graded \(\Laz\)-algebras \(\su{R}{n},\su{A}{n}\) for \(n\in \Nn\).
	In addition, we consider the graded \(\Laz\)-algebra, for each \(n\in \Nn\),
	\begin{equation}
		\label{eq:def_B}
		B_n=\su{R}{n}[[t]]/h_n.
	\end{equation}
\end{para}

\begin{lemma}
	\label{lemm:B_complete}
	For any \(n \in \Nn\) we have
	\[
		\bigcap_{k\in \Nn} t^k B_n=0.
	\]
\end{lemma}
\begin{proof}
	Let us first observe that \(h_n\) is a nonzerodivisor in \(\su{R}{n}[[t]]/t^k\) for every \(k \ge 1\).
	Indeed for \(k=1\) this follows from \eqref{eq:h_n_dom}, as \(v_n\) is a nonzerodivisor in \(\su{R}{n}\) by the assumption that \(R\) is Landweber exact.
	Assume that \(k>1\), and let \(a,b \in \su{R}{n}[[t]]\) be such that \(h_na=t^kb\).
	By induction on \(k\), we may write \(a = t^{k-1}a'\) with \(a' \in \su{R}{n}[[t]]\), and thus \(h_na'=tb\).
	By the case \(k=1\), we deduce that \(a' \in t\su{R}{n}[[t]]\), and finally \(a\in t^k\su{R}{n}[[t]]\).
	This proves the claim.

	Now let \(f\in \su{R}{n}[[t]]\) map to \(\bigcap_{k \in \Nn} t^k B_n \subset B_n\).
	Then for each \(k \geq 1\) we may find \(g_k\in \su{R}{n}[[t]]\) such that \(f=g_k h_n \mod t^k\).
	Thus \(g_{k+1} h_n =g_k h_n \mod t^k\) for all \(k \geq 1\), hence \(g_{k+1} =g_k \mod t^k\) by the observation above.
	Therefore there exists \(g \in \su{R}{n}[[t]]\) such that \(g = g_k \mod t^k\) for all \(k \geq 1\).
	Then \(f-gh_n \in \bigcap_{k \in \Nn} t^k(\su{R}{n}[[t]]) =0\), hence \(f\) maps to zero in \(B_n\).
\end{proof}

\begin{lemma}
	\label{lemm:B_un_nzd}
	For every \(n\in \Nn\), the element \(v_n\) is a nonzerodivisor in \(B_n\).
\end{lemma}
\begin{proof}
	Assume that \(v_nf = h_n g\) in \(\su{R}{n}[[t]]\), for some \(f,g \in \su{R}{n}[[t]]\).
	Then \(h_n g =0\) in \(\su{R}{n+1}[[t]]\).
	Multiplying by \(s^{p^n}\) we obtain, by \eqref{eq:s_h_q},
	\begin{equation}
		\label{eq:q_g_0}
		0 = \fglr{q}(t)\cdot g \in \su{R}{n+1}[[t]].
	\end{equation}
	As the leading coefficient of \(\fglr{q}(t)\) in \(\su{R}{n+1}[[t]]\) is a power of \(v_{n+1}\) by (\ref{lemm:fglr_powers}), it is a nonzerodivisor in \(\su{R}{n+1}\) (because \(R\) is Landweber exact), which implies that the power series \(\fglr{q}(t)\) is a nonzerodivisor in \(\su{R}{n+1}[[t]]\).
	By \eqref{eq:q_g_0} we deduce that \(g=0\) in \(\su{R}{n+1}[[t]]\).
	Thus \(g=v_ng'\) in \(\su{R}{n}[[t]]\) for some \(g'\in \su{R}{n}[[t]]\), and so \(v_nf = v_nh_n g'\) in \(\su{R}{n}[[t]]\).
	As \(v_n\) is a nonzerodivisor in \(\su{R}{n}\) (because \(R\) is Landweber exact), we deduce that \(f=h_ng'\) in \(\su{R}{n}[[t]]\), hence \(f\) has vanishing image in \(B_n\).
\end{proof}

\begin{lemma}
	\label{lemm:s_nzd_B}
	For every \(n\in \Nn \), the element \(s\) is a nonzerodivisor in \(B_n\).
\end{lemma}
\begin{proof}
	As \(h_n = v_n \mod s\) in \(\su{\Laz}{n}[[t]]\), the image of \(v_n\) in \(B_n\) is divisible by \(s\). Thus the statement follows from (\ref{lemm:B_un_nzd}).
\end{proof}

We recall that \(T\) denotes the multiplicative subset of \(\Laz[[t]]\) generated by the power series \(\fglr{m}(t)\), where \(m \in \Zz \smallsetminus q\Zz\).

\begin{proposition}
	\label{prop:kernel_loc_induct}
	\begin{enumerate}[label=(\roman*), ref=\roman*]
		\item \label{prop:kernel_loc_induct:1}
			The \(\Laz\)-module \(T^{-1}A\) is Landweber exact.

		\item \label{prop:kernel_loc_induct:2}
			For \(n \in \Nn\), the kernel of the morphism \(\su{A}{n} \to T^{-1}\su{A}{n}\) maps to zero under the composite \(\su{A}{n} \xrightarrow{t\mapsto 0}\su{R}{n}\to \su{R}{n+1}\).
		\item \label{prop:kernel_loc_induct:3}
			For \(n \in \Nn\), the image of the morphism \(\su{A}{n} \to T^{-1}\su{A}{n}\) is isomorphic to \(B_n\).
	\end{enumerate}
\end{proposition}
\begin{proof}
	Let \(n\in \Nn\).
	The natural morphism \(\su{A}{n} \to B_n\) is surjective, and its kernel is the ideal \(H\) generated by \(h_n\).
	Since \(s^{p^n}h_n=\fglr{q}(t)\) vanishes in \(\su{A}{n}\), the morphism \(\su{A}{n} \to B_n\) induces an isomorphism \(\su{A}{n}[s^{-1}] \simeq B_n[s^{-1}]\).
	As \(B_n \to B_n[s^{-1}]\) is injective by (\ref{lemm:s_nzd_B}), we deduce that the kernel of the morphism \(\su{A}{n} \to \su{A}{n}[s^{-1}]\) is \(H\), and that its image is isomorphic to \(B_n\).
	By (\ref{eq:h_n_dom}) the element \(h_n\) is mapped to \(v_n\) under the map \(\su{A}{n} \to \su{R}{n}\), hence has vanishing image in \(\su{R}{n+1}\).
	Thus \(H\) is mapped to zero in \(\su{R}{n+1}\).
	Since \(\su{A}{n}[s^{-1}] \to T^{-1}\su{A}{n}\) is an isomorphism by (\ref{lemm:invert_s}), we obtain \eqref{prop:kernel_loc_induct:2} and \eqref{prop:kernel_loc_induct:3}.
	We also deduce that \(T^{-1}\su{A}{n}=\su{(T^{-1}A)}{n}\) is isomorphic to \(B_n[s^{-1}]\), so that (\ref{lemm:B_un_nzd}) implies \eqref{prop:kernel_loc_induct:1}.
\end{proof}

\begin{lemma}
	\label{lemm:alg_indep_induct}
	Let \(n\in \Nn\), and let \(y_1,\ldots,y_m\) be elements of \(\su{A}{n}\) whose images under \(\su{A}{n} \xrightarrow{t\mapsto 0}\su{R}{n}\to \su{R}{n+1}\) are algebraically independent over \(\Fp\).
	\begin{enumerate}[label=(\roman*), ref=\roman*]
		\item \label{lemm:alg_indep_induct:1}
			If \(n \ne 0\) then \(y_1,\ldots,y_m,t^{-1} \in T^{-1}\su{A}{n}\) are algebraically independent over \(\Fp\).

		\item \label{lemm:alg_indep_induct:2}
			If \(n = 0\) then \(y_1,\ldots,y_m \in T^{-1}A\) are algebraically independent over \(\Zz\).
	\end{enumerate}
\end{lemma}
\begin{proof}
	\eqref{lemm:alg_indep_induct:1}: Assume that \(P\in \Fp[Y_1,\ldots,Y_m,T]\) is a nonzero polynomial such that \(P(y_1,\ldots,y_m,t^{-1}) =0\) in \(T^{-1}\su{A}{n}\).
	Write \(P= P_0 +\cdots +T^kP_k\), where \(P_i \in \Fp[Y_1,\ldots,Y_m]\) for each \(i=0,\ldots ,k\), and \(P_k \neq 0\).
	Then \(P_k(y_1,\ldots,y_m) + \cdots + t^k P_0(y_1,\ldots,y_m)\) belongs to the kernel of \(\su{A}{n} \to T^{-1}\su{A}{n}\), hence maps to zero in \(\su{R}{n+1}\) by (\ref{prop:kernel_loc_induct}.\ref{prop:kernel_loc_induct:2}).
	Thus \(P_k(y_1,\ldots,y_m)\) is mapped to zero in \(\su{R}{n+1}\).
	The assumption of algebraic independence then implies that \(P_k=0\), a contradiction.

	\eqref{lemm:alg_indep_induct:2}: Let \(P \in \Zz[Y_1,\ldots,Y_m]\) be a nonzero polynomial such that \(P(y_1,\ldots,y_m) =0\) in \(T^{-1}A\).
	By (\ref{prop:kernel_loc_induct}.\ref{prop:kernel_loc_induct:1}), the element \(p=v_0\) is a nonzerodivisor in \(T^{-1}A\).
	Thus we may assume that \(P\) is not divisible by \(p\) in \(\Zz[Y_1,\ldots,Y_m]\).
	Then \(P(y_1,\ldots,y_m)\) belongs to the kernel of \(A \to T^{-1}A\), hence maps to zero in \(\su{R}{1}\) by (\ref{prop:kernel_loc_induct}.\ref{prop:kernel_loc_induct:2}).
	The assumption of algebraic independence then implies that the image of \(P\) vanishes in \(\Fp[Y_1,\ldots,Y_m]\), a contradiction.
\end{proof}

\subsection{The equivariant cobordism ring}
\label{sect:H_j}
We assume given integers \(q_1,\ldots,q_r > 1\), which are powers of a common prime number \(p\).
We let \(\La=\Laz[c^{-1}]\), where \(c\in \Zz\smallsetminus \{0\}\).

\begin{para}
	\label{p:H_j}
	For every \(j\in \{0,\ldots,r\}\), we consider the graded \(\La\)-module
	\[
		H_j = \La[[t_1,\ldots,t_j]]/(\fglr{q_1}(t_1),\ldots,\fglr{q_j}(t_j)).
	\]
	Let \(S_j\) be the multiplicative subset of \(\Laz[[t_1,\ldots,t_j]]\) generated by the elements
	\[
		\fglr{m_1}(t_1) \fgl \cdots \fgl \fglr{m_j}(t_j)
	\]
	with \(m_i\in \Zz\), where at least one of the integers \(m_i\) is not divisible by \(q_i\).
\end{para}

\begin{para}
	\label{p:C_j_n}
	For \(j\in \{0,\ldots,r\}\), we set
	\[
		C_j = \im \big(\su{(H_j)}{r-j} \to S_j^{-1} \su{(H_j)}{r-j}\big).
	\]
\end{para}

\begin{para}
	\label{p:R_A_T}
	We will use the results of \S\ref{sect:formal_p} in this setting, so let us make the connection explicit.
	Write \(R_j=S_j^{-1}H_j\) for every \(j\in \{0,\ldots,r\}\). For \(j \ge 1\) we set
	\[
		A_j=S_{j-1}^{-1}H_j = R_{j-1}[[t_j]]/\fglr{q_j}(t_j).
	\]
	Let \(T_j \subset \Laz[[t_j]]\) be the multiplicative subset generated by the power series \(\fglr{m}(t_j)\) for \(m \in \Zz\smallsetminus q_j\Zz\).
	When \(j\ge 1\), we are thus in the situation of \S\ref{sect:formal_p}, with
	\[
		q=q_j, R=R_{j-1}, A=A_j, t=t_j, T=T_j,
	\]
	(except for the fact that \(R_{j-1}\) is Landweber exact, which will be proved shortly).
	By (\ref{prop:kernel_loc_induct}.\ref{prop:kernel_loc_induct:3}) we also have, using the notation of \eqref{eq:def_B}
	\[
		B_{r-j} = \im \big(S_{j-1}^{-1}\su{(H_j)}{r-j} \to S_j^{-1} \su{(H_j)}{r-j}\big),
	\]
	and therefore
	\begin{equation}
		\label{eq:C_in_B}
		C_j \subset S_{j-1}^{-1}C_j= B_{r-j} \subset S_j^{-1}C_j=S_j^{-1} \su{(H_j)}{r-j}.
	\end{equation}
\end{para}

\begin{lemma}
	\label{prop:Hj}
	\begin{enumerate}[label=(\roman*), ref=\roman*]
		\item \label{prop:Hj:1}
			For every \(j\in \{0,\ldots,r\}\), the \(\Laz\)-module \(S_j^{-1}H_j\) is Landweber exact.

		\item \label{prop:Hj:2}
			For every \(j\in \{1,\ldots,r\}\) and \(n\in \Nn\), the morphism \(H_j \to H_{j-1}\) given by \(t_j \mapsto 0\) descends to \(C_j \to C_{j-1}\).

		\item \label{prop:Hj:3}
			Let \(j\in \{1,\ldots,r\}\) and \(n \in \Nn \smallsetminus \{0\}\).
			Assume that \(y_1,\ldots,y_m \in S_{j-1}^{-1}\su{(H_j)}{n}\) map in \(S_{j-1}^{-1}\su{(H_{j-1})}{n+1}\) to algebraically independent elements over \(\Fp\).
			Then
			\[
				y_1,\ldots,y_m,t_j^{-1} \in S_j^{-1}\su{(H_j)}{n}
			\]
			are algebraically independent over \(\Fp\).

		\item \label{prop:Hj:4}
			Let \(j\in \{1,\ldots,r\}\). Assume that \(y_1,\ldots,y_m \in S_{j-1}^{-1}H_j\) map in \(S_{j-1}^{-1}\su{(H_{j-1})}{1}\) to algebraically independent elements over \(\Fp\).
			Then
			\[
				y_1,\ldots,y_m \in S_j^{-1}H_j
			\]
			are algebraically independent over \(\Zz\).
	\end{enumerate}
\end{lemma}
\begin{proof}
	With the notation of (\ref{p:R_A_T}), we claim that the morphism \(T_j^{-1}A_j \to R_j\) is bijective.
	To see this, let \(m_1,\ldots,m_j\in \Zz\) be such that \(m_i \not \in q_i \Zz\) for some \(i\in \{1,\ldots,j\}\), and consider the elements \(a=\fglr{m_1}(t_1) \fgl \cdots \fgl \fglr{m_{j-1}}(t_{j-1})\) and \(b=\fglr{m_j}(t_j)\) in \(R_{j-1}[[t_j]]\). To establish the claim, we prove that \(a\fgl b\) is invertible in \(T_j^{-1}A_j\).

	Assume first that there exists \(i\in \{1,\ldots,j-1\}\) such that \(m_i\) is not divisible by \(q_i\).
	There exists a homogeneous power series \(f \in R_{j-1}[[t]]\) such that
	\[
		a \fgl b = a+b f(b) \in R_{j-1}[[t_j]].
	\]
	Since \(b\) is divisible by \(t_j\) and \(a\) is invertible in \(R_{j-1}\), it follows that \(a\fgl b\) invertible in the ring \(R_{j-1}[[t_j]]\), hence so is its image in the quotient \(A_j\).

	If now instead \(m_i\) is divisible by \(q_i\) for each \(i\in \{1,\ldots,j-1\}\), then \(a=0\) in \(R_{j-1}\), and so \(a \fgl b=b\).
	As \(m_j\) is not divisible by \(q_j\), the element \(b\) is invertible in \(T_j^{-1}A_j\).
	We have proved the claim.

	\eqref{prop:Hj:1}:
	We proceed by induction on \(j\).
	The case \(j=0\) follows from (\ref{p:Laz_exact}), since \(S_0^{-1}H_0=L\).
	When \(j\ge 1\), we have just seen that \(R_j \simeq T_j^{-1}A_j\).
	Thus \eqref{prop:Hj:1} follows from the induction hypothesis and (\ref{prop:kernel_loc_induct}.\ref{prop:kernel_loc_induct:1}).

	\eqref{prop:Hj:2}:
	It follows from (\ref{prop:kernel_loc_induct}.\ref{prop:kernel_loc_induct:2}) and (\ref{prop:kernel_loc_induct}.\ref{prop:kernel_loc_induct:3}) that the composite morphism
	\[
		\su{(A_j)}{r-j} \xrightarrow{t_j\mapsto 0} \su{(R_{j-1})}{r-j} \to \su{(R_{j-1})}{r+1-j}
	\]
	factors through the subring
	\[
		B_{r-j} = \im \big(\su{(A_j)}{r-j} \to T_j^{-1}\su{(A_j)}{r-j} =\su{(R_j)}{r-j}\big).
	\]
	So we have a commutative diagram
	\[
		\begin{tikzcd}
			H_j \ar[r] \ar[d] &\su{(H_j)}{r-j} \ar[r] \ar[d] &\su{(A_j)}{r-j} \ar[r] \ar[d] & B_{r-j} \ar[d] \ar[r, hook]& \su{(R_{j})}{r-j}\\
			H_{j-1} \ar[r] & \su{(H_{j-1})}{r-j} \ar[r] & \su{(R_{j-1})}{r-j} \ar[r] &\su{(R_{j-1})}{r+1-j}
		\end{tikzcd}
	\]
	Thus the image of the upper horizontal composite, namely \(C_j\), gets mapped into the image of the lower horizontal composite, namely \(C_{j-1}\).

	\eqref{prop:Hj:3} and \eqref{prop:Hj:4}: These follow from (\ref{lemm:alg_indep_induct}).
\end{proof}

\begin{para}
	Let us consider the ring morphism
	\begin{equation}
		\varepsilon \colon H_r \to H_0=\La, \quad t_i \mapsto 0 \quad \text{ for all \(i=1,\ldots,r\)}.
	\end{equation}
\end{para}

\begin{proposition}
	\label{prop:H_ker}
	The kernel of \(H_r \to S_r^{-1}H_r\) maps to \(\Ipn{r}L\) under \(\varepsilon\).
\end{proposition}
\begin{proof}
	It follows by induction on \(j\in \{0,\ldots,r\}\) from (\ref{prop:Hj}.\ref{prop:Hj:2}) that the composite morphism \(H_j \to \La \to \su{\La}{r}\) factors through \(C_j\).
	For \(j=r\) this is precisely the proposition.
\end{proof}

\section{Diagonalizable \texorpdfstring{\(p\)}{p}-groups}
\label{sect:fixed_locus}
This section establishes the main results of the paper. We begin in \S\ref{sect:fpf_actions} by describing the cobordism classes of varieties admitting fixed-point-free actions, which is achieved by combining the results obtained so far with the concentration theorem.

Next we prove in \S\ref{sect:Mcc_e} the ``injectivity theorem'' (\ref{th:inj_M}), which complements the concentration theorem in describing the relation between the equivariant cobordism class of a variety and its fixed data.

In \S\ref{sect:lin_indep}, we show how to lift relations of linear dependence from the nonequivariant cobordism ring to the equivariant one.
This method is then used in \S\ref{sect:filtering} to prove the main theorem on the dimension of the fixed locus.

\begin{para}
	\label{p:p_rank}
	Let \(p\) be a prime number.
	A \emph{finite diagonalizable \(p\)-group} over \(k\) is a diagonalizable group \(G\) over \(k\) such that \(|\widehat{G}|\) is a (finite) power of \(p\).
	The \emph{rank} of \(G\) is the minimal number of generators of the abelian group \(\widehat{G}\).
	So \(G\) is a finite diagonalizable group of rank \(r\) if and only if it is isomorphic to \(\mu_{q_1} \times \cdots \times \mu_{q_r}\), where each \(q_i\) is a nontrivial power of \(p\).
\end{para}
Let us now fix the setting for this section.
Let \(p\) be prime number different from the characteristic of the field \(k\).
We fix a finite diagonalizable \(p\)-group \(G\) over \(k\), set \(q=|\widehat{G}|\), and let \(r\) be the rank of \(G\).

\begin{para}
	\label{p:S}
	We consider the multiplicative subset \(S \subset \Omega_G(\Speck)\) generated by the classes \(c_1(\character{g})\) for \(g \in \widehat{G} \smallsetminus \{0\}\).
\end{para}

\begin{para}
	\label{p:G_H}
	We use the notation of \S\ref{sect:H_j} with \(\La=\Laz[c^{-1}]=\Omega(\Speck)\), where \(c\) is the exponent characteristic of \(k\).
	Let us fix \(r\) characters \(g_1, \ldots, g_r \in \widehat{G}\) inducing an isomorphism \(G \simeq \mu_{q_1} \times \cdots \times \mu_{q_r}\), where each \(q_i\) is a nontrivial power of \(p\).
	We have seen in (\ref{cor:Omega_diag}) that \(t_i \mapsto c_1(\character{g_i})\) induces an isomorphism
	\[
		H_r = \La[[t_1,\ldots,t_j]]/(\fglr{q_1}(t_1),\ldots,\fglr{q_j}(t_j)) \simeq \Omega_G(\Speck)
	\]
	Under this isomorphism the subset \(S_r \subset H_r\) is mapped to \(S \subset \Omega_G(\Speck)\).
	We will often implicitly make that identification, and for instance write \(t_i \in \Omega_G(\Speck)\) to denote \(c_1(\character{g_i})\).
\end{para}

\subsection{Fixed-point-free actions}
\label{sect:fpf_actions}
\begin{para}
	\label{p:lc_equiv}
	When \(X\) is a smooth projective \(k\)-variety with a \(G\)-action, we consider the elements in \(\Omega_G(\Speck)\)
	\[
		\lc X \rc_G = p^G_*(1) \quad \text{and} \quad \lc u \rc_G = p^G_*(u) \text{ for \(u\in \Omega_G(X)\)},
	\]
	where \(p^G_* \colon \Omega_G(X) \to \Omega_G(\Speck)\) is the equivariant push-forward map along the structural morphism \(p\colon X\to \Speck\).
	We use the same notation for their images in \(S^{-1}\Omega_G(\Speck)\).
\end{para}

\begin{para}
	\label{p:lc_rc_G}
	Let \(X\) be a smooth projective \(k\)-variety with a \(G\)-action.
	It follows from (\ref{prop:Chern_numbers}) that the image of the class \(\lc X \rc_G\) under the natural morphism \(\Omega_G(\Speck) \to \Omega(\Speck)\) coincides with the image of the class \(\lc X \rc \in \Laz\) of (\ref{def:lc_rc}) under the morphism \(\Laz \to \Omega(\Speck)\).
\end{para}

\begin{theorem}
	\label{th:empty_ipnr}
	Let \(X\) be a smooth projective \(k\)-variety with a \(G\)-action such that \(X^G=\varnothing\).
	Then the class \(\lc X \rc\in \Laz\) (see (\ref{def:lc_rc})) belongs to \(\Ipn{r}\).
\end{theorem}
\begin{proof}
	We may assume that \(G\) is nontrivial, i.e.\ that \(r\ge 1\).
	The concentration theorem \cite[(4.5.4)]{conc} asserts that
	\[
		S^{-1} \Omega_G(X) \simeq S^{-1} \Omega_G(X^G).
	\]
	Since \(\Omega_G(X^G)=\Omega_G(\varnothing)=0\), we deduce that \(S^{-1} \Omega_G(X)=0\).
	This implies that \(\lc X \rc_G\) vanishes in \(S^{-1}\Omega_G(\Speck)\).
	We conclude using (\ref{prop:H_ker}), considering (\ref{p:Laz_La}) and (\ref{p:lc_rc_G}).
\end{proof}

\begin{remark}
	\label{rem:fpt}
	If \(X\) possesses a Chern number prime to \(p\), then \(\lc X \rc \not \in \Ipn{r}\) by (\ref{rem:Ipn_Chern_numbers}).
	Therefore any \(G\)-action on \(X\) must fix a point;
	we thus recover the fixed-point theorem \cite[(1.1.2)]{fpt}.
\end{remark}

\begin{remark}
	We have seen that the varieties \(Y_s\) of (\ref{prop:no_fixed_on_gen}) admit a fixed-point-free \(G\)-action, for \(s=0,\ldots,r-1\).
	On the other hand, they do not admit such action for \(s \ge r\).
	Indeed \(\Ipn{n} \ne \Ipn{n+1}\) for every \(n\in \Nn\) by (\ref{prop:coeff_I}), and (\ref{prop:no_fixed_on_gen}) implies that \(\Ipn{n+1}\) is generated by \(\lc Y_{n+1} \rc\) and \(\Ipn{n}\).
	Therefore \(\lc Y_s \rc \not \in \Ipn{r}\) for all \(s \ge r\), and so the claim follows from (\ref{th:empty_ipnr}).

	Note that every Chern number of \(Y_s\) is divisible by \(p\), so that (\ref{rem:fpt}) does not apply to \(X=Y_s\).
\end{remark}

\subsection{The injectivity theorem}
\label{sect:Mcc_e}
\begin{para}
	\label{p:Euler}
	Let \(X\) be a smooth \(k\)-variety with a \(G\)-action, and \(p\colon E\to X\) a \(G\)-equivariant vector bundle.
	Denote by \(s\colon X \to E\) its zero section.
	The \emph{Euler class} of \(E\) is defined as
	\[
		e(E) = (p^*)^{-1} \circ s_*(1) \in \Omega_G(X).
	\]
	Note that \(e(E)=c_1(E)\) when \(E\) is a line bundle.
\end{para}

\begin{para}
	\label{p:euler_mult}
	If \(0\to E_1 \to E_2 \to E_3 \to 0\) is an exact sequence of \(G\)-equivariant vector bundles over \(X\) then \(e(E_2)=e(E_1) e(E_3)\).
	This follows from the nonequivariant version of this statement, which is itself a consequence of the transversality axiom \cite[(2.1.3.v)]{inv} (see e.g.\ the proof of \cite[Proposition~53.1]{EKM}).
\end{para}

\begin{para}
	\label{p:Euler_inverse}
	Let \(E\to X\) be a \(G\)-equivariant vector bundle, where \(X\) is a smooth \(k\)-variety with trivial \(G\)-action.
	If \(E^G=0\), it is proved in \cite[(4.5.6)]{conc} that the class \(e(E)\) becomes invertible in \(S^{-1}\Omega_G(X)\).
	We denote by \(e(-E)\) its inverse.
\end{para}

\begin{proposition}
	\label{prop:euler_morphism}
	There is a morphism of \(\La\)-algebras
	\[
		\Mcc \to S^{-1}\Omega_G(\Speck), \quad \quad \lc E \to X \rc \mapsto \lc e(-E) \rc_G.
	\]
\end{proposition}
\begin{proof}
	The power series \(x \fgl y\in \Laz[[x,y]]=(\Laz[[x]])[[y]]\) has constant term \(x\), when viewed as a power series in the variable \(y\) with coefficients in \(\Laz[[x]]\), and thus becomes invertible in the ring \((\Laz[[x]][x^{-1}])[[y]]\).
	Write
	\[
		(x \fgl y)^{-1}= \sum_{i\in \Nn} f_i(x) y^i, \quad \text{with \(f_i(x)\in \Laz[[x]][x^{-1}]\)}.
	\]
	For every smooth \(k\)-variety \(X\) with trivial \(G\)-action, consider the morphism of \(\Omega(X)\)-algebras
	\[
		\eta \colon \Omega(X)[\ag] \to S^{-1}\Omega_G(X), \quad
		a_{i,g} \mapsto f_i(c_1(\character{g})) \quad \text{ (for \(i \in \Nn\) and \(g \in \widehat{G}\smallsetminus \{0\}\))}.
	\]
	Then \(\eta\) commutes with pull-backs and push-forwards along morphisms between smooth \(k\)-varieties with trivial \(G\)-action.
	For every vector bundle \(E \to X\), where \(X\) is a smooth \(k\)-variety with trivial \(G\)-action, and \(E^G=0\), we have in \(S^{-1}\Omega_G(X)\)
	\begin{equation}
		\label{eq:Q_E}
		\eta(\T(E)) = e(-E).
	\end{equation}
	Indeed this is clear when \(E\) is a line bundle on which \(G\) acts by a single character (using the formula (\ref{p:def_Q}.\ref{p:def_Q:2})), and in general follows by multiplicativity of the Euler class (\ref{p:euler_mult}), in view of the construction of the class \(\T(E)\) given in (\ref{p:def_Q}).
	When the \(k\)-variety \(X\) is additionally projective, pushing forward the formula \eqref{eq:Q_E} along the morphism \(X\to \Speck\), we obtain in \(S^{-1}\Omega_G(\Speck)\)
	\[
		\eta(\lc E\to X \rc) = \lc e(-E) \rc_G.\qedhere
	\]
\end{proof}

\begin{para}
	\label{p:concentration}
	Let \(X\) be a smooth projective \(k\)-variety with a \(G\)-action.
	Let \(N\to X^G\) be the normal bundle to the closed immersion of the fixed locus \(X^G\to X\).
	Recall from (\ref{prop:fixed_smooth}) that \(N^G=0\), and so we have a class \(\lc N \to X^G\rc \in \Mcc\) (see \S\ref{sect:Mcc}).
	The concentration theorem \cite[(4.5.7)]{conc} asserts that, under the map of (\ref{prop:euler_morphism}), the element \(\lc N \to X^G\rc \in \Mcc\) is sent to the class \(\lc X \rc_G\) in \(S^{-1}\Omega_G(\Speck)\).
\end{para}

\begin{para}
	\label{p:x_i}
	Let us pick smooth projective \(k\)-varieties \(X_i^+,X_i^-\) with a \(G\)-action, for \(i \in \Nn \smallsetminus \{0\}\), verifying the conditions of (\ref{prop:actions}).
	We denote by \(x_i\) the classes \(\lc X_i^+ \rc - \lc X_i^-\rc\) in \(S^{-1}\Omega_G(\Speck)\).
	We also denote by \(n_i\in \Mcc\) the classes \(\lc N_i^+ \to (X_i^+)^G \rc - \lc N_i^- \to (X_i^-)^G \rc\), where \(N_i^+\), resp.\ \(N_i^-\), is the normal bundle to the immersion of the fixed locus \((X_i^+)^G\to X_i^+\), resp.\ \((X_i^-)^G \to (X_i)^-\).
\end{para}

\begin{para}	
	\label{p:n_i_x_i}
	As noted in (\ref{p:concentration}), the element \(n_i\) is mapped to \(x_i\) under the map of (\ref{prop:euler_morphism}).
\end{para}

\begin{para}
	\label{p:Mcc_ge_q}
	Recall from (\ref{p:grading_Mcc}) that the \(\Mcc(0)\)-algebra \(\Mcc\) is graded.
	For \(d \in \Nn\), we denote by \(\Mcd{d}\) the subgroup of \(\Mcc\) generated by the homogeneous elements of degrees \(\ge -d\).
\end{para}

\begin{para}
	\label{p:Mcc_dim}
	Let \(E \to X\) be a \(G\)-equivariant vector bundle, where \(X\) is a smooth projective \(k\)-variety with trivial \(G\)-action, and \(E^G=0\).
	If \(\dim X \leq d\), then \(\lc E \to X \rc \in \Mcd{d}\).
\end{para}

\begin{para}
	\label{p:y_Mcc_d}
	\label{p:deg_q}
	It follows from (\ref{prop:actions}) and (\ref{p:Mcc_dim}) that \(n_i \in \Mcd{\lfloor i/q \rfloor}\) for every \(i\in \Nn\smallsetminus \{0\}\).
	Thus for any nonzero \(P \in (\Zz[c^{-1}])[Y_i, i \in \Nn\smallsetminus \{0\}]\), we have \(P(n_1,\ldots) \in \Mcd{\deg_q P}\) (see (\ref{p:grading_pol})).
\end{para}

\begin{lemma}
	\label{lemm:alg_indep_S_Omega}
	For \(j\in \{0,\ldots,r-1\}\), the elements
	\[
		x_1,\ldots,x_m,t_1^{-1},\ldots,t_j^{-1} \in S_j^{-1}\su{(H_j)}{r-j}
	\]
	are algebraically independent over \(\Fp\).
\end{lemma}
\begin{proof}
	The images in \(\su{L}{r}=S_0^{-1}\su{(H_0)}{r}\) of the elements \(x_i\) for \(i \in N_{(r)}\) are the classes \(\lc X_i^+\rc - \lc X_i^- \rc\), and so are algebraically independent over \(\Fp\), by (\ref{prop:actions}.\ref{prop:actions:1}) and (\ref{lemm:pol_gen_L}).
	This proves the case \(j=0\) of the lemma.
	The case \(0<j<r\) follows by induction using (\ref{prop:Hj}.\ref{prop:Hj:3}).
\end{proof}

\begin{lemma}
	\label{lemm:alg_indep_S_Omega_r}
	The elements
	\[
		x_1,\ldots,x_m,t_1^{-1},\ldots,t_{r-1}^{-1} \in S_r^{-1}H_r
	\]
	are algebraically independent over \(\Zz\).
\end{lemma}
\begin{proof}
	When \(r=0\) this follows from (\ref{prop:actions}.\ref{prop:actions:1}).
	When \(r>0\), this follows from (\ref{lemm:alg_indep_S_Omega}) and (\ref{prop:Hj}.\ref{prop:Hj:4}).
\end{proof}

\begin{para}
	\label{p:p_g_t_i}
	For every \(i=1,\ldots,r\), the element \(p_{0,g_i} \in \Mcc\) of \eqref{eq:def_p_ig} is mapped to \(t_i^{-1} \in S^{-1}\Omega_G(\Speck)\) under the morphism of (\ref{prop:euler_morphism}).
\end{para}

\begin{lemma}
	\label{lemm:alg_inj}
	Let \(d\in \Nn\).
	Let \(\Pc(d) \subset \Mcc\) be the subring generated by the elements \(n_i\) (see (\ref{p:x_i})) for \(i \in N_{(r)}(qd+q-1)\) (see (\ref{p:N})) and \(p_{0,g_1},\ldots ,p_{0,g_{r-1}}\).
	\begin{enumerate}[label=(\roman*), ref=\roman*]
		\item
			\label{lemm:alg_inj:1}
			We have \(\Pc(d) \subset \Mcc(d)\).

		\item
			\label{lemm:alg_inj:2}
			The ring \(\Mcc(d)\) is algebraic over \(\Pc(d)\).

		\item
			\label{lemm:alg_inj:3}
			The morphism \(\Pc(d) \to S^{-1}\Omega_G(\Speck)\) induced by (\ref{prop:euler_morphism}) is injective.
	\end{enumerate}
\end{lemma}
\begin{proof}
	We have \(p_{0,g_1},\ldots ,p_{0,g_{r-1}} \in \Mcc(0) \subset \Mcc(d)\).
	In addition, for \(i\in N_{(r)}(qd+q-1)\) we have \(\lfloor i/q \rfloor\le d\), and thus by (\ref{p:y_Mcc_d}) the element \(n_i\) belongs to \(\Mcd{d} \subset \Mcc(d)\), proving \eqref{lemm:alg_inj:1}.

	Note that, when \(r \ge 1\)
	\[
		|N_{(r)}(qd+q-1)|=(qd+q-1)-(r-1)=qd+q-r.
	\]
	It follows from (\ref{lemm:alg_indep_S_Omega_r}) (given (\ref{p:n_i_x_i}) and (\ref{p:p_g_t_i})) that the \(qd+q-1\) elements \(p_{0,g_1},\ldots ,p_{0,g_{r-1}}\) and \(n_i\) for \(i\in N_{(r)}(qd+q-1)\) have algebraically independent images in \(S^{-1}\Omega_G(\Speck)\) over \(\Zz\) (we include the case \(r=0\)).
	This implies \eqref{lemm:alg_inj:3}, and also that the \(\Zz\)-algebra \(\Pc(d)\) is polynomial in \(qd+q-1\) variables.
	So is the \(\Zz[c^{-1}]\)-algebra \(\Mcc(d)\) by (\ref{prop:ell_pol}).
	Therefore the fraction fields of the domains \(\Pc(d)\) and \(\Mcc(d)\) have the same transcendence degree over \(\Qq\), hence the latter is algebraic over the former, which yields \eqref{lemm:alg_inj:2}.
\end{proof}

\begin{theorem}
	\label{th:inj_M}
	The morphism \(\Mcc\to S^{-1}\Omega_G(\Speck)\) of (\ref{prop:euler_morphism}) is injective.
\end{theorem}
\begin{proof}
	Assume that a nonzero element \(b \in \Mcc\) maps to zero in \(S^{-1}\Omega(\Speck)\).
	Then \(b\in \Mcc(d)\) for some \(d \in \Nn\).
	By (\ref{lemm:alg_inj}.\ref{lemm:alg_inj:2}), we may find a nonzero polynomial \(Q\in \Pc(d)[X]\) such that \(Q(b)=0\).
	Since \(\Mcc(d)\) is a domain and \(b \ne 0\), we may assume that \(Q\) is not divisible by \(X\), i.e.\ that \(Q(0)\ne 0\) in \(\Pc(d)\).
	But \(Q(0)\in \Pc(d)\) is mapped in \(S^{-1}\Omega(\Speck)\) to \(Q(b)=0\) (as \(b\) is mapped to \(0\)).
	Hence by (\ref{lemm:alg_inj}.\ref{lemm:alg_inj:3}) we have \(Q(0)=0\) in \(\Pc(d)\), a contradiction.
\end{proof}

We will thus view the elements of \(\Mcc\) as elements of \(S^{-1}\Omega_G(\Speck)\), and so in particular will identify the elements \(x_i\) and \(n_i\), or the elements \(p_{0,g_j}\) and \(t_j^{-1}\).
\begin{remark}
	The proof of (\ref{th:inj_M}) follows the arguments used in topology by tom Dieck in \cite[\S2]{tomDieck-Characteristic-I}.
\end{remark}

\subsection{Linear independence}
\label{sect:lin_indep}
This section is the technical heart of the paper.
The aim is to prove (\ref{lemm:lin_indep:base}), which will be the central tool used to prove the main theorem in the next section.

\begin{para}
	Recall from (\ref{p:G_H}) that mapping \(t_i\) to \(c_1(\character{g_i}) \in \Omega_G(\Speck)\), for \(i=1,\ldots,r\), induces an isomorphism of graded \(L\)-algebras
	\begin{equation}
		\label{eq:H_Omega}
		S_r^{-1}H_r \simeq S^{-1}\Omega_G(\Speck).
	\end{equation}
	From (\ref{prop:Hj}.\ref{prop:Hj:2}) we deduce that, for \(j=1,\ldots,r\), the morphism \(H_j \to H_{j-1}\) given by \(t_j \mapsto 0\) induces a morphism \(C_j \to C_{j-1}\) (the subring \(C_j \subset S_j^{-1} \su{(H_j)}{r-j}\) is defined in (\ref{p:C_j_n})), whose kernel is the ideal generated by \(t_j\).
	We denote its localization at \(S_{j-1}\) by
	\[
		\rho_j \colon S_{j-1}^{-1}C_j \to S_{j-1}^{-1}C_{j-1}.
	\]
\end{para}

\begin{para}
	Let us construct inductively subrings, for \(j=r,\ldots,0\)
	\[
		M_j \subset S_j^{-1}C_j= S_j^{-1}\su{(H_j)}{r-j}.
	\]
	To start, let \(M_r\) correspond to \(\Mcc(0)\) under the isomorphism (\ref{eq:H_Omega}).
	Next, assuming \(M_j\) constructed for some \(j\in \{1,\ldots,r\}\), set
	\begin{equation}
		\label{eq:def_M_j}
		M_{j-1} = \rho_j\big(M_j \cap S_{j-1}^{-1}C_j\big) \subset S_{j-1}^{-1}C_{j-1}.
	\end{equation}
\end{para}

\begin{remark}
	We may think of \(M_j\) as the set of images of ``sufficiently integral'' elements of \(\Mcc(0)\), interpolating between \(M_r\simeq \Mcc(0) \subset S^{-1}\Omega_G(\Speck)\) and \(M_0 \subset \su{\Laz}{r}\).
\end{remark}

\begin{para}
	\label{p:M_C}
	Since \(\rho_j(C_j)\subset C_{j-1}\) for every \(j\in \{1,\ldots,r\}\), the morphism \(\rho_j\) induces a morphism
	\begin{equation}
		\label{eq:M_C}
		M_j \cap C_j \to M_{j-1} \cap C_{j-1}.
	\end{equation}
\end{para}

\begin{lemma}
	\label{lemm:t_in_M}
	For every \(j \in \{1,\ldots,r\}\) the elements \(t_1^{-1}, \ldots , t_j^{-1} \in S_j^{-1}C_j\) belong to \(M_j\).
\end{lemma}
\begin{proof}
	We proceed by descending induction on \(j\), the case \(j=r\) being clear from (\ref{p:p_g_t_i}).
	If \(j<r\) we know by induction that \(t_1^{-1}, \ldots , t_j^{-1} \in M_{j+1}\), and those elements belong to \(S_j^{-1}C_{j+1}\) as well.
	So by the defining formula \eqref{eq:def_M_j} their images under \(\rho_{j+1}\) belong to \(M_j\).
\end{proof}

\begin{lemma}
	\label{lemm:lin_indep_rec}
	Let \(j \in \{1,\ldots,r\}\).
	Let \(y_1,\dots,y_m \in C_j\) be such that the elements \(\rho_j(y_1),\dots,\rho_j(y_m) \in C_{j-1}\) are linearly independent over \(M_{j-1} \cap C_{j-1}\).
	Then \(y_1,\dots,y_m \in C_j\) are linearly independent over \(M_j \cap C_j\).
\end{lemma}
\begin{proof}
	Assume that \(a_1,\dots,a_m \in M_j \cap C_j\) are such that
	\[
		\sum_{i=1}^{m} a_i y_i=0 \in C_j.
	\]
	We prove that \(a_i=0\) for each \(i=1,\ldots,m\).
	To do so, we prove by induction on \(k \in \Nn\) that \(a_i \in t_j^k(M_j \cap C_j)\) for all \(i\).
	By (\ref{lemm:B_complete}) (recall from (\ref{eq:C_in_B}) that \(C_j \subset B_{r-j}\)), this will suffice to conclude the proof of the lemma.

	So let \(k\ge 0\), and assume that \(a_i = t_j^k b_i\) with \(b_i \in M_j \cap C_j\) for all \(i=1,\ldots,m\).
	Since \(C_j \subset S_j^{-1}\su{(H_j)}{r-j}\) and \(t_j\in S_j\), it follows that \(t_j^k\) is a nonzerodivisor in \(C_j\), and thus
	\[
		\sum_{i=1}^{m} b_i y_i=0\in C_j, \text{ and so } \sum_{i=1}^{m} \rho_j(b_i) \rho_j(y_i)=0\in C_{j-1}.
	\]
	Note that \(\rho_j(b_i) \in \rho_j(M_j \cap C_j) \subset M_{j-1} \cap C_{j-1}\) (see (\ref{p:M_C})), and thus the assumption of the lemma implies that \(\rho_j(b_i)=0\) for all \(i\).
	This means that \(b_i=t_jc_i\) with \(c_i \in C_j\) for all \(i=1,\ldots,m\).
	As \(t_j^{-1}M_j \subset M_j\) by (\ref{lemm:t_in_M}), we have \(c_i \in M_j \cap C_j\).
	This concludes the inductive proof.
\end{proof}

\begin{para}
	Recall from (\ref{lemm:alg_inj}.\ref{lemm:alg_inj:1}) that the element \(n_i=x_i\) of \(\Mcc \subset S^{-1}\Omega_G(\Speck)\) belongs to \(\Mcc(0)\) when \(i \in \su{N}{r}(q-1)\).
	It is thus mapped to \(M_r\) under the morphism \eqref{eq:H_Omega}.

	On the other hand, the element \(x_i\) belongs to \(\Omega_G(\Speck) \simeq S^{-1}\Omega_G(\Speck)\), hence is mapped to \(C_r=\im(H_r \to S_r^{-1}H_r)\) under the morphism \eqref{eq:H_Omega}.

	Using the maps \eqref{eq:M_C}, we may thus consider the image of \(x_i\) in \(M_j \cap C_j\) for every \(j=0,\ldots,r\) and \(i \in \su{N}{r}(q-1)\).
\end{para}

\begin{lemma}
	\label{lemm:M_alg_induc}
	For \(j\in \{0,\ldots,r-1\}\), the ring \(M_j\) is algebraic over the subring generated by the elements \(t_1^{-1},\ldots,t_j^{-1}\) and the images of the elements \(x_i\) for \(i \in \su{N}{r}(q-1)\).
\end{lemma}
\begin{proof}
	We denote again by \(x_i\in M_j\) the images of the elements \(x_i\).
	We proceed by descending induction on \(j\).
	Denote by \(D_j\) the subring of \(M_j\) generated by the elements \(t_1^{-1},\ldots,t_j^{-1}\) and \(x_i\) for \(i \in \su{N}{r}(q-1)\).
	Assume that \(M_j\) contains an element \(y\) which is not algebraic over \(D_j\).
	By (\ref{lemm:alg_indep_S_Omega}), we deduce that the elements \(y,t_1^{-1},\ldots,t_j^{-1}\) and \(x_i\) for \(i \in \su{N}{r}(q-1)\) are algebraically independent in \(M_j\) over \(\Fp\).
	Write \(y=\rho_{j+1}(x)\) with \(x \in M_{j+1}\), and let us seek a contradiction.

	Assume first that \(j=r-1\).
	Then by (\ref{prop:Hj}.\ref{prop:Hj:4}) the elements \(x,t_1^{-1},\ldots,t_{r-1}^{-1}\) and \(x_i\) for \(i \in \su{N}{r}(q-1)\) are algebraically independent in \(M_r\) over \(\Zz\).
	In particular \(x\in M_r \simeq \Mcc(0)\) is not algebraic over \(D_r \simeq \Pc(0)\), a contradiction with (\ref{lemm:alg_inj}).

	Next assume instead that \(j<r-1\).
	Then by (\ref{prop:Hj}.\ref{prop:Hj:3}) the elements \(x,t_1^{-1},\ldots,t_{j+1}^{-1}\) and \(x_i\) for \(i \in \su{N}{r}(q-1)\) are algebraically independent in \(M_{j+1}\) over \(\Fp\).
	In particular \(x \in M_{j+1}\) is not algebraic over \(D_{j+1}\), a contradiction with the induction hypothesis.
\end{proof}

\begin{para}
	For \(d\in \Nn\), we denote by \(\su{\La}{r}(d)\) the subring of \(\su{\La}{r}\) generated by the homogeneous elements of degree \(\ge -d\).
\end{para}

\begin{proposition}
	\label{lemm:M_0}
	We have \(M_0=\su{\La}{r}(q-1)\).
\end{proposition}
\begin{proof}
	When \(r=0\), we have \(q=1\), and \(\Mcc(0)=\La(0)=\Zz[c^{-1}]\) by (\ref{p:Mcc_1}), which implies that \(M_0 = \Zz[c^{-1}]\), as required.
	So we assume that \(r>0\), so that \(\su{\Laz}{r}=\su{\La}{r}\) by (\ref{p:Laz_La}).
	By (\ref{lemm:M_alg_induc}) the subring \(M_0 \subset \su{\La}{r}\) is algebraic over the \(\Fp\)-subalgebra generated by the images in \(\su{\La}{r}\) of the elements \(x_i\) for \(i\in \su{N}{r}(q-1)\), namely the subring \(\su{\La}{r}(q-1)\) (by (\ref{prop:actions}) and (\ref{lemm:pol_gen_L})).
	But \(\su{\La}{r}(q-1)\) is algebraically closed in \(\su{\La}{r}\) (the latter being a polynomial ring over the former by (\ref{lemm:pol_gen_L})), whence the statement.
\end{proof}

\begin{para}
	\label{p:C}
	We denote by \(C\) the image of \(\Omega_G(\Speck) \to S^{-1}\Omega_G(\Speck)\).
\end{para}

\begin{lemma}
	\label{lemm:p_power_in_B}
	The cokernel of the inclusion \(\Mcc(0) \cap C\subset \Mcc(0)\) is of \(p\)-primary torsion.
\end{lemma}
\begin{proof}
	Let \(g\in \widehat{G} \smallsetminus \{0\}\), and let \(x=c_1(\character{g}) \in \Omega_G(\Speck)\).
	Then for \(m\) large enough, we have \(p^m g=0\) in \(\widehat{G}\), and so
	\[
		\fglr{p^m}(x) =0 \in \Omega_G(\Speck).
	\]
	Recall that the leading coefficient of the power series \(\fglr{p^m}(t) \in \Laz[[t]]\) is \(p^m\). Since \(x\in S\), we deduce that there exists \(a \in C\) such that
	\[
		p^m= xa \in C \subset S^{-1}\Omega_G(\Speck).
	\]
	As \(p_{0,g}=x^{-1} \in \Mcc(0)\), we conclude that \(p^m p_{0,g}=a \in \Mcc(0) \cap C\).
	Since the ring \(\Mcc(0)\) is generated by the elements \(p_{0,g}\) for \(g\in \widehat{G} \smallsetminus \{0\}\), the lemma follows.
\end{proof}

\begin{proposition}
	\label{lemm:lin_indep:base}
	Let \(y_1,\dots,y_m \in \Omega_G(\Speck)\) be elements whose images in \(\su{\La}{r}\) are linearly independent over \(\su{\La}{r}(q-1)\).
	Then \(y_1,\dots,y_m \in S^{-1}\Omega_G(\Speck)\) are linearly independent over \(\Mcc(0)\).
\end{proposition}
\begin{proof}
	Note that \(C\) corresponds to \(C_r\) under the isomorphism \eqref{eq:H_Omega}, and that \(M_0=M_0\cap C_0\).
	Given (\ref{lemm:M_0}), it follows inductively from (\ref{lemm:lin_indep_rec}) that \(y_1,\dots,y_m \in S^{-1}\Omega_G(\Speck)\) are linearly independent over \(\Mcc(0)\cap C\).
	Since \(\Mcc(0)\) has no nonzero \(p\)-primary torsion (see (\ref{prop:ell_pol})), we deduce from (\ref{lemm:p_power_in_B}) that those elements are linearly independent over \(\Mcc(0)\).
\end{proof}

\subsection{Filtering by fixed locus dimension}
\label{sect:filtering}

It follows from (\ref{prop:H_ker}) that the forgetful morphism \(\Omega_G(\Speck) \to \Omega(\Speck)=\La\) descends to a morphism (see (\ref{p:C}))
\[
	\epsilon\colon C \to \su{\La}{r}.
\]

Let us consider the filtration \(F_q^{\bullet}\su{\La}{n}\) defined in (\ref{def:F_q_L}), and the subgroup \(\Mcd{d} \subset \Mcc\) defined in (\ref{p:Mcc_ge_q}).
\begin{proposition}
	\label{prop:main}
	Let \(d \in \Nn\). The morphism \(\epsilon \colon C\to \su{\La}{r}\) maps \(\Mcd{d} \cap C\) into \(F_q^d\su{\La}{r}\).
\end{proposition}
\begin{proof}
	First recall that \(F_q^0\su{\La}{r}\) is a subring of \(\su{\La}{r}\), indeed, as noted in (\ref{p:F0_L_q-1}), we have \(F_q^0\su{\La}{r}=\su{\La}{r}(q-1)\).

	Let us fix a family of polynomial generators \(\ell_i\in \Laz^{-i}\) of \(\Laz\) (which is possible by (\ref{p:pol_gen_laz})).
	It follows from (\ref{prop:ell_pol}) that the graded \(\Mcc(0)\)-algebra \(\Mcc\) is polynomial, with \(q\) generators in each negative degree:
	namely in degree \(-e<0\) those are \(p_{e,g}\) for \(g \in \widehat{G}\smallsetminus \{0\}\), together with \(\ell_e\).

	By (\ref{prop:L_q_grading}), (\ref{p:F_loc}), and (\ref{lemm:pol_gen_L}), the \(F_q^0\su{\La}{r}\)-algebra \(\su{\La}{r}\) is also polynomial, with the same number \(q\) of generators in each positive degree with respect to the grading of (\ref{prop:L_q_grading}):
	namely in degree \(e>0\) those are \(\ell_{qe},\dots,\ell_{qe+q-1}\); note that \(\{qe,\ldots ,qe+q-1\} \subset \su{N}{r}\) because \(qe\ge q\ge p^r> p^{r-1}-1\).
	In particular the \(\Mcc(0)\)-module \(\Mcd{d}\) and the \(F_q^0\su{\La}{r}\)-module \(F_q^d\su{\La}{r}\) are both free of the same rank \(r(d)\).

	Let \(y_1,\dots,y_{r(d)} \in F_q^d\su{\La}{r}\) be a linearly independent family over \(F_q^0\su{\La}{r}\).
	For each \(i \in \{1,\ldots,r(d)\}\), we may then write \(y_i=\epsilon(z_i)\) with \(z_i \in \Mcd{d} \cap C\):
	indeed by (\ref{prop:L_q_grading}), we have \(y_i=P_i(\ell_1,\dots)\) for some polynomial \(P_i \in (\Zz[c^{-1}])[Y_1,\dots]\) with \(\deg_q P_i\le d\); then by \eqref{p:deg_q} we may take \(z_i = P_i(x_1,\dots)\).

	Now assume that \(x \in \Mcc \cap C\) is such that \(\epsilon(x) \not \in F_q^d\su{\La}{r}\).
	Recall that \(F_q^d\su{\La}{r}\) is the set of elements of degrees \(\le d\) with respect to a certain grading of the \(F_q^0\su{\La}{r}\)-algebra \(\su{\La}{r}\).
	Therefore the quotient \(\su{\La}{r}/F_q^d\su{\La}{r}\) has no \(F_q^0\su{\La}{r}\)-torsion, being a free module.
	It follows that the elements \(\epsilon(x),y_1,\dots,y_{r(d)}\) of \(\su{\La}{r}\) are linearly independent over \(F_q^0\su{\La}{r}=\su{\La}{r}(q-1)\).
	Thus by (\ref{lemm:lin_indep:base}), the \(r(d)+1\) elements \(x,z_1,\dots,z_{r(d)}\) are linearly independent over \(\Mcc(0)\).
	Those elements cannot all belong to \(\Mcd{d}\), which is free of rank \(r(d)\) as an \(\Mcc(0)\)-module (this can be seen for instance by tensoring with the function field of the domain \(\Mcc(0)\)).
	Therefore \(x \not \in \Mcd{d}\), concluding the proof.
\end{proof}

We can now prove the main theorem, which establishes one of the inclusions in Theorem~\ref{th:main} in the introduction.
We write \(F_q^{-\infty}\su{\Laz}{r}=0\).
\begin{theorem}
	\label{th:fixed_Fd}
	Let \(X\) be a smooth projective \(k\)-variety with a \(G\)-action, and let \(d=\dim X^G\).
	Then \(\lc X \rc \in F_q^d\su{\Laz}{r}\).
\end{theorem}
\begin{proof}
	Since the case \(X^G= \varnothing\) is covered by (\ref{th:empty_ipnr}), we assume that \(d\ne -\infty\).
	We may also assume that \(G\) is nontrivial, i.e.\ \(r \ge 1\).
	This implies that \(\su{\Laz}{r}=\su{\La}{r}\) by (\ref{p:Laz_La}).
	Then the class \(\lc X \rc_G\in S^{-1}\Omega_G(\Speck)\) belongs to the subgroup \(\Mcd{d}\cap C\), and therefore it image in \(\su{\Laz}{r}\) belongs to the subring \(F^d_q\su{\Laz}{r}\), by (\ref{prop:main}).
	That image is \(\lc X \rc\) by (\ref{p:lc_rc_G}).
\end{proof}

\section{Conclusion}
\label{sect:conclusion}
In this section \(G\) is an algebraic group of finite type over \(k\).

\subsection{Fixed-point-free actions}
\begin{definition}
	\label{def:JK}
	We denote by \(\JK(G) \subset \Laz\) the subgroup generated by the classes \(\lc X \rc\), where \(X\) runs over the smooth projective \(k\)-varieties admitting a \(G\)-action such that \(X^G=\varnothing\).
\end{definition}

\begin{lemma}
	\label{lemm:equidim}
	Every \(k\)-variety with a \(G\)-action is covered by equidimensional \(G\)-invariant open subschemes.
\end{lemma}
\begin{proof}
	Let \(X\) be a \(k\)-variety with a \(G\)-action, and \(U\) a connected component of \(X\).
	The image \(V\) of \(G \times U \to X\) is a \(G\)-invariant open subscheme of \(X\) containing \(U\).
	We prove that \(V\) is equidimensional.
	To do so, we may extend scalars, and thus assume that \(k\) is separably closed.
	Let \(G^\circ\) be the identity component of \(G\).
	Then \(G^\circ \times U\) is connected \cite[Proposition~1.3.4]{Milne-AG}, which implies that its image in \(X\) is \(U\); in other words the open subscheme \(U\) of \(X\) is \(G^\circ\)-invariant.
	Now, by our assumption on \(k\), the variety \(G\) is covered by the open subschemes \(\alpha G^\circ\), where \(\alpha\) runs over \(G(k)\).
	This implies that \(V\) is covered by the open subschemes \(\alpha U\), each of which is isomorphic to \(U\).
	Since \(U\) is equidimensional, so is \(V\).
\end{proof}

\begin{para}
	It follows from (\ref{lemm:equidim}) that the subgroup \(\JK(G)\) of \(\Laz\) is graded.
\end{para}

\begin{para}
	\label{p:J_ideal}
	Let \(X,Y\) be smooth projective \(k\)-varieties, and assume \(G\) acts without fixed point on \(X\).
	If we endow \(Y\) with the trivial \(G\)-action, the group \(G\) acts without fixed point on the product \(X \times Y\).
	This shows that \(\JK(G)\) is an ideal of \(\Laz\).
\end{para}

\begin{para}
	\label{p:J_quotient}
	If \(G \to H\) is a surjective morphism of algebraic groups over \(k\), then \(\JK(H) \subset \JK(G)\).
\end{para}

\begin{proposition}
	\label{prop:J_Gm}
	We have \(\JK(G\times \Gm)= \JK(G)\).
\end{proposition}
\begin{proof}
	Let \(X\) be a smooth projective \(k\)-variety with a \(G\times \Gm\)-action such that \(X^{G \times \Gm}=\varnothing\).
	Then the projective \(k\)-variety \(Y=X^G\) carries a \(\Gm\)-action such that \(Y^{\Gm}=\varnothing\).
	Borel's fixed-point theorem (see e.g.\ \cite[Theorem~(16.51)]{Milne-AG}) implies that \(Y=\varnothing\).
	This proves that \(\JK(G \times \Gm) \subset \JK(G)\).
	The other inclusion follows from (\ref{p:J_quotient}).
\end{proof}

\begin{proposition}
	\label{prop:J_mu_pq}
	Let \(p,q\) be two distinct prime numbers, both different from the characteristic of \(k\). Then \(\JK(\mu_{pq})=\Laz\).
\end{proposition}
\begin{proof}
	For each \(n\in \{p,q\}\) the quotient morphism \(\mu_{pq} \to \mu_n\) induces an action of the algebraic group \(\mu_{p,q}\) on the projective \(k\)-variety \(\mu_n\), which is smooth since \(n\) is prime to the characteristic of \(k\).
	Thus \(n=[\mu_n] \in \JK(\mu_{pq})\).
	Therefore \(1 \in \JK(\mu_{pq})\), and the statement follows from (\ref{p:J_ideal}).
\end{proof}

The case of diagonalizable \(p\)-groups is Theorem~\ref{th:no_fixed} of the introduction, namely:
\begin{theorem}
	\label{th:J_Ipn}
	Let \(p\) be a prime different from the characteristic of \(k\).
	Let \(G\) be a finite diagonalizable \(p\)-group, and let \(r\) be its rank (see (\ref{p:p_rank})).
	Then \(\JK(G)=\Ipn{r}\).
\end{theorem}
\begin{proof}
	This follows by combining (\ref{prop:no_fixed_on_gen}) and (\ref{th:empty_ipnr}).
\end{proof}

\begin{para}
	We have thus determined the ideal \(\JK(G)\) for all diagonalizable groups \(G\) of finite type over \(k\) having no characters of order divisible by the characteristic of \(k\).
	Namely:
	\[
		\JK(G) =
		\left\{
			\begin{array}{@{}l@{\quad}p{0.6\textwidth}@{}}
				\Ipn{r} & if \(G \simeq \mu_{p^{n_1}} \times \ldots \times \mu_{p^{n_r}} \times (\Gm)^s\) for some prime \(p\) and integers \(s\in \Nn\) and \(n_1,\ldots,n_r\in \Nn \smallsetminus \{0\}\),\\[3ex]
				\Laz & otherwise.
			\end{array}
		\right.
	\]
\end{para}

\subsection{Fixed locus dimension}
\begin{definition}
	\label{def:Delta}
	For \(d\in \Nn\), we denote by \(\Delta^d(G) \subset \Laz\) the subgroup generated by the classes \(\lc X \rc\), where \(X\) runs over the smooth projective \(k\)-varieties admitting a \(G\)-action such that \(\dim X^G=d\).
\end{definition}

\begin{lemma}
	\label{lemm:only_lower_bounds}
	The subgroup \(\Delta^d(G) \subset \Laz\) is generated by the classes \(\lc X \rc\), where \(X\) runs over the smooth projective \(k\)-varieties admitting a \(G\)-action such that \(\dim X^G \le d\).
\end{lemma}
\begin{proof}
	Let \(C\subset \Pp^2\) be a smooth projective cubic curve (such exists over every field \(k\)).
	Then the line bundle \(\Tan_C\) is trivial, so that \(\lc C \rc=0 \in \Laz\) (see (\ref{def:lc_rc})).
	Thus, if \(X\) is a smooth projective \(k\)-variety with a \(G\)-action, for every \(d\ge \dim X^G\) the smooth projective \(k\)-variety \(X'=X \sqcup C^d=X \sqcup (C\times \cdots \times C)\) satisfies
	\[
		\dim X'^G=\dim C^d=d \quad \text{ and } \quad \lc X' \rc = \lc X \rc + \lc C \rc^d=\lc X \rc.\qedhere
	\]
\end{proof}

\begin{para}
	\label{p:J_in_Delta}
	Lemma~(\ref{lemm:only_lower_bounds}) asserts in particular that \(\JK(G) \subset \Delta^d(G)\) for every \(d \in \Nn\).
\end{para}

\begin{remark}
	Lemma~(\ref{lemm:only_lower_bounds}) means that the only restrictions on the fixed locus dimension that can be obtained from the subgroup generated by the cobordism class of the ambient variety are lower bounds.
\end{remark}

\begin{para}
	From the description given in (\ref{lemm:only_lower_bounds}), we see using (\ref{lemm:equidim}) that \(\Delta^d(G)\) is a graded subgroup of \(\Laz\).
\end{para}

\begin{para}
	If \(G \to H\) is a surjective morphism of algebraic groups over \(k\), then \(\Delta^d(H) \subset \Delta^d(G)\) for every \(d \in \Nn\).
\end{para}

\begin{proposition}
	\label{prop:Delta_Gm}
	For every \(d \in \Nn\), we have \(\Delta^d(\Gm) = \Laz\).
\end{proposition}
\begin{proof}
	By (\ref{lemm:only_lower_bounds}) it suffices to treat the case \(d=0\).
	By (\ref{prop:fixed_Fd}) we have \(F^0_q\Laz \subset \Delta^0(\Gm)\) for every \(q \in \Nn\smallsetminus \{0\}\).
	But (\ref{def:F_q_L}.\ref{def:F_q_L:1}) implies that \(\Laz = \bigcup_{q \in \Nn\smallsetminus \{0\}} F_q^0\Laz\).
\end{proof}

\begin{proposition}
	Let \(p,q\) be two distinct prime numbers, both different from the characteristic of \(k\). Then \(\Delta^d(\mu_{pq}) = \Laz\) for every \(d \in \Nn\).
\end{proposition}
\begin{proof}
	This follows from (\ref{prop:J_mu_pq}) and (\ref{p:J_in_Delta}).
\end{proof}

We can now prove Theorem~\ref{th:main} of the introduction:
\begin{theorem}
	\label{th:Delta^d}
	Let \(p\) be a prime different from the characteristic of \(k\).
	Let \(G\) be a finite diagonalizable \(p\)-group over \(k\), let \(r\) be its rank (see (\ref{p:p_rank})), and set \(q=|\widehat{G}|\).
	Then for every \(d\in \Nn\), as subgroups of \(\Laz\)
	\[
		\Delta^d(G) = F^d_q \Laz + \Ipn{r}.
	\]
\end{theorem}
\begin{proof}
	We have \(\Ipn{r} \subset \Delta^d(G)\) by (\ref{th:J_Ipn}) and (\ref{p:J_in_Delta}), and we have \(F^d_q \Laz \subset \Delta^d(G)\) by (\ref{prop:fixed_Fd}).
	The other inclusion follows from (\ref{th:fixed_Fd}).
\end{proof}

\begin{remark}
	Taking \(d=0\) in (\ref{th:Delta^d}), and in view of (\ref{p:F0_L_q-1}) we obtain Corollary~\ref{cor:isolated} of the introduction.
\end{remark}

\begin{para}
	We have thus determined the subgroup \(\Delta^d(G) \subset \Laz\) for all \(d \in \Nn\) and all diagonalizable groups \(G\) of finite type over \(k\) having no characters of order divisible by the characteristic of \(k\).
	Namely:
	\[
		\Delta^d(G) =
		\begin{cases}
			F^d_q \Laz + \Ipn{r} & \text{ if \(G\) is a finite diagonalizable \(p\)-group of rank \(r\) and \(q = |\widehat{G}|\),} \\
			\Laz & \text{ otherwise.}
		\end{cases}
	\]
\end{para}

\subsection{Chern numbers}
Let us now draw a few concrete consequences of (\ref{th:Delta^d}).
Let \(p\) be a prime number different from the characteristic of the field \(k\).
Let \(G\) be a finite diagonalizable \(p\)-group over \(k\) of rank \(r\) (see (\ref{p:p_rank})), and let \(q=|\widehat{G}|\).
We will use the notation \(\deg_q\) of (\ref{p:grading_pol}).

\begin{theorem}
	Let \(\ell_i \in \Laz^{-i}\) be polynomial generators of the ring \(\Laz\).
	Let \(X\) be a smooth projective \(k\)-variety with a \(G\)-action.
	Let \(A=\Zz\) if \(G=1\) and \(A=\Fp\) otherwise.
	Let \(P\in A[T_i, i \in \su{N}{r}]\) be the unique polynomial such that \(\lc X \rc = P(\ell_1,\ldots)\) in \(\su{\Laz}{r}\).
	Then
	\[
		\dim X^G \ge \deg_qP.
	\]
\end{theorem}
\begin{proof}
	Such \(P\) exists and is unique by (\ref{lemm:pol_gen_L}).
	Let \(d=\dim X^G\).
	Then \(\lc X \rc \in F^d_q\su{\Laz}{r}\) by (\ref{th:fixed_Fd}), hence by (\ref{prop:L_q_grading}) we have \(\deg_qP \le d\).
\end{proof}

\begin{para}
	For a partition \(\alpha=(\alpha_1, \ldots, \alpha_n)\), we consider the integer
	\[
		\pi_q(\alpha) = \big \lfloor \frac{\alpha_1}{q} \big \rfloor + \cdots + \big \lfloor \frac{\alpha_n}{q} \big \rfloor.
	\]
\end{para}

\begin{para}
	\label{p:pi_q_deg_q}
	Let us note that:
	\begin{enumerate}[label=(\roman*), ref=\roman*]
		\item
			\label{p:pi_q_deg_q:1}
			\(\pi_q(\alpha) = \deg_q(T_\alpha)\) for every partition \(\alpha\) (see (\ref{p:grading_pol})),
		\item
			\label{p:pi_q_deg_q:2}
			\(\pi_q(\alpha) \le \lfloor |\alpha|/q \rfloor\) for every partition \(\alpha\) (see (\ref{p:partitions})),
		\item
			\label{p:pi_q_deg_q:3}
			\(\pi_q(\alpha \cup \beta)=\pi_q(\alpha)+ \pi_q(\beta)\) for every partitions \(\alpha,\beta\).
	\end{enumerate}
\end{para}

We recall from (\ref{p:def_succ}) that \(\succeq\) refers to the refinement relation between partitions.
\begin{lemma}
	\label{lemm:pi_succ}
	Let \(\alpha,\beta\) be partitions such that \(\alpha \succeq \beta\). Then \(\pi_q(\alpha) \le \pi_q(\beta)\).
\end{lemma}
\begin{proof}
	If \(\alpha^1,\ldots,\alpha^m\) are partitions such that \(\beta=(|\alpha^1|,\ldots,|\alpha^m|)\) and \(\alpha = \alpha^1 \cup \cdots \cup \alpha^m\), we have by (\ref{p:pi_q_deg_q}.\ref{p:pi_q_deg_q:2}) and (\ref{p:pi_q_deg_q}.\ref{p:pi_q_deg_q:3})
	\[
		\pi_q(\beta)
		= \big \lfloor \frac{|\alpha^1|}{q} \big \rfloor + \cdots + \big \lfloor \frac{|\alpha^m|}{q} \big \rfloor
		\ge \pi_q(\alpha^1) + \cdots + \pi_q(\alpha^m)
		=\pi_q(\alpha).\qedhere
	\]
\end{proof}

Recall from (\ref{eq:c_alpha}) that to each partition \(\alpha\) corresponds a morphism \(c_\alpha\colon\Laz \to \Zz\).

\begin{lemma}
	\label{lemm:c_F_pi}
	Let \(\alpha\) be a partition.
	Then \(c_\alpha(F^d_q\Laz)=0\) for any \(d<\pi_q(\alpha)\).
\end{lemma}
\begin{proof}
	Let \(\ell_i \in \Laz^{-i}\) be a family of polynomial generators of \(\Laz\).
	Then by (\ref{prop:L_q_grading}) (and (\ref{p:pi_q_deg_q}.\ref{p:pi_q_deg_q:1})) the group \(F^d_q\Laz\) is generated by the elements \(\ell_\beta\), where \(\beta\) is a partition such that \(\pi_q(\beta) \le d\).
	If \(d<\pi_q(\alpha)\), such \(\beta\) cannot verify \(\alpha \succeq \beta\) by (\ref{lemm:pi_succ}), hence \(c_\alpha(\ell_\beta)=0\) by (\ref{lemm:succ}).
\end{proof}

\begin{proposition}
	\label{prop:dim_c}
	Let \(X\) be a smooth projective \(k\)-variety with a \(G\)-action.
	Let \(d\in \Nn\), and let \(c\) be a \(\Zz\)-linear combination of the morphisms \(c_\beta\colon \Laz \to \Zz\), where \(\beta=(\beta_1,\ldots,\beta_m)\) are partitions satisfying
	\[
		\big \lfloor \frac{\beta_1}{q} \big \rfloor + \cdots +\big \lfloor \frac{\beta_m}{q}\big\rfloor\ge d.
	\]
	If \(c(\lc X \rc)\not \in c(\Ipn{r})\), then \(\dim X^G \ge d\).
\end{proposition}
\begin{proof}
	Assume the contrary.
	Then \(\lc X \rc \in F^{d-1}_q\Laz + \Ipn{r}\) by (\ref{th:Delta^d}).
	Since \(c(F^{d-1}_q\Laz)=0\) by (\ref{lemm:c_F_pi}), it follows that \(c(\lc X \rc) \in c(\Ipn{r})\), a contradiction.
\end{proof}

The next remark shows that the bound of (\ref{prop:dim_c}) is sharp:
\begin{remark}
	Let \(X\) be a smooth projective \(k\)-variety with a \(G\)-action.
	Assume that \(\lc X \rc \not \in \Delta^{d-1}(G)\), with \(d\in \Nn\smallsetminus \{0\}\).
	Let us also assume that \(G\ne 1\), i.e.\ \(r\ge 1\).
	Let \(\ell_i \in \Laz^{-i}\) be a family of polynomial generators of the ring \(\Laz\), such that the ideal \(\Ipn{r}\) is generated by \(p\) and \(\ell_{p^i-1}\) for \(i=1,\ldots,r-1\) (see (\ref{p:gen_in_I})).
	Write \(\lc X\rc = P(\ell_1,\ldots)\) with \(P \in \Zz[T_i, i\in \Nn]\).
	It follows from (\ref{th:Delta^d}) and (\ref{prop:L_q_grading}) that the polynomial \(P\) contains a monomial of the form \(\lambda T_\alpha\) with \(\lambda \in \Zz\smallsetminus p \Zz\), where \(\alpha\) is a partition verifying \(\pi_q(\alpha) \ge d\) and containing no elements of the form \(p^i-1\) with \(i\in \{1,\ldots,r-1\}\).
	The morphism \(d_\alpha\) given by (\ref{lemm:d_alpha}) verifies \(c(\lc X \rc) = \lambda d_\alpha(\ell_\alpha)\).
	On the other hand the group \(\Ipn{r}\) is generated by \(p\Laz\) and the elements \(\ell_\beta\), where \(\beta\) runs over a set of partitions distinct from \(\alpha\).
	Consequently \(d_\alpha(\Ipn{r})= p c(\Laz) = pd_\alpha(\ell_\alpha)\).
	By (\ref{lemm:pi_succ}), the morphism \(c=d_\alpha\) thus satisfies the conditions of (\ref{prop:dim_c}).
\end{remark}

\begin{lemma}
	\label{lemm:d_alpha}
	Let \(\ell_i \in \Laz^{-i}\) be a family of polynomial generators of \(\Laz\), and let \(\alpha\) be a partition.
	Then there exists a map \(d_\alpha \colon \Laz \to \Zz\) such that
	\begin{enumerate}[label=(\roman*), ref=\roman*]
		\item \(d_\alpha(\ell_\alpha) \ne 0\),
		\item \(d_\alpha(\ell_\beta) = 0\) for all \(\beta \ne \alpha\),
		\item \(d_\alpha\) is a \(\Zz\)-linear combination of the morphisms \(c_\beta\) for \(\alpha \succeq \beta\).
	\end{enumerate}
\end{lemma}
\begin{proof}
	We proceed by induction on the length of \(\alpha\).
	If \(\alpha= \varnothing\), then we set \(d_\alpha= c_\varnothing \colon \Laz \to \Zz\), the projection to the homogeneous component of degree \(0\).
	Now suppose that \(\length(\alpha) >0\).
	By induction and (\ref{p:succ_length}), we may assume \(d_\beta\) constructed for \(\beta \ne \alpha\) such that \(\alpha \succeq \beta\).
	We may also assume that the integer \(d_\beta(\ell_\beta) = u \in \Zz\smallsetminus \{0\}\) is independent of \(\beta\) (upon multiplying each \(d_\beta\) by an appropriate integer).
	Set
	\[
		d_\alpha = u c_\alpha- \sum_{\substack{\alpha \succeq \beta \\ \alpha \ne \beta}} c_\alpha(\ell_\beta) d_\beta.
	\]
	Then \(d_\alpha(\ell_\alpha)=uc_\alpha(\ell_\alpha) \ne 0\), and \(d_\alpha(\ell_\beta)=0\) for all \(\beta \ne \alpha\) such that \(\alpha \succeq \beta\).
	On the other hand \(d_\alpha(\ell_\beta)= 0\) when \(\alpha \not \succeq \beta\) by (\ref{lemm:succ}) (and transitivity of the relation \(\succeq\)).
\end{proof}

We now present a few concrete ways of finding morphisms \(c\) satisfying the conditions of (\ref{prop:dim_c}).
The most straightforward is to use Chern numbers modulo \(p\):

\begin{corollary}[of (\ref{prop:dim_c})]
	\label{prop:dim_pi}
	Let \(X\) be a smooth projective \(k\)-variety with a \(G\)-action.
	If \(\alpha\) is a partition such that the Chern number \(c_\alpha(X)\) is not divisible by \(p\), then
	\[
		\dim X^G \ge \big \lfloor \frac{\alpha_1}{q} \big \rfloor + \cdots + \big \lfloor \frac{\alpha_n}{q} \big \rfloor.
	\]
\end{corollary}
\begin{proof}
	Since \(c_\alpha(\Ipn{r}) \subset p \Zz\), we may apply (\ref{prop:dim_c}) with \(c=c_\alpha\).
\end{proof}

In order to use Chern numbers modulo higher powers of \(p\), we will need to exclude certain partitions.
\begin{para}
	For \(n\in \Nn\), we denote by \(\Apar{n}\) the set of partitions that are not a refinement (see (\ref{p:def_succ})) of a partition containing an element of \(\{p-1,\ldots,p^{n-1}-1\}\) (in particular \(\Apar{0}=\Apar{1}\) is the set of all partitions).
	Equivalently, a partition \(\alpha=(\alpha_1,\ldots,\alpha_m)\) belongs to \(\Apar{n}\) if and only if for every subset \(I \subset \{1,\ldots,m\}\)
	\[
		\sum_{i\in I} \alpha_i \not \in \{p-1,\ldots,p^{n-1}-1\}.
	\]
\end{para}

\begin{remark}
	\label{rem:refine}
	A partition \(\alpha=(\alpha_1,\ldots ,\alpha_m)\) belongs to \(\Apar{n}\) under any of the following assumptions:
	\begin{enumerate}[label=(\roman*), ref=\roman*]
		\item
			\(n=1\),
		\item
			\(\alpha_m\ge p^{n-1}\),
		\item
			every \(\alpha_i\) is divisible by \(p\),
		\item
			\label{rem:refine:3}
			\(m=1\) and \(\alpha_1 \neq p^i-1\) for all \(i \in \{1,\ldots,n-1\}\).
	\end{enumerate}
\end{remark}

\begin{lemma}
	\label{lemm:c_Ipn}
	Let \(n\in \Nn\).
	If \(\alpha \in \Apar{n}\), then \(c_{\alpha}(\Ipn{n}) \subset p c_{\alpha}(\Laz)\).
\end{lemma}
\begin{proof}
	We may assume that \(n\ge 1\), so that \(p\Laz \subset \Ipn{n}\).
	Let \(\ell_i \in \Laz^{-i}\), for \(i\in \Nn\smallsetminus \{0\}\), be a family of polynomial generators of the ring \(\Laz\), such that the ideal \(\Ipn{n}\) is generated by \(p\) and \(\ell_{p^i-1}\) for \(i=1,\ldots,n-1\) (see (\ref{p:gen_in_I})).
	The group \(\Ipn{n}/(p\Laz)\) is then generated by the elements \(\ell_\beta\), where \(\beta\) contains an element of \(\{p-1,\ldots,p^{n-1}-1\}\).
	Since \(\alpha \in \Apar{n}\), such a partition \(\beta\) cannot verify \(\alpha \succeq \beta\), and thus \(c_\alpha(\ell_\beta)=0\) by (\ref{lemm:succ}).
\end{proof}

\begin{corollary}[of (\ref{prop:dim_c})]
	\label{cor:dim_pi}
	Let \(X\) be a smooth projective \(k\)-variety with a \(G\)-action.
	Let \(\alpha=(\alpha_1,\ldots,\alpha_n) \in \Apar{r}\) be such that \(c_\alpha(X)\not \in p c_\alpha(\Laz)\).
	Then
	\[
		\dim X^G \ge \big \lfloor \frac{\alpha_1}{q} \big \rfloor + \cdots +\big \lfloor \frac{\alpha_n}{q}\big\rfloor.
	\]
\end{corollary}
\begin{proof}
	By (\ref{lemm:c_Ipn}) the morphism \(c=c_\alpha\) satisfies the conditions of (\ref{prop:dim_c}).
\end{proof}

\begin{remark}
	Let \(\alpha=(\alpha_1,\ldots,\alpha_n)\) be a partition such that  \(\alpha_n \ge p^r-1\), and \(\alpha_i+1\) is a power of \(p\) for every \(i=1,\ldots,n\).
	Then \(\alpha \in \Apar{r}\) by (\ref{rem:refine}), and \(c_\alpha(\Laz)\subset p \Zz\) by \cite[(7.3.3)]{inv}.
	Thus, if \(X\) is a smooth projective \(k\)-variety with a \(G\)-action such that \(c_\alpha(X)\ne 0 \mod p^2\), then \(\dim X^G \ge \pi_q(\alpha)\) by (\ref{cor:dim_pi}).
\end{remark}

\begin{corollary}
	\label{cor:indec:}
	Let \(X\) be a smooth projective \(k\)-variety of pure dimension \(n\not \in \{p-1,\ldots,p^{r-1}-1\}\).
	Assume that \(\lc X \rc\) is indecomposable in the graded ring \(\Laz/p\).
	Then for any \(G\)-action on \(X\), we have
	\[
		\dim X^G \ge \big \lfloor \frac{n}{q} \big \rfloor.
	\]
\end{corollary}
\begin{proof}
	The indecomposability assumption implies that \(c_{(n)}(X) \not \in p c_{(n)}(\Laz)\), by (\ref{p:indec_cn}).
	As \((n) \not \in \Apar{r}\) by (\ref{rem:refine}.\ref{rem:refine:3}), the statement follows from (\ref{cor:dim_pi}).
\end{proof}

\begin{para}
	\label{p:inter_hyper}
	Let \(X\) be the intersection of \(c\) hypersurfaces of degrees \(d_1,\ldots,d_c\) in \(\Pp^{n+c}\).
	Assume that \(X\) is smooth of pure dimension \(n\), and that \(n\ge 1\).
	Then
	\[
		c_{(n)}(X) = d_1\cdots d_c (d_1^n + \cdots + d_c^n -n -c -1).
	\]
	By (\ref{p:indec_cn}), the class \(\lc X \rc\) is indecomposable in \(\Laz/p\) if this integer is prime to \(p\), or if lies in \(p \Zz \smallsetminus p^2 \Zz\) when \(n+1\) is a power of \(p\).
	Let us list a few explicit cases where this happens:
	\begin{enumerate}[label=(\roman*), ref=\roman*]
		\item \(n+c=-1 \mod p\), \(c\ne 0 \mod p\), \(d_1=\cdots =d_c\ne 0 \mod p\),
		\item \(n\ne -1 \mod p\), \(c=0 \mod p\), \(d_1=\cdots =d_c\ne 0 \mod p\),
		\item \(n=(p-1)s\) with \(s\ne 1 \mod p\), and \(d_1,\ldots,d_c\) prime to \(p\).
		\item \(n+1\) a power of \(p\), \(d_1 \in p \Zz\smallsetminus p^2 \Zz\), and \(d_2,\ldots,d_c\) prime to \(p\).
	\end{enumerate}
\end{para}

\begin{example}
	Let \(X\) be as in any of the cases listed in (\ref{ex:c_n_hypersurface}) or (\ref{p:inter_hyper}).
	If the dimension \(n\) of \(X\) differs from \(p-1,\ldots,p^{r-1}-1\), and \(G\) acts on \(X\), then \(\dim X^G \ge \lfloor n/q \rfloor\) by (\ref{cor:indec:}).
\end{example}

\begin{example}
	Let \(X\) be a smooth hypersurface of degree \(d\) in \(\Pp^{p^{s+1}+1}\), with \(p^s\ge q\) and \(d\) prime to \(p\).
	Then one may verify that the Chern number \(c_{(p^s,\ldots,p^s)}(X)\) is prime to \(p\).
	Thus if \(G\) acts on \(X\), then (\ref{prop:dim_pi}) implies that \(\dim X^G \ge p^{s+1}/q\).
	(Note that when \(p\ne 2\) and \(d\) is congruent to \(2\) modulo \(p\), the class \(\lc X\rc\) is decomposable in \(\Laz/p\), by (\ref{lemm:c_n_hyp}).)
\end{example}

\begin{example}
	Let \(c,d,n \in \Nn\) be such that \(n=cd\ge 1\) and that \(n+c+1\) is a power of \(p\).
	Let \(X\) be the intersection of \(c\) hypersurfaces of degrees \(d_1,\ldots,d_c\) in \(\Pp^{n+c}\), with \(d_1,\ldots,d_c\) prime to \(p\).
	Assume that \(X\) is smooth of pure dimension \(n\).
	Then one may verify that the Chern number \(c_{(d,\ldots,d)}(X)\) is prime to \(p\).
	Thus if \(G\) acts on \(X\), then (\ref{prop:dim_pi}) implies that \(\dim X^G \ge c\lfloor d/q \rfloor\).
\end{example}

\appendix
\section{Effectivity}
\label{app:eff}
\numberwithin{theorem}{section}
\numberwithin{lemma}{section}
\numberwithin{proposition}{section}
\numberwithin{corollary}{section}
\numberwithin{example}{section}
\numberwithin{definition}{section}
\numberwithin{remark}{section}

In this section \(G\) is an algebraic group over \(k\).

\begin{para}
	Let us set \(\Laz^{<0}= \bigoplus_{i>0} \Laz^{-i} \subset \Laz\).
\end{para}

\begin{lemma}[{\cite[\S5]{Thom-Bourbaki}}]
	\label{lemm:opposite_L}
	Every element of \(\Laz^{<0}\) is of the form \(\lc X \rc\) where \(X\) is a smooth projective \(k\)-variety.
\end{lemma}
\begin{proof}
	Let \(V\subset \Laz\) be the set of classes \(\lc X \rc\) with \(X\) a smooth projective \(k\)-variety.
	Since the cobordism class of the disjoint union, resp.\ product, of smooth projective \(k\)-varieties is the sum, resp.\ product, of their cobordism classes, it follows that \(V\) is stable under addition and multiplication.

	We prove by induction on \(n\) that \(\Laz^{-n} \subset V\) when \(n>0\).
	Let \(H\subset \Pp^{n+1}\) be a smooth hypersurface of degree \(d\ge 4\) (such exist over any field \(k\)).
	Then \(c_{(n)}(H)=d(d^n-n-2)>0\).
	On the other hand \(c_{(n)}(\Pp^n)=-n-1<0\).
	So \(c_{(n)}(V)\subset \Zz\) is a subset stable under addition, which generates the subgroup \(c_{(n)}(\Laz^{-n})\) by (\ref{p:L_Zb}), and contains both a positive and a negative integer.
	This easily implies that \(c_{(n)}(V)=c_{(n)}(\Laz^{-n})\).
	Since the kernel of \(c_{(n)}\colon \Laz^{-n} \to \Zz\) consists of decomposable elements by (\ref{p:pol_gen_laz}), and since \(V\) is stable under addition, it will suffice to prove that \(V\) contains all decomposable elements of degree \(-n\).
	This is certainly true when \(n=1\), as zero is the only such element.
	If \(n>1\), this follows by induction on \(n\), as \(V\) is stable under addition and multiplication.
\end{proof}

\begin{definition}
	\label{def:JKe}
	We denote by \(\JKe(G) \subset \Laz\) the subset of those classes \(\lc X \rc\), where \(X\) is a smooth projective \(k\)-variety admitting a \(G\)-action such that \(X^G=\varnothing\).
\end{definition}

\begin{para}
	\label{p:J_eff_sum}
	The subset \(\JKe(G)\) is stable under addition.
\end{para}

\begin{para}
	\label{p:J_eff_component}
	It follows from (\ref{lemm:equidim}) that, for every \(i \in \Zz\), the homogeneous component of degree \(i\) of an element of \(\JKe(G)\subset \Laz\) is again in \(\JKe(G)\).
\end{para}

\begin{para}
	\label{p:J_eff_ideal}
	The argument of (\ref{p:J_ideal}), together with (\ref{lemm:opposite_L}), shows that the subset \(\JKe(G)\) is stable under multiplication by elements of \(\Laz^{<0}\).
\end{para}

\begin{proposition}
	\label{prop:J_eff_J}
	Let \(G\) be a finite diagonalizable group over \(k\), and assume that \(|\widehat{G}|\) is not a power of the characteristic of \(k\). Then \(\JKe(G) \cap \Laz^{<0}= \JK(G)\cap \Laz^{<0}\).
\end{proposition}
\begin{proof}
	By the assumption, the group \(G\) admits a quotient of the form \(\mu_n\), with \(n>1\) an integer prime to the characteristic of \(k\).
	Thus \(G\) acts without fixed points on the smooth projective \(k\)-variety \(\mu_n\), and so \([\mu_n]=n \in \JKe(G)\).

	Let \(X\) be a smooth projective \(k\)-variety with a \(G\)-action such that \(X^G =\varnothing\).
	If \(\lc X \rc \in \Laz^{<0}\), then \(-\lc X \rc \in \Laz^{<0}\), hence it follows from (\ref{p:J_eff_ideal}) that \(-n\lc X \rc \in \JKe(G)\).
	Since we also have \(\lc X \rc \in \JKe(G)\), we deduce using (\ref{p:J_eff_sum}) that \(\JKe(G)\) contains the subgroup generated by \(\lc X \rc\), which proves the proposition.
\end{proof}

\begin{para}
	\label{p:J_eff_Gm}
	The proof of (\ref{prop:J_Gm}) shows that \(\JKe(G \times \Gm)=\JKe(G)\).
\end{para}

\begin{remark}
	When \(G\) is a finite diagonalizable group, and \(p_1,\ldots,p_n\) are the prime divisors of the integer \(|\widehat{G}|\) that differ from the characteristic of \(k\), we have
	\[
		\JKe(G) \cap \Laz^0 = \Big\{ \sum_{i=1}^n a_i p_i | a_1,\ldots,a_n \in \Nn\Big\}.
	\]

	Together with (\ref{p:J_eff_Gm}), (\ref{prop:J_eff_J}), (\ref{th:J_Ipn}), and in view of (\ref{p:J_eff_component}), this completely determines \(\JKe(G)\) when \(G\) is a diagonalizable group of finite type over \(k\) having no character of order divisible by the characteristic of \(k\).
\end{remark}

\begin{definition}
	\label{def:Delta_eff}
	Let \(d\in \Nn\). We denote by \(\Delta_{\eff}^d(G) \subset \Laz\) the subset of those classes \(\lc X \rc\), where \(X\) is a smooth projective \(k\)-variety admitting a \(G\)-action such that \(\dim X^G=d\).
\end{definition}

\begin{para}
	For any \(d \in \Nn\), we have \(\Delta_{\eff}^d(G) \cap \Laz^0 = \Nn \subset \Zz = \Laz^0\).
\end{para}

\begin{para}
	\label{p:Delta_eff_stable}
	The set \(\Delta^d_{\eff}(G)\) is stable under addition, and \(\Delta^d_{\eff}(G) \cdot \Delta^e_{\eff}(G) \subset \Delta^{d+e}_{\eff}(G)\), for every \(e,f \in \Nn\).
\end{para}

\begin{para}
	\label{p:Delta_eff_only_lower_bounds}
	The proof of (\ref{lemm:only_lower_bounds}) shows that \(\Delta_{\eff}^d(G)\) is the subset of those classes \(\lc X \rc\), where \(X\) is a smooth projective \(k\)-variety admitting a \(G\)-action such that \(\dim X^G \le d\).
\end{para}

\begin{para}
	\label{p:J_eff_in_Delta}
	By (\ref{p:Delta_eff_only_lower_bounds}), we have in particular \(\JKe(G) \subset \Delta^d_{\eff}(G)\) for every \(d \in \Nn\).
\end{para}

\begin{para}
	It follows from (\ref{p:Delta_eff_only_lower_bounds}) and (\ref{lemm:equidim}) that, for every \(i \in \Zz\), the homogeneous component of degree \(i\) of an element of \(\Delta^d_{\eff}(G)\subset \Laz\) lies again in \(\Delta^d_{\eff}(G)\).
\end{para}

\begin{lemma}
	\label{lemm:Delta_eff_pq}
	Let \(p,q\) be two distinct prime numbers, both different from the characteristic of \(k\). Then \(\Delta^d_{\eff}(\mu_{pq}) \cap \Laz^{<0}= \Laz^{<0}\) for every \(d \in \Nn\).
\end{lemma}
\begin{proof}
	This follows from (\ref{prop:J_mu_pq}), (\ref{prop:J_eff_J}) and (\ref{p:J_eff_in_Delta}).
\end{proof}

\begin{lemma}
	\label{lemm:c_i_neg}
	Assume that \(G\) is a finite diagonalizable group over \(k\), and let \(q=|\widehat{G}|\).
	For every \(i\in \Nn \smallsetminus \{0\}\) there exists a smooth projective \(k\)-variety \(Y_i\) of pure dimension \(i\), with a \(G\)-action such that \(c_{(i)}(Y_i) >0\) and \(\dim (Y_i)^G \le \lfloor i/q \rfloor\).
\end{lemma}
\begin{proof}
	Let \(a = \lfloor (i+1)/q \rfloor\).
	Pick a map \(\varphi \colon \{1,\ldots,i+2\} \to \widehat{G}\) such that
	\begin{equation}
		\label{eq:varphi_cond}
		\text{\(|\varphi^{-1}\{g\}| \leq a+1\) for every \(g \in \widehat{G}\)}.
	\end{equation}
	This is possible, because
	\[
		q(a+1) = q \lfloor (i+1)/q \rfloor +q \ge i+1-(q-1) +q=i+2.
	\]
	For \(j\in \{1,\ldots,i+2\}\), set \(L_j = \character{\varphi(j)}\).
	Consider the \(G\)-equivariant vector bundle over \(\Speck\)
	\[
		V = \bigoplus_{j=1}^{i+2} L_j,
	\]
	and the projective space \(\Pp(V)=\Pp^{i+1}\) with its induced \(G\)-action.
	The fixed locus \(\Pp(V)^G\) is the disjoint union of the projective spaces
	\[
		P_g = \Pp\Big(\bigoplus_{\varphi(j)=g} L_j\Big)
	\]
	for \(g \in \widehat{G}\) (we write \(\Pp(0)= \varnothing\)).
	Condition \eqref{eq:varphi_cond} implies that \(\dim \Pp(V)^G \le a\).

	Now let \(m =p^2q\) (we need \(m\) to be a power of \(p\) divisible by \(q\) and \(m\ge 4\)).
	For each \(j\in \{1,\ldots,i+2\}\), the projection \(V\to L_j = \character{\varphi(j)}\) induces a \(G\)-equivariant section \(x_j\) of \(\Oc(1)\otimes \character{\varphi(j)}\) over \(\Pp(V)\).
	Using the canonical isomorphisms \((\character{\varphi(j)})^{\otimes m}= \character{m\varphi(j)}=\character{0}=1\), we construct a \(G\)-equivariant section
	\[
		\sigma = \sum_{j=1}^{i+2} (x_j)^{\otimes m}
	\]
	of the \(G\)-equivariant line bundle \(\Oc(m)\) over \(\Pp(V)\).
	This section is nonzero, and its zero locus is a smooth \(G\)-invariant hypersurface \(Y_i \subset \Pp(V)\) of degree \(m\) (by our assumption on the characteristic of \(k\)).
	For every \(g \in \widehat{G}\), the section \(\sigma\) restricts over \(P_g\) to a nonzero section of \(\Oc(m)|_{P_g}\), and thus \(Y_i \cap P_g\) is a (degree \(m\)) hypersurface in \(P_g\).
	In particular \((Y_i)^G = Y_i \cap \Pp(V)^G\) is nowhere dense in \(\Pp(V)^G\), and thus
	\[
		\dim (Y_i)^G \leq \dim \Pp(V)^G -1 \le a -1 = \Big\lfloor \frac{i+1}{q} \Big\rfloor - 1 \le \Big\lfloor \frac{i}{q} \Big\rfloor.
	\]
	Finally, as \(m \geq 4\) and \(i \geq 1\), it follows from (\ref{ex:c_n_hypersurface}) that \(c_{(i)}(Y_i)>0 \).
\end{proof}

\begin{remark}
	The crucial cases in (\ref{lemm:c_i_neg}) are \(i = 1,2\).
	Indeed, let \(G\) be an arbitrary algebraic group, and \(Q \subset \widehat{G}\) a finite set of characters containing \(0\).
	Set \(q=|Q|\).
	For \(i\ge 3\), the Milnor hypersurface \(Y_i=H_{2,i-1}\) (see (\ref{p:Milnor})) has pure dimension \(i\), satisfies \(c_{(i)}(Y_i)>0\) by (\ref{p:c_H_m_n}), and carries a \(G\)-action such that \(\dim (Y_i)^G \le \lfloor i/q \rfloor\) by (\ref{prop:action_on_Hnm}).
	On the other hand, we will see in (\ref{rem:Gm_curves_0_dim}) that it is not possible to construct such \(Y_i\) for \(i=1\), when \(G=\Gm\) and \(q\ge 2\).
\end{remark}

\begin{proposition}
	\label{prop:Delta_eff_p_group}
	Let \(G\) be a finite diagonalizable group over \(k\), and assume that \(|\widehat{G}|\) is not a power of the characteristic of \(k\). Then for every \(d\in \Nn\)
	\[
		\Delta^d_{\eff}(G) \cap \Laz^{<0}= \Delta^d(G) \cap \Laz^{<0}.
	\]
\end{proposition}
\begin{proof}
	By (\ref{lemm:Delta_eff_pq}) we may assume that \(G\) is a finite diagonalizable \(p\)-group, where \(p\) is a prime number different from the characteristic of \(k\).
	Let \(r\) be the rank of \(G\), and set \(q = |\widehat{G}|\).
	It follows from (\ref{th:J_Ipn}), (\ref{prop:J_eff_J}) and (\ref{p:J_eff_in_Delta}) that
	\[
		\Ipn{r} \cap \Laz^{<0}= \JK(G) \cap \Laz^{<0} = \JKe(G) \cap \Laz^{<0} \subset \Delta^d_{\eff}(G).
	\]
	By (\ref{th:Delta^d}) and (\ref{p:Delta_eff_stable}), it will suffice to prove that
	\begin{equation}
		\label{eq:F_in_Delta_eff}
		(F^d_q\Laz)\cap \Laz^{<0}\subset \Delta^d_{\eff}(G).
	\end{equation}
	Let \(i \in \Nn\smallsetminus \{0\}\).
	We have seen in (\ref{prop:action_on_Hnm}) that \(\lc \Pp^i\rc=\lc H_{0,i+1}\rc \in \Delta^{\lfloor i/q \rfloor}_{\eff}(G)\).
	Recall from (\ref{p:c_H_m_n}) that \(c_{(i)}(\Pp^i)=-i-1 <0\).
	Therefore, by (\ref{lemm:c_i_neg}), the subset \(c_{(i)}(\Delta^{\lfloor i/q \rfloor}_{\eff}(G) ) \subset \Zz\) contains both a positive and negative integer.
	Being stable under addition by (\ref{p:Delta_eff_stable}), it is a subgroup of \(\Zz\).
	Since the subgroup generated by \(c_{(i)}(\Delta^{\lfloor i/q \rfloor}_{\eff}(G) )\) is \(c_{(i)}(\Delta^{\lfloor i/q \rfloor}(G))\), which coincides with \(c_{(i)}(\Laz^{-i})\) by (\ref{prop:actions}) and (\ref{p:pol_gen_laz}), we conclude that
	\begin{equation}
		\label{eq:c_i_D_i}
		c_{(i)}(\Delta^{\lfloor i/q \rfloor}_{\eff}(G) )=c_{(i)}(\Laz^{-i}).
	\end{equation}

	By (\ref{prop:L_q_grading}), every element of \((F^d_q\Laz)\cap \Laz^{-i}\) is a sum of elements of the products \(\Laz^{-i_1}\cdots \Laz^{-i_n}\), where \(\lfloor i_1/q \rfloor + \ldots+ \lfloor i_n/q \rfloor \le d\) and \(i_1 + \ldots+i_n = i\).
	In view of (\ref{p:Delta_eff_stable}), in order to prove \eqref{eq:F_in_Delta_eff} it will thus suffice to prove that \(\Laz^{-i} \subset \Delta^{\lfloor i/q \rfloor}_{\eff}(G)\) for all \(i\in \Nn\smallsetminus \{0\}\), which we do by induction on \(i\).
	Given an element \(x \in \Laz^{-i}\) with \(i>0\), we may find by \eqref{eq:c_i_D_i} an element \(y \in \Delta^{\lfloor i/q \rfloor}_{\eff}(G)\) such that \(c_{(i)}(y)=c_{(i)}(x)\).
	By (\ref{p:cn_dec}) this implies that \(x-y\in \Laz^{-i}\) is decomposable in \(\Laz\).
	If \(i>1\), by induction and (\ref{p:Delta_eff_stable}) we deduce that to \(x\in \Delta^{\lfloor i/q \rfloor}_{\eff}(G)\).
	The same conclusion holds if \(i=1\), as then \(x=y\).
	We have thus established (\ref{eq:F_in_Delta_eff}), which concludes the proof of the proposition.
\end{proof}

\begin{remark}
	\label{rem:Gm_curves_0_dim}
	The conclusion of (\ref{prop:Delta_eff_p_group}) does not hold for \(G=\Gm\).
	Indeed assume that \(X\) is a smooth projective \(k\)-variety with a \(\Gm\)-action.
	Then \(\chi(X)=\chi(X^{\Gm})\) by \cite{Iversen-tori}, where \(\chi\) is the Euler number (i.e.\ \(\chi(Y) = \deg c_d(\Tan_Y)\) when \(Y\) is a smooth projective \(k\)-variety of pure dimension \(d\)).
	Since the function \(\chi\) takes positive values on \(0\)-dimensional varieties, this implies that \(\dim X^{\Gm} >0 \) whenever \(\chi(X) <0\).
	Thus, for instance, the set \(\Delta^0_{\eff}(\Gm)\) does not contain the class \(-\lc \Pp^1 \rc\), which does belong to \(\Delta^0(\Gm)\) by (\ref{prop:Delta_Gm}).
\end{remark}

\section{Glossary of notation}
\label{app:glossary}
\begin{center}
	\renewcommand{\arraystretch}{1.1}
	\begin{xltabular}{\textwidth}{@{} l l X @{}}
		\hline
		\textbf{Symbol} & \textbf{Reference} & \textbf{Description} \\
		\hline
		\endfirsthead

		\hline
		\textbf{Symbol} & \textbf{Reference} & \textbf{Description} \\
		\hline
		\endhead
		\(\Laz\) & (\ref{p:Laz}) &Lazard ring\\
		\(\fgl\) & (\ref{p:fgl}) &Formal sum\\
		\(\fglr{n}\) & (\ref{p:fgl})&Formal multiplication by \(n\in \Zz\)\\
		\(R[\bb]\) & (\ref{p:L_in_Zb})&Graded polynomial \(R\)-algebra \(R[b_1,\ldots]\)\\
		\(\length(\alpha)\) & (\ref{p:partitions})&Length of the partition \(\alpha\)\\
		\(|\alpha|\) & (\ref{p:partitions})&Weight of the partition \(\alpha\)\\
		\(b_{\alpha}\) & (\ref{p:partitions})&Monomial in \(R[\bb]\) associated with the partition \(\alpha\)\\
		\(c_{\alpha}\) & (\ref{p:partitions})&Map \(\Laz \to \Zz\) given by the \(b_\alpha\)-coefficient\\
		\(\ell_i\) & (\ref{p:pol_gen_laz})&Chosen set of polynomial generators of \(\Laz\)\\
		\(\succeq\) & (\ref{p:def_succ})& Refinement relation between partitions\\
		\(c_{\alpha}(X)\) & (\ref{def:lc_rc})&Chern number of the variety \(X\) for the partition \(\alpha\)\\
		\(\lc X \rc\) & (\ref{def:lc_rc})&Cobordism class in \(\Laz\) of the variety \(X\)\\
		\(\Ipn{n}\) & (\ref{def:Ipn})&Landweber ideal in \(\Laz\)\\
		\(u_i\) & (\ref{def:u_n})&Coefficients of the power series \(\fglr{p}(t)\)\\
		\(v_n\) & (\ref{def:u_n})&\(=u_{p^n-1}\), generators of \(\Ipn{\infty}\)\\
		\(\su{M}{n}\) & (\ref{p:su_n})&Quotient of the \(\Laz\)-module \(M\) by the submodule \(\Ipn{n}M\)\\
		\(\su{\Laz}{n}(d)\) & (\ref{p:Laz_n_d})&Subring of \(\su{\Laz}{n}\) generated by elements of degree \(\ge -d\) \\
		\(\su{N}{n}\) & (\ref{p:N})&Set of positive integers not of the form \(p^m-1\) with \(m< n\)\\
		\(\su{N}{n}(d)\) & (\ref{p:N})&Set of integers \(\le d\) in \(\su{N}{n}\)\\
		\(X^G\) & (\ref{p:fixed_locus})& Fixed locus of the \(G\)-action on the variety \(X\) \\
		\(\widehat{G}\) & (\ref{p:L_g})& Group of characters of \(G\) \\
		\(\character{g}\) & (\ref{p:L_g})&\(G\)-equivariant line bundle over \(\Speck\) given by \(g \in \widehat{G}\)\\
		\(\wgt{E}{g}\) & (\ref{p:character})&Eigenspace for the character \(g\) of the \(G\)-equivariant vector bundle \(E\) over a base with trivial \(G\)-action\\
		\(H_{m,n}\) & (\ref{p:Milnor})&Milnor hypersurface in \(\Pp^m \times \Pp^n\), of dimension \(m+n-1\)\\
		\(X_i^{+},X_i^{-}\) & (\ref{prop:actions})& Explicit varieties with a \(G\)-action (\(i \in \Nn\smallsetminus \{0\}\))\\
		\(F^d_q\su{\Laz}{n}\) & (\ref{def:F_q_L})&\(q\)-filtration of \(\su{\Laz}{n}\)\\
		\(\deg_q P\) & (\ref{p:grading_pol}) &\(q\)-degree of the polynomial \(P\)\\
		\(\Omega\) & (\ref{p:Omega})&Algebraic cobordism (with the characteristic exponent inverted)\\
		\(\La\) & (\ref{p:Laz_Omega_k})&\(=\Laz[c^{-1}]=\Omega(\Speck)\) (where \(c\) is the characteristic exponent of \(k\))\\
		\(R[\ag]\) & (\ref{p:[ag]})&Graded polynomial \(R\)-algebra in \(a_{i,g}\) for \(i\in \Nn, g\in \widehat{G}\smallsetminus \{0\}\)\\
		\(\Mcc\) & (\ref{def:Mcc})&\(=\La[\ag]\), cobordism ring of \((\widehat{G}\smallsetminus \{0\})\)-graded vector bundles over a smooth projective base\\
		\(\Mcc(d)\) & (\ref{p:grading_Mcc})& \(\Zz[c^{-1}]\)-subalgebra of \(\Mcc\) generated by classes of vector bundles over bases of dimension \(\le d\)\\
		\(\Mcc^{\le d}\) & (\ref{p:Mcc_ge_q})& \(\Zz[c^{-1}]\)-submodule of \(\Mcc\) generated by classes of vector bundles over bases of dimension \(\le d\)\\
		\(\lc E\to X \rc\) & (\ref{def:Mcc})&Class in \(\Mcc\) of the \(G\)-equivariant vector bundle \(E \to X\), where \(G\) acts trivially on \(X\) and \(E^G=0\), with \(X\) smooth projective\\
		\(p_{i,g}\) & (\ref{p:p_i_g})&\(\lc \Oc(1) \to \Pp^i\rc \in \Mcc\), where \(G\) acts on \(\Oc(1)\) via \(g \in \widehat{G}\smallsetminus \{0\}\)\\
		\(\Omega_G\) & (\ref{def:Omega_G})&Equivariant algebraic cobordism theory (with \(c\) inverted)\\
		\(S\) & (\ref{p:S})&Multiplicative subset of \(\Omega_G(\Speck)\) generated by Euler classes of the nontrivial characters of \(G\) \\
		\(C\) & (\ref{p:C}) & Image of \(\Omega_G(\Speck) \to S^{-1}\Omega_G(\Speck)\) \\
		\(H_j\) & (\ref{p:H_j}) & \(=\La[[t_1,\ldots,t_j]]/(\fglr{q_1}(t_1),\ldots,\fglr{q_j}(t_j))\) (with \(H_r \simeq \Omega_G(\Speck)\))\\
		\(S_j\) & (\ref{p:H_j})& A multiplicative subset of \(H_j\) (with \(S_r \simeq S\))\\
		\(C_j\) & (\ref{p:C_j_n}) & Image of \(\su{(H_j)}{r-j} \to S_j^{-1}\su{(H_j)}{r-j}\) (with \(C_r \simeq C)\) \\
		\(\lc X \rc_G\) & (\ref{p:lc_equiv})&Equivariant cobordism class in \(\Omega_G(\Speck)\) of the \(G\)-variety \(X\)\\
		\(e(E)\) & (\ref{p:Euler})&Euler class in \(\Omega_G(X)\) of the \(G\)-vector bundle \(E\to X\)\\
		\(e(-E)\) & (\ref{p:Euler_inverse})&Inverse Euler class in \(S^{-1}\Omega_G(X)\) of \(E\to X\), when \(E^G=0\)\\
		\(n_i\) &(\ref{p:x_i}) & \(=\lc N_i^+ \to X_i^+\rc_G- \lc N_i^- \to X_i^-\rc_G\) in \(\Mcc\)\\
		\(x_i\) &(\ref{p:x_i}) & \(=\lc X_i^+\rc_G- \lc X_i^-\rc_G\) in \(S^{-1}\Omega_G(\Speck)\)\\
		\(\JK(G)\) & (\ref{def:JK})&Subgroup of \(\Laz\) generated by the classes of varieties admitting a fixed-point-free \(G\)-action\\
		\(\Delta^d(G)\) & (\ref{def:Delta})&Subgroup of \(\Laz\) generated by the classes of varieties admitting a \(G\)-action with a fixed locus of dimension \(d\) \\
		\(\JKe(G)\) & (\ref{def:JKe})&Effective version of \(\JK(G)\), a subset of \(\Laz\)\\
		\(\Delta^d_{\eff}(G)\) & (\ref{def:Delta_eff})&Effective version of \(\Delta^d(G)\), a subset of \(\Laz\)\\
	\end{xltabular}
\end{center}

\def\cprime{$'$}

\end{document}